\newif \ifSIAM
\SIAMfalse

\ifSIAM\documentclass[final]{siamltex} \else
\documentclass[10pt,a4paper]{article}
\fi

\usepackage{graphicx}
\usepackage{amssymb}
\usepackage{amsmath}
\usepackage{booktabs}
\usepackage{enumerate}
\usepackage{paralist}
\usepackage{bm,bbm}
\usepackage{url}
\usepackage[usenames]{xcolor}
\usepackage{mathtools}
\usepackage[normalem]{ulem} 
\usepackage{cancel}
\usepackage[textsize=footnotesize,color=blue!30!white]{todonotes}

\ifSIAM \else
\usepackage{amsthm}
\fi

\numberwithin{equation}{section}

\newif \ifbbm
\bbmtrue

\newif \iffract

\newif \ifPDF
\PDFfalse

\newif \ifwhere

\newcommand{\be}{\begin{equation}}
\newcommand{\ee}{\end{equation}}
\newcommand{\bea}{\begin{eqnarray}}
\newcommand{\eea}{\end{eqnarray}}
\newcommand{\bean}{\begin{eqnarray*}}
\newcommand{\eean}{\end{eqnarray*}}
\def\ba#1\ea{\begin{align}#1\end{align}}
\def\ban#1\ean{\begin{align*}#1\end{align*}}
\def\bat#1\eat{\begin{alignat}#1\end{alignat}}
\def\batn#1\eatn{\begin{alignat*}#1\end{alignat*}}
\def\bs#1\es{\begin{split}#1\end{split}}
\newcommand{\bse}{\begin{subequations}}
\newcommand{\ese}{\end{subequations}}
\newcommand{\bt}{\begin{theorem}}
\newcommand{\et}{\end{theorem}}
\newcommand{\bl}{\begin{lemma}}
\newcommand{\el}{\end{lemma}}
\newcommand{\bc}{\begin{corollary}}
\newcommand{\ec}{\end{corollary}}
\newcommand{\bcr}{\begin{rcorollary}}
\newcommand{\ecr}{\end{rcorollary}}
\newcommand{\bp}{\begin{proof}}
\newcommand{\ep}{\end{proof}}
\newcommand{\bd}{\begin{definition}}
\newcommand{\ed}{\end{definition}}
\newcommand{\br}{\begin{remark}}
\newcommand{\er}{\end{remark}}
\newcommand{\bas}{\begin{assumption}}
\newcommand{\eas}{\end{assumption}}
\newcommand{\bex}{\begin{example}}
\newcommand{\eex}{\end{example}}
\newcommand{\bqo}{\begin{quote}}
\newcommand{\eqo}{\end{quote}}
\newcommand{\bdc}{\begin{description}}
\newcommand{\edc}{\end{description}}
\newcommand{\bi}{\begin{itemize}}
\newcommand{\ei}{\end{itemize}}
\newcommand{\ben}{\begin{enumerate}}
\newcommand{\een}{\end{enumerate}}

\newcommand\Om{\Omega}

\newcommand\oma{{\omega_\ver}}
\newcommand\omab{{\omega_\ver^{\mathrm{ext}}}}
\newcommand\bomab{{\overline\omega_\ver^{\mathrm{ext}}}}
\newcommand\omaD{{\omega_\ver^{\mathrm D}}}
\newcommand\omaN{{\omega_\ver^{\mathrm N}}}
\newcommand\omabis{{\omega_\verbis}}
\newcommand\toma{{\widetilde\omega_\ver}}
\newcommand\homa{{\widehat\omega_\ver}}
\newcommand\boma{{\overline\omega_\ver}}

\newcommand\GD{\Gamma_{\mathrm D}}
\newcommand\GN{\Gamma_{\mathrm N}}

\newcommand\Gr{\nabla}
\newcommand\Grb{\nabla_{\!\mathcal{T}}}
\newcommand\Grbt{\nabla_{\!\widetilde{\mathcal{T}}}}
\newcommand\Grbh{\nabla_{\!\widehat{\mathcal{T}}}}

\newcommand\Grd{\nabla_{\mathrm{d}}}
\newcommand\Dv{\nabla {\cdot}}
\newcommand\Dvb{\nabla_{\!\mathcal{T}} {\cdot}}

\newcommand\Crl{\nabla {\times}}

\newcommand\dv{\mathrm{div}}

\newcommand\crl{\mathrm{curl}}
\newcommand\scp{{\cdot}}

\newcommand\Lap{\Delta}
\newcommand\Lapb{\Delta_{\mathcal{T}}}
\newcommand{\jump}[1]{[\![#1]\!]}

\DeclarePairedDelimiter\norm{\|}{\|}

\ifbbm

\else

\fi

\newcommand\Ho{H^1(\Om)}
\newcommand\Hoi[1]{H^1(#1)}

\newcommand\Hooi[1]{H^1_0(#1)}

\newcommand\Hsa{H^1_*(\oma)}
\newcommand\HTa{H^1(\Ta)}

\newcommand\Lti[1]{L^2(#1)}

\newcommand\tLti[1]{\tL^2(#1)}

\newcommand\Hdv{\tH(\dv,\Om)}
\newcommand\Hdvi[1]{\tH(\dv,#1)}
\newcommand\Hdvs{\tH_0(\dv,\oma)}

\newcommand\HdvTa{\tH(\dv,\Ta)}

\newcommand\ie{i.e.}

\newcommand\cf{cf.}
\newcommand\eal{{\em et al.}}
\newcommand\eq{:=}

\newcommand\ds{\displaystyle}
\newcommand\nn{\nonumber}
\newcommand\pt{\partial}

\newcommand{\elm}{{K}}
\newcommand{\telm}{{\widetilde K}}
\newcommand{\helm}{{\widehat K}}

\newcommand{\sd}{{F}}
\newcommand{\tsd}{{\widetilde F}}
\newcommand{\hsd}{{\widehat F}}
\newcommand{\bsd}{{\overline F}}

\newcommand{\ver}{{\ta}}
\newcommand{\verbis}{{\mathbf a}}

\newcommand\Fh{\mathcal{F}_h}

\newcommand\Fhint{\Fh^{\mathrm{int}}}

\newcommand\FN{\Fh^{\mathrm{N}}}
\newcommand\FD{\Fh^{\mathrm{D}}}
\newcommand\FK{\mathcal{F}_\elm}
\newcommand\hFK{\widehat{\mathcal{F}}_{\helm}}
\newcommand\tFK{\widetilde{\mathcal{F}}_{\telm}}

\newcommand\Fa{\mathcal{F}_{\ver}}
\newcommand\tFa{\widetilde{\mathcal{F}}_{\ver}}
\newcommand\hFa{\widehat{\mathcal{F}}_{\ver}}
\newcommand\bFa{\overline{\mathcal{F}}_{\ver}}
\newcommand\tFaN{\tFa^{\mathrm{N}}}
\newcommand\Fain{{\Fa^{\mathrm{int}}}}
\newcommand\FKin{{\FK^{\mathrm{int}}}}
\newcommand\FKiin{{\F_{\elm_i}^{\mathrm{int}}}}
\newcommand\hFain{{\hFa^{\mathrm{int}}}}
\newcommand\tFain{{\tFa^{\mathrm{int}}}}
\newcommand\bFain{{\bFa^{\mathrm{int}}}}
\newcommand\Fab{{\Fa^{\mathrm{ext}}}}
\newcommand\hFab{{\hFa^{\mathrm{ext}}}}
\newcommand\bFab{{\bFa^{\mathrm{ext}}}}
\newcommand\tFab{{\tFa^{\mathrm{ext}}}}

\newcommand\FaN{\Fa^{\mathrm{N}}}
\newcommand\FaD{\Fa^{\mathrm{D}}}
\newcommand\tFaD{\tFa^{\mathrm{D}}}

\newcommand\bFaD{\bFa^{\mathrm{D}}}
\newcommand\FKN{\FK^{\mathrm{N}}}

\newcommand\FKD{\FK^{\mathrm{D}}}

\newcommand\Fe{{\F_{\!\edg}}}
\newcommand\hFe{{\widehat{\F}_{\!\edg}}}
\newcommand\bFe{{\overline{\F}_{\!\edg}}}

\newcommand{\edg}{{e}}
\newcommand{\hedg}{{\widehat e}}
\newcommand{\bedg}{{\overline e}}
\newcommand\Ea{\mathcal{E}_{\ver}}
\newcommand\hEa{\widehat{\mathcal{E}}_{\ver}}
\newcommand\bEa{\overline{\mathcal{E}}_{\ver}}

\newcommand\Th{\mathcal{T}_h}

\newcommand\Ta{{\mathcal{T}_{\ver}}}
\newcommand\tTa{{\widetilde{\mathcal{T}}_{\ver}}}
\newcommand\hTa{{\widehat{\mathcal{T}}_{\ver}}}
\newcommand\bTa{{\overline{\mathcal{T}}_{\ver}}}

\newcommand\Te{{\mathcal{T}_{\edg}}}

\newcommand\Vh{\mathcal{V}_h}
\newcommand\Vhint{\mathcal{V}^{\mathrm{int}}_h}
\newcommand\Vhext{\mathcal{V}^{\mathrm{ext}}_h}
\newcommand\VK{\mathcal{V}_\elm}

\newcommand\uu{u}
\newcommand\uh{u_h}

\newcommand\uuD{\uu_{\mathrm D}}
\newcommand\uuN{\sigma_{\mathrm N}}

\newcommand\pr{s}
\newcommand\pra{s_p} 
\newcommand\prh{s_h}

\newcommand\fr{{\bm \sigma}}
\newcommand\fra{{\bm \sigma}_p} 
\newcommand\frh{{\bm \sigma}_h}

\ifbbm

\else

\fi

\newcommand\ta{{\bm a}}

\newcommand\tn{{\bm n}}

\newcommand\tr{{\bm r}}

\newcommand\tv{{\bm v}}
\newcommand\tw{{\bm w}}
\newcommand\tx{{\bm x}}
\newcommand\ty{{\bm y}}
\newcommand\tz{{\bm z}}

\newcommand\tA{{\bm A}}

\newcommand\tH{{\bm H}}

\newcommand\tJ{{\bm J}}

\newcommand\tL{{\bm L}}

\newcommand\tS{{\bm S}}
\newcommand\tT{{\bm T}}

\newcommand\tV{{\bm V}}

\newcommand{\bxi}{{\bm\xi}}

\newcommand\btau{{\bm \tau}}

\newcommand\F{\mathcal{F}}

\newcommand\T{\mathcal{T}}

\newcommand\V{\mathcal{V}}

\newcommand\RR{{\mathbb R}}

\newcommand\PP{{\mathbb P}}

\newcommand\RTN{\bm{RTN}}

\iffract

\newcommand\ft{{\frac 1 2}}
\newcommand\mft{{-\frac 1 2}}

\else

\newcommand\ft{{1/2}}
\newcommand\mft{{-1/2}}

\fi

\newcommand\dx{\, \mathrm{d} \tx}
\newcommand\dy{\, \mathrm{d} \ty}

\newcommand{\upb}{^\mathrm{ext}}

\newcommand{\argmin}{\mathop{\mbox{\rm arg~min}}}

\newcommand{\wC}{\widehat C}
\newcommand{\wK}{\widehat\elm}
\newcommand{\wF}{\widehat\sd}
\renewcommand{\wr}{\widehat r}

\newcommand{\wphi}{\widehat\phi}
\newcommand{\wtphi}{\widetilde\phi}
\newcommand{\Jac}{{\bm J}}
\newcommand{\piola}{{\bm \psi}}
\newcommand{\bzeta}{{\bm \zeta}}
\renewcommand{\det}{\operatorname{det}}
\newcommand\FwKD{\F_{\wK}^{\mathrm{D}}}
\newcommand\FwKN{\F_{\wK}^{\mathrm{N}}}
\newcommand\FwK{\F_{\wK}}
\newcommand{\wbxi}{\widehat\bxi}
\newcommand{\calL}{\mathcal{L}}

\ifSIAM
\newtheorem{remark}[theorem]{Remark}
\newtheorem{assumption}[theorem]{Assumption}

\else

\newtheorem{lemma}{Lemma}[section]
\newtheorem{corollary}[lemma]{Corollary}
\newtheorem{assumption}[lemma]{Assumption}
\newtheorem{theorem}[lemma]{Theorem}
\newtheorem{definition}[lemma]{Definition}
\newtheorem{remark}[lemma]{Remark}
\newtheorem{example}[lemma]{Example}

\fi

\newif \ifPDF
\PDFtrue

\newif \ifPdfTex
\PdfTextrue

\ifPdfTex \usepackage[breaklinks,bookmarks=false]{hyperref}
\else
\ifPDF \usepackage[dvips,breaklinks,bookmarks=true]{hyperref} \fi \fi

\ifPDF
\hypersetup{colorlinks, linkcolor=blue, citecolor=blue,
urlcolor=blue, pagecolor=blue, plainpages=truee, pdfwindowui=false,
pdfstartview={FitH}, pdftitle={Stable broken H1 and H(div) polynomial extensions for polynomial-degree-robust potential and flux reconstruction in three space dimensions}, pdfauthor={Alexandre Ern and Martin Vohral\'\i k} }
\fi

\title{Stable broken $H^1$ and $\tH(\dv)$ polynomial extensions for polynomial-degree-robust potential and flux reconstruction in three space dimensions\thanks{This project has received funding from the European Research Council (ERC)
under the European Union’s Horizon 2020 research and innovation program (grant agreement No 647134 GATIPOR).}}

\author{Alexandre Ern\footnotemark[2] \footnotemark[3]
\and Martin Vohral\'ik\footnotemark[3] \footnotemark[2]}

\begin{document}
\maketitle

\renewcommand{\thefootnote}{\fnsymbol{footnote}}

\footnotetext[2]{Universit\'e Paris-Est, CERMICS (ENPC), 77455
Marne-la-Vall\'ee 2, France
(\ifPDF\href{mailto:alexandre.ern@enpc.fr}{\texttt{alexandre.ern@enpc.fr}}\else{\texttt{alexandre.ern@enpc.fr}}\fi).}
\footnotetext[3]{Inria, 2 rue Simone Iff, 75589 Paris, France
(\ifPDF\href{mailto:martin.vohralik@inria.fr}{\texttt{martin.vohralik@inria.fr}}\else{\texttt{martin.vohralik@inria.fr}}\fi).}

\renewcommand{\thefootnote}{\arabic{footnote}}

\begin{abstract}
\noindent We study extensions of piecewise polynomial data prescribed on faces and possibly in elements of a patch of simplices sharing a vertex. In the $H^1$ setting, we look for functions whose jumps across the faces are prescribed, whereas in the $\tH(\dv)$ setting, the normal component jumps and the piecewise divergence are prescribed. We show stability in the sense that the minimizers over piecewise polynomial spaces of the same degree as the data are subordinate in the broken energy norm to the minimizers over the whole broken $H^1$ and $\tH(\dv)$ spaces. Our proofs are constructive and yield constants independent of the polynomial degree. One particular application of these results is in a posteriori error analysis, where the present results justify polynomial-degree-robust efficiency of potential and flux reconstructions.
\end{abstract}

\bigskip

\ifSIAM \begin{keywords} \else \noindent{\bf Key words:} \fi
polynomial extension operator, broken Sobolev space, potential reconstruction, flux reconstruction, a posteriori error estimate, robustness, polynomial degree, best approximation, patch of elements

\ifSIAM \end{keywords}\fi

\ifSIAM

\pagestyle{myheadings} \thispagestyle{plain} \markboth{A. ERN AND M.
VOHRAL\'IK}{3D}

\else


\fi

\section{Introduction}

Braess~\eal\ \cite[Theorem~1]{Brae_Pill_Sch_p_rob_09} showed that equilibrated flux a posteriori error estimates lead to local efficiency and polynomial-degree robustness (in short, $p$-robustness). This means that the estimators upper-bounding the error also give local lower bounds for the error, up to a generic constant independent of the polynomial degree of the approximate solution. These results apply to conforming finite element methods in two space dimensions. They are based on flux reconstructions obtained by solving, via the mixed finite element method, a homogeneous Neumann, hat-function-weighted residual problem on each vertex-centered element patch of the mesh. The proof of the $p$-robustness in \cite{Brae_Pill_Sch_p_rob_09} relies on two key components: $p$-robust stability of the right inverse of the divergence operator shown in Costabel and McIntosh~\cite[Corollary~3.4]{Cost_McInt_Bog_Poinc_10} and $p$-robust stability of the right inverse of the normal trace shown in Demkowicz~\eal\ \cite[Theorem~7.1]{Demk_Gop_Sch_ext_III_12}. In our contribution~\cite[Theorem~3.17]{Ern_Voh_p_rob_15}, we extended $p$-robustness of a posteriori error estimates to any numerical scheme satisfying a couple of clearly identified assumptions, including nonconforming, discontinuous Galerkin, and mixed finite elements, still in two space dimensions, while proceeding through similar stability arguments. A second type of local problem appears here, where one is led to solve a homogeneous Dirichlet, conforming finite element problem on each vertex-centered element patch, with a hat-function-weighted discontinuous datum, yielding a potential reconstruction.

The present work extends the results of~\cite{Brae_Pill_Sch_p_rob_09} on flux reconstruction to three space dimensions and reformulates the methodology of~\cite{Ern_Voh_p_rob_15} for potential reconstruction so that it can be applied in the same way in two and three space dimensions. In doing so, we adopt a different viewpoint leading to a larger abstract setting not necessarily linked to a posteriori error analysis. The two main results of this paper are Theorems~\ref{thm_H1_ext} and~\ref{thm_Hdv_ext}. They concern a setting where one considers a shape-regular patch of simplicial mesh elements sharing a given vertex, say $\ver$, together with a $p$-degree polynomial $r_\sd$ associated with each interior face $\sd$ of the patch ($H^1$ potential reconstruction setting) or $p$-degree polynomials $r_\sd$ and $r_\elm$ associated with each interior and boundary face $\sd$ and element $\elm$ of the patch respectively ($\tH(\dv)$ flux reconstruction setting). These data, satisfying appropriate compatibility conditions, are to be extended to functions defined over the patch, such that the jumps across the interior faces of the patch are prescribed by $r_\sd$ ($H^1$ setting) or such that the normal component jumps and boundary values are prescribed by $r_\sd$ and the piecewise divergence is prescribed by $r_\elm$ ($\tH(\dv)$ setting). Crucially, we prove that the extension into piecewise polynomials of degree $p$ that minimizes the broken energy norm is, up to a constant only depending on the patch shape regularity, as good as the extension into the whole broken $H^1$ space with the same jump constraints. Similarly, our broken $p$-degree Raviart--Thomas--N\'ed\'elec extension is stable with respect to the broken $\tH(\dv)$ one.

Section~\ref{sec_eq_ref} reformulates equivalently the above theorems as Corollaries~\ref{cor_H1_ext} and~\ref{cor_Hdv_ext} to show that best-approximation of trace-discontinuous or normal-trace discontinuous piecewise polynomial data by $\Hooi{\oma}$- or $\Hdvs$-conforming piecewise polynomials (\ie, by trace-continuous or normal-trace continuous piecewise polynomials on the open set $\oma$ composed of the elements in the patch sharing the given vertex $\ta$) is, up to a $p$-independent constant, as good as by all $\Hooi{\oma}$ or $\Hdvs$ Sobolev functions. This section also sheds more light on the continuous level, uncovering that three different equivalent formulations of our results can be devised using the equivalence principle of primal and dual energies. This, in particular, allows us to make a link with the previously obtained results in~\cite{Brae_Pill_Sch_p_rob_09, Ern_Voh_p_rob_15} and to describe the application of our results to a posteriori error analysis in Section~\ref{sec_a_post_appl}. In particular, a guaranteed error upper bound for a generic numerical approximation of the Laplace equation is recalled in Corollary~\ref{cor_rel} and $p$-robust local efficiency is stated in Corollary~\ref{cor_eff}, with potential reconstructions treated in formula~\eqref{eq_eff_H1} and flux reconstructions in formula~\eqref{eq_eff_Hdv}.

The proofs of Theorems~\ref{thm_H1_ext} and~\ref{thm_Hdv_ext} are respectively presented in Sections~\ref{sec:proof_pot} and~\ref{sec:proof_flux}. In contrast to~\cite{Brae_Pill_Sch_p_rob_09}, where the work with dual norms was essential, we design here a procedure only working in the (broken) energy norms. The proofs are constructive and therefore indicate a possible practical reconstruction of the potential and the flux which avoid the patchwise problem solves by replacing them by a single explicit run through the patch, with possibly a solve of a local problem in each element.
The key ingredients on a single element are still the right inverse of the divergence~\cite[Corollary~3.4]{Cost_McInt_Bog_Poinc_10} and the right inverse of the normal trace~\cite[Theorem~7.1]{Demk_Gop_Sch_ext_III_12} in the $\tH(\dv)$ setting, but the single key ingredient becomes the right inverse of the trace shown in Demkowicz~\eal\ \cite[Theorem~6.1]{Demk_Gop_Sch_ext_I_09} in the $H^1$ setting. We combine these building blocks into a stability result on a single tetrahedron in Lemmas~\ref{lem_H1_simpl} and~\ref{lem_Hdv_simpl} in Appendix~\ref{app_ext_tetra}.
Gluing the elemental contributions together at the patch level turns out to be a rather involved ingredient of the proofs in three space dimensions, and we collect some auxiliary results for that purpose in Appendix~\ref{app_graphs}. A first difficulty is that the two-dimensional argument of turning around a vertex can no longer be invoked. To achieve a suitable enumeration of the mesh cells composing the patch in three dimensions, we rely on the notion of shelling of polytopes, see Ziegler~\cite[Chap.~8]{Zieg_poly_95}, which we reformulate for the present purposes in Lemma~\ref{lem:int_patch_enum}. Another difficulty is that we need to devise suitable functional transformations between different cells in the patch. This is done by introducing two- and three-coloring of some vertices lying on the boundary of the patch, possibly on a submesh of the original patch; how to achieve such colorings is described in Lemmas~\ref{lem:2color} and~\ref{lem:3color}.

For the sake of clarity of our exposition, we focus on discussing in details patches completely surrounding the vertex $\ver$, corresponding to an ``interior'' vertex when considering a mesh of some computational domain. Our technique, though, extends to the case where one considers a ``boundary'' vertex as well. Our main results in this context are Theorems~\ref{thm_H1_ext_gen} and~\ref{thm_Hdv_ext_gen}, whereas the reformulations as best-approximation results on piecewise polynomial data can be found in Corollaries~\ref{cor_H1_ext_gen} and~\ref{cor_Hdv_ext_gen}. The proofs of our main results concerning boundary vertices are given in Section~\ref{sec:boundary_ver}. The aforementioned application to a posteriori error analysis (Section~\ref{sec_a_post_appl}) then also covers some configurations of inhomogeneous Dirichlet and Neumann boundary conditions.
In the $H^1$ setting, we restrict ourselves for simplicity to the case where either Dirichlet or Neumann conditions are enforced; there is no such assumption in the $\tH(\dv)$ setting, which allows also for mixed Neumann--Dirichlet conditions.

Let us finally discuss some extensions of the present results.
In~\cite{Ern_Smears_Voh_H-1_lift_17}, we were recently able to employ them to construct $p$-robust $\tH(\dv)$ liftings over arbitrary domains, not just patches of elements sharing a given point. A natural extension of~\cite{Ern_Smears_Voh_H-1_lift_17} would be to obtain the same type of results in the $H^1$ setting. Another extension of the present work would be to cover the $\tH(\crl)$ case, hinging on the single tetrahedron results of Demkowicz~\eal\ \cite[Theorem~7.2]{Demk_Gop_Sch_ext_II_09}. We also mention that the application of the present results to the construction of $p$-robust a posteriori error estimates for problems with arbitrarily jumping coefficients is detailed in~\cite{Ciar_Voh_chan_coef_18}, to eigenvalue problems in~\cite{Canc_Dus_Mad_Stam_Voh_eigs_conf_17, Canc_Dus_Mad_Stam_Voh_eigs_nonconf_18}, to the Stokes problem in~\cite{Cer_Hech_Tang_Voh_ad_inex_S_18},
to linear elasticity in~\cite{Dors_Mel_p_rob_elast_13},
and to the heat equation in~\cite{Ern_Sme_Voh_heat_HO_Y_17, Ern_Sme_Voh_heat_HO_X_19}.

\section{Main results} \label{sec_res}

This section presents our main results, once the setting and basic notation have been fixed.

\subsection{Setting and basic notation} \label{sec_set}

We call tetrahedron any non-degenerate (closed) simplex in $\RR^3$, uniquely determined by four points in $\RR^3$ not lying in a plane. Let $\ver$ be a point in $\RR^3$. We consider a patch of tetrahedra around $\ver$, say $\Ta$, \ie, a finite collection of tetrahedra having $\ver$ as vertex, such that the intersection of any two distinct tetrahedra in $\Ta$ is either $\ver$, or a common edge, or a common face. A generic tetrahedron in $\Ta$ is denoted by $\elm$ and is also called an element or a cell. We let $\oma\subset\RR^3$ denote the interior of the subset $\cup_{\elm\in\Ta} \elm$. For the time being, we focus on the case where $\oma$ contains an open ball around $\ver$. The main application we have in mind is when $\ver$ is the interior vertex of a simplicial mesh $\Th$ of some computational domain $\Om$, so that $\ver$ lies in the interior of the patch $\oma$ surrounding it, see the left panel of Figure~\ref{fig:patch_int} for an illustration. The case where $\ver$ is a boundary vertex of the mesh entails some additional technicalities that we detail in Section~\ref{sec:boundary_ver_set}.

\begin{figure}[htb]
\centerline{\includegraphics[height=0.4\textwidth]{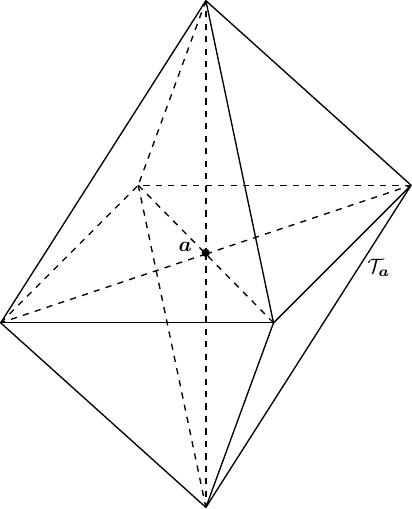} \qquad \qquad \includegraphics[height=0.4\textwidth]{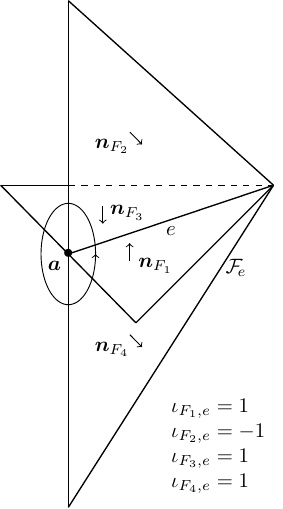}}
\caption{Example of an interior patch $\Ta$ (left); an edge $\edg \in \Ea$, the set $\Fe$ of all the faces that share it, the face normals, and the orientation indicators $\iota_{\sd,\edg}$ (right)}
\label{fig:patch_int}
\end{figure}

All the faces of the elements in the patch $\Ta$ are collected in the set $\Fa$ which is split into
\be
\Fa=\Fain \cup \Fab,
\ee
with $\Fain$ collecting all the interior faces (containing the vertex $\ver$ and shared by two distinct elements in $\Ta$) and $\Fab$ collecting the faces located in $\pt\oma$. For all faces $\sd\in\Fa$, $\tn_\sd$ denotes a unit normal vector to $\sd$ whose orientation is arbitrary but fixed for all $\sd\in\Fain$ and coinciding with the unit outward normal $\tn_{\oma}$ to $\oma$ for all $\sd\in\Fab$. We consider the jump operator $\jump{{\cdot}}_\sd$ for all $\sd\in\Fain$, yielding the difference (evaluated along $\tn_\sd$) of the traces of the argument from the two elements that share the interior face $\sd$ (the subscript $\sd$ is omitted if there is no ambiguity). We also need to consider edges. Let $\Ea$ collect all the edges in $\Ta$ sharing the vertex $\ver$; we refer to these edges as interior edges. Then, for each $\edg\in\Ea$, the set $\Fe$ collects all the faces in $\Fain$ sharing $\edg$, and the set $\Te$ collects all the cells in $\Ta$ sharing $\edg$. For each $\edg \in \Ea$, we fix one direction of rotation around $\edg$, and indicate for all $\sd \in \Fe$ by $\iota_{\sd,\edg}$ either equal to $1$ or to $-1$ whether $\tn_\sd$ complies with this direction or not, see the right panel of Figure~\ref{fig:patch_int} for an illustration.

We define the broken $H^1$-space on the patch $\Ta$ as
\be \label{eq_H_Ta}
    \HTa \eq \{v \in \Lti{\oma}; \, v|_\elm \in \Hoi{\elm},\,\, \forall \elm \in \Ta\},
\ee
and similarly the broken ${\bm H}(\dv)$-space on the patch $\Ta$ as
\be \label{eq_Hdv_Ta}
    \HdvTa \eq \{\tv \in \tLti{\oma}; \, \tv|_\elm \in \Hdvi{\elm},\,\, \forall \elm \in
    \Ta\}.
\ee
For any $v \in \HTa$, we can consider its piecewise (broken) gradient $\Grb
v$ defined as $(\Grb v)|_\elm = \Gr (v|_\elm)$, and similarly for any $\tv
\in \HdvTa$, we can consider its piecewise (broken) divergence $\Dvb \tv$
defined as $(\Dvb \tv)|_\elm = \Dv (\tv|_\elm)$, for all $\elm \in \Ta$. For
any $v \in \HTa$, the jumps $\jump{v}_\sd$ across any face $\sd\in\Fain$ are
well defined since the traces of $v$ on $\sd$ from the two cells sharing
$\sd$ are in $\Lti{\sd}$; similarly, the traces $v|_\Fab$ are well-defined.
We note that any function $v \in H^{1+\epsilon}(\Ta)$, $\epsilon>0$, is such that
\be \label{eq:sum_edge_jumps}
\sum_{\sd \in \Fe} \iota_{\sd, \edg} \, \jump{v}_\sd|_\edg = 0 \quad  \text{for all interior edges } \edg \in \Ea,
\ee
since the oriented sum of the jumps along a closed path around an interior edge is always zero. The definition of traces is a bit more subtle when one considers a field $\tv \in \HdvTa$. Let
$r_\sd \in \Lti{\sd}$ for all $\sd \in \Fa$. Then we say that
\bse \label{eq_norm_jump} \bat{2}
    \tv \scp \tn_\sd & = r_\sd & \quad & \forall \sd \in \Fab, \label{eq_norm_jump_1} \\
    \jump{\tv} \scp \tn_\sd & = r_\sd & \quad & \forall \sd \in \Fain \label{eq_norm_jump_2}
\eat
for a function $\tv \in \HdvTa$ if and only if
\be \label{eq_norm_jump_def}
    (\Dvb\tv,v)_\oma +  (\tv, \Gr v)_\oma = \sum_{\sd \in
    \Fa} (r_\sd, v)_\sd \qquad \forall v \in H^1(\oma).
\ee \ese
We will also need to prescribe the normal
component of vector fields in a single cell $\elm\in\Ta$ with unit outward normal $\tn_\elm$. Consider a non-empty
subset $\FKN\subset \FK$ where $\FK$ collects the faces of $\elm$. Given functions $r_\sd \in L^2(\sd)$
for all $\sd\in\FKN$, we
say that $\tv \scp \tn_\elm|_\sd = r_\sd$, $\forall \sd\in\FKN$, for a function $\tv \in \Hdvi{\elm}$ if
\be \label{eq:def_vn_K}
(\Dv \tv, \phi)_\elm + (\tv, \Gr \phi)_\elm = \sum_{\sd
\in \FKN}(r_\sd,\phi)_\sd \qquad \forall \phi \in \Hoi{\elm} \text{ s.t. } \phi|_\sd = 0 \; \forall \sd \in \FK
\setminus \FKN.
\ee

Let $p \ge 0$ denote an integer. We use the notation $\PP_{p}(\elm)$
for polynomials of order at most $p$ in the element $\elm \in \Ta$ and
$\PP_{p}(\sd)$ for polynomials of order at most $p$ in the face $\sd \in
\Fa$. We denote by $\PP_p(\Ta)$ the space composed of all functions defined
on the patch $\Ta$ whose restriction to any $\elm\in\Ta$ is in
$\PP_{p}(\elm)$. Similarly, $\PP_p(\Fa)$ stands for the space composed of all
functions defined on all faces from $\Fa$ whose restriction to any
$\sd\in\Fa$ is in $\PP_{p}(\sd)$. Analogous notation is used for any subset
of $\Fa$. We denote by $r_\elm$ the restriction of $r \in \PP_p(\Ta)$ to
$\elm \in \Ta$ and similarly by $r_\sd$ the restriction of $r \in \PP_p(\Fa)$
to $\sd \in \Fa$. Let $\RTN_p(\elm)$ be the Raviart--Thomas--N\'ed\'elec
polynomial space of vector-valued functions of order $p$ in the element $\elm \in
\Ta$, \ie, $\RTN_p(\elm) \eq [\PP_{p}(\elm)]^3 + \PP_{p}(\elm) \tx$.
Finally, $\RTN_p(\Ta)$ denotes the broken space composed of all functions whose
restriction to any element $\elm\in\Ta$ is in $\RTN_p(\elm)$.

For an element $\elm\in\Ta$, its shape-regularity parameter $\gamma_\elm$ is defined to be the ratio of its diameter to the diameter of the largest inscribed ball, and the shape-regularity parameter of the patch $\Ta$ is then defined to be $\gamma_{\Ta} \eq \max_{\elm\in\Ta} \gamma_\elm$.

\br[Orientation] \label{rem:orient}
The orientation of $\tn_\sd$ is irrelevant in~\eqref{eq:sum_edge_jumps}.
Indeed, changing the orientation of $\tn_\sd$ changes the sign of the jumps (evaluated along $\tn_\sd$) and at the same time the sign of $\iota_{\sd, \edg}$. Similarly, the orientation of $\tn_\sd$ is irrelevant in the left-hand side of~\eqref{eq_norm_jump_2}.
\er

\subsection{Broken $H^1$ polynomial extension}

Our main result for broken scalar extensions is the following.

\bt[Stable broken $H^1$ polynomial extension] \label{thm_H1_ext} Let $p \ge 1$. Let the interface-based $p$-degree polynomial $r \in \PP_{p}(\Fain)$ satisfy the following compatibility conditions:
\bse \label{eq_conds_H1} \bat{2}
    r_\sd|_{\sd\cap\pt\oma} & = 0 \qquad & & \text{on all interior faces } \sd \in \Fain, \label{eq_conds_H1_1} \\
    \sum_{\sd \in \Fe} \iota_{\sd, \edg} \, r_\sd|_\edg & = 0 & &  \text{on all interior edges } \edg \in \Ea. \label{eq_conds_H1_2}
\eat \ese
Then there exists a constant $C_{\rm st} > 0$ only depending on the patch shape-regularity parameter $\gamma_{\Ta}$ such that
\be \label{eq_H1_ext}
    \min_{\substack{v_p \in \PP_p(\Ta)\\
    v_p|_\sd = 0\,\,\forall \sd\in\Fab,\\
    \jump{v_p}_\sd = r_\sd \,\, \forall \sd \in \Fain}} \norm{\Grb v_p}_\oma
    \leq C_{\rm st} \min_{\substack{v \in \HTa\\
    v|_\sd = 0 \,\, \forall \sd \in \Fab,\\
    \jump{v}_\sd = r_\sd \,\, \forall \sd \in \Fain}} \norm{\Grb v}_\oma,
\ee
where the minimization sets are non-empty and both minimizers
in~\eqref{eq_H1_ext} are unique. \et

The compatibility conditions~\eqref{eq_conds_H1} are natural since $r_\sd$ is used to prescribe interface jumps. Indeed, these jumps necessarily vanish on the points of the interfaces located on $\pt\oma$ since the considered functions vanish on $\pt\oma$; moreover, \eqref{eq_conds_H1_2} follows from~\eqref{eq:sum_edge_jumps}. The minimizers in~\eqref{eq_H1_ext} are respectively denoted by $v_p^*$ and $v^*$, so that \eqref{eq_H1_ext} becomes
\be \norm{\Grb v_p^*}_\oma \leq C_{\rm st} \norm{\Grb v^*}_\oma. \ee
Note also that since the minimization sets are non-empty and the left one is a subset of the right one by definition, the inequality in the other direction, $\norm{\Grb v^*}_\oma\leq \norm{\Grb v_p^*}_\oma$, is trivial.

\subsection{Broken $\tH(\dv)$ polynomial extension}

Our main result for broken vector extensions is the following.

\bt[Stable broken $\tH(\dv)$ polynomial extension] \label{thm_Hdv_ext} Let $p \ge
0$. Let the element- and face-based $p$-degree polynomial $r \in \PP_p(\Ta) \times \PP_{p}(\Fa)$ satisfy the following
compatibility condition:
\be \label{eq_comp}
    \sum_{\elm \in \Ta} (r_\elm, 1)_\elm - \sum_{\sd \in
    \Fa} (r_\sd, 1)_\sd = 0.
\ee
Then there exists a constant $C_{\rm st} > 0$ only depending on the patch shape-regularity parameter $\gamma_{\Ta}$ such that
\be \label{eq_Hdv_ext}
    \min_{\substack{\tv_p \in \RTN_p(\Ta)\\
    \tv_p \scp \tn_\sd = r_\sd \,\, \forall \sd \in \Fab\\
    \jump{\tv_p}\scp \tn_\sd = r_\sd \,\, \forall \sd \in \Fain\\
    \Dvb \tv_p|_\elm = r_\elm \,\, \forall \elm \in \Ta}}
    \norm{\tv_p}_\oma \leq C_{\rm st}
    \min_{\substack{\tv \in \HdvTa\\
    \tv \scp \tn_\sd = r_\sd \,\, \forall \sd \in \Fab\\
    \jump{\tv}\scp \tn_\sd = r_\sd \,\, \forall \sd \in \Fain\\
    \Dvb \tv|_\elm = r_\elm \,\, \forall \elm \in \Ta}}
    \norm{\tv}_\oma,
\ee
where the minimization sets are non-empty and both minimizers
in~\eqref{eq_Hdv_ext} are unique. \et

The compatibility condition~\eqref{eq_comp} is again natural here, since it follows from~\eqref{eq_norm_jump_def} with the test function equal to $1$ in $\oma$. The minimizers in~\eqref{eq_Hdv_ext} are respectively denoted by $\tv_p^*$ and $\tv^*$, so that~\eqref{eq_Hdv_ext} becomes \be \norm{\tv_p^*}_\oma \le C_{\rm st} \norm{\tv^*}_\oma. \ee Since the minimization sets are non-empty and the left one is a subset of the right one by definition, the inequality in the other direction, $\norm{\tv^*}_\oma\le \norm{\tv_p^*}_\oma$, is again trivial.

\subsection{Boundary vertices} \label{sec:boundary_ver_set}

\begin{figure}[htb]
\centerline{\includegraphics[height=0.4\textwidth]{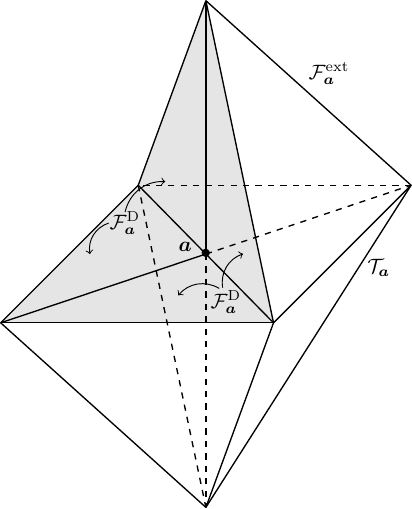} \qquad\qquad \includegraphics[height=0.4\textwidth]{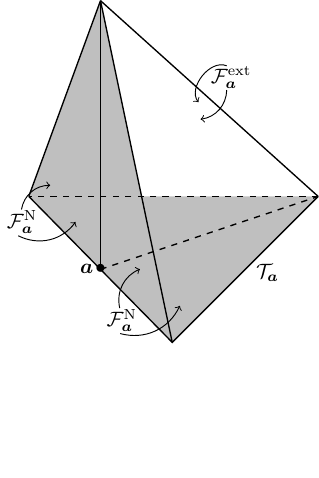}}
\caption{Two examples of boundary patches $\Ta$;
in both cases, the faces in $\Fab$, and thus the subset $\pt\omab$,
are shown in white.
Left: patch where all the faces in $\Fab$ have at least one vertex lying in
the interior of $\pt\omab$; right: patch where there are faces in $\Fab$ (actually both of them) that do not have any vertex lying in
the interior of $\pt\omab$ but where $|\Ta| \leq 2$}
\label{fig:patch_bnd}
\end{figure}

We consider in this section the case where the patch domain $\oma$
does not contain an open ball around the point $\ver$; typically, $\ver$ is a mesh vertex lying on the boundary of some computational domain $\Om$. In this case, the patch domain $\oma$ only contains an open ball around $\ver$ minus some sector with solid angle $\theta_\ver\in(0,4\pi)$, see Figure~\ref{fig:patch_bnd} for two examples.

The set $\Fa$ collecting all the faces of $\Ta$ is now divided into four disjoint subsets:
\be \label{eq_faces_bnd}
\Fa = \Fain \cup \Fab \cup \FaD \cup \FaN,
\ee
where the set $\Fain$ collects (as before) the faces interior to $\oma$, that is, the faces containing the vertex $\ver$ and shared by two distinct elements in $\Ta$, the set $\Fab$ collects the faces that are subsets of $\pt \oma$ that do not contain $\ver$, and $\FaD\cup \FaN$ collects the faces that are subsets of $\pt \oma$ that contain $\ver$. The distinction between $\FaD$ and $\FaN$ is needed because some prescription is to be enforced on $\FaD$ ($H^1$-extension) or on $\FaN$ ($\tH(\dv)$-extension). These additional prescriptions are motivated by the handling of Dirichlet or Neumann boundary conditions as further highlighted in Section~\ref{sec_a_post_appl}.
Correspondingly, we set $\pt\omab \eq \cup_{\sd\in\Fab} \sd$, $\pt\omaD \eq \cup_{\sd\in\FaD}\sd$, and $\pt\omaN \eq \cup_{\sd\in\FaN}\sd$, so that
\be
    \pt \oma = \pt\omab \cup \pt\omaD \cup \pt\omaN,
\ee
see Figure~\ref{fig:patch_bnd}.
Faces in the three sets $\Fab$, $\FaD$, and $\FaN$ are assigned a unit normal vector $\tn_\sd$ pointing outward $\oma$. We remark that $\Fain$ can be empty (if $\Ta$ consists of a single tetrahedron), that $\Fab$ is always non-empty, and that either $\FaD$ or $\FaN$ can be empty, but not both at the same time. Finally, the set $\Ea$ collects all the edges in $\Ta$ sharing the vertex $\ver$ (note that some of these edges are now located on $\pt\oma$) and, for each edge $\edg\in\Ea$, $\Fe$ collects all the faces in $\Fa$ sharing $\edg$ (note that $\Fe$ is now a subset of $\Fain\cup\FaD\cup\FaN$). As above, for each edge $\edg \in \Ea$ and each face $\sd \in \Fe$, $\iota_{\sd,\edg}$ is either equal to $1$ or to $-1$ and indicates whether $\tn_\sd$ complies with the fixed direction of rotation around $\edg$ or not, see the right panel of Figure~\ref{fig:patch_int}.

We now present our main results for boundary vertices. In the $H^1$ setting, they request that either (the Dirichlet part of the boundary) $\pt\omaD$ is empty, or (the Neumann part of the boundary) $\pt\omaN$ is empty.
No such assumption on $\pt\omaD$ and $\pt\omaN$
is needed in the $\tH(\dv)$ setting. Moreover, in both settings,
we assume that either all the faces in $\Fab$ have at least one vertex lying in the interior of $\pt\omab$, or that the number of elements in the patch $\Ta$ is at most two. For instance, the patch in the left
panel of Figure~\ref{fig:patch_bnd} satisfies the first assumption, and that in the right panel the second assumption. Other cases can be treated, see Remark~\ref{rem_bnd_other} below, but the analysis is increasingly technical.


\bt[Stable broken $H^1$ polynomial extension] \label{thm_H1_ext_gen} Let $p \ge 1$ and let either $\FaD = \emptyset$ or $\FaN = \emptyset$. Assume either that all the faces in $\Fab$ have at least one vertex lying in the interior of $\pt\omab$, or that $|\Ta| \leq 2$.
Let $r \in \PP_{p}(\Fain \cup \FaD)$ satisfy the following compatibility
conditions:
\bse \label{eq_conds_H1_gen} \bat{2}
    r_\sd|_{\sd\cap\pt\omab} & = 0 \qquad & & \forall \sd \in \Fain \cup \FaD, \label{eq_conds_H1_1_gen} \\
    \sum_{\sd \in \Fe} \iota_{\sd, \edg} \, r_\sd|_\edg & = 0 & &  \forall \edg \in \Ea \text{ such that }
    \Fe \cap \FaN = \emptyset. \label{eq_conds_H1_2_gen}
\eat \ese
Then there exists a constant $C_{\rm st} > 0$ only depending on the patch shape-regularity parameter $\gamma_{\Ta}$ such that
\be \label{eq_H1_ext_gen}
    \min_{\substack{v_p \in \PP_p(\Ta)\\
    v_p|_\sd = 0\,\,\forall \sd\in\Fab,\\
    v_p|_\sd = r_\sd\,\,\forall \sd\in\FaD,\\
    \jump{v_p}_\sd = r_\sd \,\, \forall \sd \in \Fain}} \norm{\Grb v_p}_\oma
    \leq C_{\rm st} \min_{\substack{v \in \HTa\\
    v|_\sd = 0 \,\, \forall \sd \in \Fab,\\
    v|_\sd = r_\sd\,\,\forall \sd\in\FaD,\\
    \jump{v}_\sd = r_\sd \,\, \forall \sd \in \Fain}} \norm{\Grb v}_\oma,
\ee
where the minimization sets are non-empty and both minimizers
in~\eqref{eq_H1_ext_gen} are unique. \et

\bt[Stable broken $\tH(\dv)$ polynomial extension] \label{thm_Hdv_ext_gen} Let $p \ge 0$.
Assume either that all the faces in $\Fab$ have at least one vertex lying in the interior of $\pt\omab$, or that $|\Ta| \leq 2$.
Let $r \in \PP_p(\Ta) \times \PP_{p}(\Fain \cup \Fab \cup \FaN)$
satisfy the following compatibility condition:
\be \label{eq_comp_gen}
    \sum_{\elm \in \Ta} (r_\elm, 1)_\elm - \sum_{\sd \in
    \Fa} (r_\sd, 1)_\sd = 0 \qquad
    \text{ if } \FaD = \emptyset.
\ee
Then there exists a constant $C_{\rm st} > 0$ only depending on the
patch shape-regularity parameter $\gamma_{\Ta}$ such that
\be \label{eq_Hdv_ext_gen}
    \min_{\substack{\tv_p \in \RTN_p(\Ta)\\
    \tv_p \scp \tn_\sd = r_\sd \,\, \forall \sd \in \Fab\\
    \tv_p \scp \tn_\sd = r_\sd \,\, \forall \sd \in \FaN\\
    \jump{\tv_p}\scp \tn_\sd = r_\sd \,\, \forall \sd \in \Fain\\
    \Dvb \tv_p|_\elm = r_\elm \,\, \forall \elm \in \Ta}}
    \norm{\tv_p}_\oma \leq C_{\rm st}
    \min_{\substack{\tv \in \HdvTa\\
    \tv \scp \tn_\sd = r_\sd \,\, \forall \sd \in \Fab\\
    \tv \scp \tn_\sd = r_\sd \,\, \forall \sd \in \FaN\\
    \jump{\tv}\scp \tn_\sd = r_\sd \,\, \forall \sd \in \Fain\\
    \Dvb \tv|_\elm = r_\elm \,\, \forall \elm \in \Ta}}
    \norm{\tv}_\oma,
\ee
where the minimization sets are non-empty and both minimizers
in~\eqref{eq_Hdv_ext_gen} are unique. \et

\br[$\FaD \cup \FaN$ lying in two hyperplanes]\label{rem_bnd_other} Theorems~\ref{thm_H1_ext_gen} and~\ref{thm_Hdv_ext_gen} for instance also hold in the case where the set $\FaD \cup \FaN$ is contained in two hyperplanes as in the right panel of Figure~\ref{fig:patch_bnd} and in topologically equivalent situations, in place of the interior vertex condition in $\Fab$ or the condition $|\Ta| \leq 2$, with a similar proof as in Section~\ref{sec_bound_H1_Dir} below. \er

\section{Equivalent reformulations} \label{sec_eq_ref}

We reformulate in this section Theorems~\ref{thm_H1_ext} and~\ref{thm_Hdv_ext} in an equivalent way as best-approximation results of discontinuous piecewise polynomial data. This will in particular allow for a straightforward application to a posteriori error analysis in Section~\ref{sec_a_post_appl}. For further insight, as well as to make a link with previous contributions on the subject, we also give equivalent reformulations of the right-hand sides in~\eqref{eq_H1_ext} and~\eqref{eq_Hdv_ext}. Finally, we also reformulate Theorems~\ref{thm_H1_ext_gen} and~\ref{thm_Hdv_ext_gen} as best-approximation results.

\subsection{Reformulation as best-approximation results}

Let us set
\bse \ba
\Hooi{\oma} &\eq \{v\in H^1(\oma);~ v|_{\pt\oma}=0 \}, \\
\label{eq_H_dvs} \Hdvs &\eq \{\tv \in \Hdvi{\oma};~ \tv \scp \tn_{\pt\oma} =
0\}. \ea \ese
The result of Theorem~\ref{thm_H1_ext} can be rephrased as follows.

\bc[$H^1$ best-approximation] \label{cor_H1_ext} Let the assumptions of
Theorem~\ref{thm_H1_ext} hold true. Consider any $\tau_p\in \PP_p(\Ta)$ so that
$\tau_p|_\sd = 0$ $\forall \sd\in\Fab$ and $\jump{\tau_p}_\sd = r_\sd$
$\forall \sd \in \Fain$. Then the following holds true: \be \label{eq_H1_ext_equiv}
    \min_{v_p \in \PP_p(\Ta) \cap \Hooi{\oma}} \norm{\Grb (\tau_p - v_p)}_\oma
    \leq C_{\rm st}
    \min_{v \in \Hooi{\oma}} \norm{\Grb (\tau_p - v)}_\oma.
\ee \ec

\bp Direct consequence of~\eqref{eq_H1_ext} upon shifting the minimization
sets by $\tau_p$. Note that the existence of $\tau_p$ follows from the
non-emptiness of the discrete minimization set in~\eqref{eq_H1_ext}. \ep

\br[Minimizers] The unique minimizers in~\eqref{eq_H1_ext_equiv} are
respectively $\pra^\ver \in \PP_p(\Ta) \cap \Hooi{\oma}$ such that
\be \label{eq_sah}
    (\Gr \pra^\ver, \Gr v_p)_\oma = (\Grb \tau_p, \Gr v_p)_\oma \qquad \forall v_p \in \PP_p(\Ta) \cap \Hooi{\oma},
\ee
and $s^\ver \in \Hooi{\oma}$ such that
\be \label{eq_sa}
    (\Gr s^\ver, \Gr v)_\oma = (\Grb \tau_p, \Gr v)_\oma \qquad \forall v \in \Hooi{\oma}. \ee
The minimizers in~\eqref{eq_H1_ext} are such that $v_p^* = \tau_p -
\pra^\ver$ and $v^* = \tau_p - s^\ver$. \er

Similarly, in the $\tH(\dv)$-setting, Theorem~\ref{thm_Hdv_ext} can be reformulated as follows.

\bc[$\tH(\dv)$ best-approximation] \label{cor_Hdv_ext} Let the assumptions of Theorem~\ref{thm_Hdv_ext} hold true. Consider any $\btau_p \in \RTN_p(\Ta)$ so
that $\btau_p\scp\tn_\sd=r_\sd$ $\forall \sd\in\Fab$ and
$\jump{\btau_p} \scp \tn_\sd = r_\sd$ $\forall \sd \in \Fain$. Then the
following holds true:
\be \label{eq_Hdv_ext_equiv}
    \min_{\substack{\tv_p \in \RTN_{p}(\Ta) \cap \Hdvs\\
    \Dv \tv_p|_\elm=r_\elm - \Dvb \btau_p|_\elm \,\, \forall \elm \in \Ta}} \norm{\btau_p + \tv_p}_\oma
    \leq C_{\rm st}
    \min_{\substack{\tv \in \Hdvs\\
    \Dv \tv|_\elm=r_\elm - \Dvb \btau_p|_\elm \,\, \forall \elm \in \Ta}} \norm{\btau_p + \tv}_\oma.
\ee
\ec

\bp Direct consequence of~\eqref{eq_Hdv_ext} upon shifting the minimization
sets by $\btau_p$, the existence of $\btau_p$ following from the
non-emptiness of the discrete minimization set in~\eqref{eq_Hdv_ext}. \ep

\br[Minimizers] The unique minimizers in~\eqref{eq_Hdv_ext_equiv} are
respectively $\fra^\ver \in \RTN_{p}(\Ta) \cap \Hdvs$ with $\Dv
\fra^\ver|_\elm=r_\elm - \Dvb \btau_p|_\elm$ for all $\elm \in \Ta$
such that
\be \label{eq_fah}
    (\fra^\ver, \tv_p)_\oma = - (\btau_p, \tv_p)_\oma \qquad \forall \tv_p \in \RTN_{p}(\Ta) \cap
    \Hdvs, \, \Dv \tv_p = 0,
\ee
and $\fr^\ver \in \Hdvs$ with $\Dv \fr^\ver|_\elm = r_\elm - \Dvb
\btau_p|_\elm$ for all $\elm \in \Ta$ such that
\be \label{eq_fa}
    (\fr^\ver, \tv)_\oma = - (\btau_p, \tv)_\oma \qquad \forall \tv \in \Hdvs,
    \, \Dv \tv=0.
\ee
The minimizers in~\eqref{eq_Hdv_ext} are such that $\tv_p^* = \btau_p +
\fra^\ver$ and $\tv^* = \btau_p + \fr^\ver$. \er

\subsection{Equivalent reformulations at the continuous level} \label{sec_ref}

We summarize here additional equivalence results on the continuous-level minimizations appearing in the right-hand sides of~\eqref{eq_H1_ext_equiv} and~\eqref{eq_Hdv_ext_equiv}. Let us first set
\[
    \Hsa \eq \{v \in \Hoi{\oma}; \, (v, 1)_\oma = 0\},
\]
and let us define the following subspace of $\tH(\crl,\oma)$:
\[
\tH_*(\crl,\oma) \eq \{ \tv \in \tH(\crl,\oma);\, (\tv,\nabla\phi)_{\oma} = 0,\,\forall \phi\in \Hsa\}.
\]

We first show that the $\Hooi{\oma}$-minimization of Corollary~\ref{cor_H1_ext} is equivalent to evaluating a dual $\tH(\crl)$-norm of a suitable linear form defined from the data $r_\sd$, and consequently to evaluating the energy norm of its $\tH_*(\crl,\oma)$-lifting.

\bc[$\tH(\crl)$ form of the $H^1$-minimization] \label{cor_H1_ext_equiv} Let the assumptions of Corollary~\ref{cor_H1_ext} hold true. Recall that $\tau_p\in \PP_p(\Ta)$ verifies $\tau_p|_\sd = 0$ $\forall \sd\in\Fab$ and $\jump{\tau_p}_\sd = r_\sd$ $\forall \sd \in \Fain$. Let $\tr^\ver \in \tH_*(\crl,\oma)$ solve
\be \label{eq_tr_ver}
    (\Crl \tr^\ver, \Crl \tv)_\oma = -(\Grb \tau_p, \Crl \tv)_\oma
    \qquad \forall \tv \in \tH_*(\crl,\oma),
\ee
where the right-hand side can be formally rewritten as $(\Grb \tau_p, \Crl \tv)_\oma = \sum_{\sd\in\Fain} (r_\sd\tn_\sd,\Crl\tv)_\sd$ owing to the elementwise Green formula. Then, we have
\be
    \min_{v \in \Hooi{\oma}} \norm{\Grb (\tau_p - v)}_\oma = \norm{\Crl \tr^\ver}_\oma = \max_{\substack{\tv \in \tH(\crl,\oma)\\ \norm{\Crl \tv}_\oma = 1}} \left\{ \sum_{\sd\in\Fain} (r_\sd\tn_\sd,\Crl\tv)_\sd\right\}.
\ee
\ec

\bp Since $s^\ver$ solves~\eqref{eq_sa}, i.e., $(\Gr s^\ver-\Grb \tau_p, \Gr v)_\oma = 0$ for all $v \in \Hooi{\oma}$, a distributional argument implies that the vector field $\Gr \pr^\ver-\Grb \tau_p$ is divergence-free in $\oma$. The boundary $\pt\oma$ being connected, we infer that there is
$\tr^\ver\in \tH(\crl,\oma)$ such that $\Gr \pr^\ver - \Grb \tau_p = \Crl \tr^\ver$, and without loss of generality, we can take $\tr^\ver\in \tH_*(\crl,\oma)$
since $\oma$ is simply connected so that $\tH(\crl,\oma) = \tH_*(\crl,\oma) \oplus \nabla \Hsa$ (the sum being $L^2$-orthogonal) and fields in $\nabla \Hsa$ are curl-free. We now observe that we have
\be \label{eq_crl}
    (\Crl \tr^\ver, \Crl \tv)_\oma = (\Gr \pr^\ver - \Grb \tau_p, \Crl \tv)_\oma = - (\Grb \tau_p, \Crl \tv)_\oma \qquad \forall \tv \in \tH_*(\crl,\oma),
\ee
since $\pr^\ver \in \Hooi{\oma}$, so that~\eqref{eq_tr_ver} follows. Finally,
\[
    \min_{v \in \Hooi{\oma}} \norm{\Grb (\tau_p - v)}_\oma =
        \norm{\Grb (\tau_p - \pr^\ver)}_\oma = \norm{\Crl \tr^\ver}_\oma
        = \max_{\substack{\tv \in \tH(\crl,\oma)\\ \norm{\Crl \tv}_\oma = 1}} (\Grb \tau_p, \Crl \tv)_\oma,
\]
using~\eqref{eq_crl} and noting that any function $\tv \in \nabla \Hsa$ is automatically excluded from the maximization set by the constraint $\norm{\Crl \tv}_\oma = 1$.
\end{proof}

Let us now show that the constrained $\Hdvs$-minimization of Corollary~\ref{cor_Hdv_ext} is equivalent to evaluating a dual $H^1$-norm of a suitable linear form defined from the data $r_\elm$ and $r_\sd$ and consequently to evaluating the energy norm of its $\Hsa$-lifting.

\bc[$H^1$ form of the $\tH(\dv)$-minimization] \label{cor_Hdv_ext_equiv} Let the assumptions of Corollary~\ref{cor_Hdv_ext} hold true. Let $r^\ver \in \Hsa$ solve
\be \label{eq_r_ver}
    (\Gr r^\ver, \Gr v)_\oma = \sum_{\elm \in \Ta} (r_\elm, v)_\elm - \sum_{\sd \in
    \Fa} (r_\sd, v)_\sd \qquad \forall v \in \Hsa.
\ee
Then
\be \label{eq_r_ver_dual}
    \min_{\substack{\tv \in \Hdvs\\
    \Dv \tv|_\elm = r_\elm - \Dvb \btau_p|_\elm\,\, \forall \elm \in \Ta}} \norm{\btau_p + \tv}_\oma =
    \norm{\Gr r^\ver}_\oma = \max_{\substack{v \in \Hoi{\oma}\\ \norm{\Gr v}_\oma = 1}}\left\{\sum_{\elm \in \Ta} (r_\elm, v)_\elm - \sum_{\sd \in
    \Fa} (r_\sd, v)_\sd\right\}.
\ee
\ec

\bp The elementwise Green formula combined with the definition of $\btau_p$ gives
\ban \sum_{\elm \in \Ta} (r_\elm, v)_\elm - \sum_{\sd \in
    \Fa} (r_\sd, v)_\sd {} & = \sum_{\elm \in \Ta} (r_\elm, v)_\elm - \sum_{\elm \in \Ta}(\btau_p \scp
    \tn_\elm, v)_{\pt \elm}\\
    {} & = \sum_{\elm \in \Ta} (r_\elm - \Dvb \btau_p, v)_\elm - \sum_{\elm \in \Ta}(\btau_p, \Gr v)_\elm,
\ean
and we immediately see that~\eqref{eq_r_ver} is the primal formulation of~\eqref{eq_fa}. As both formulations are equivalent, $\fr^\ver = - \Gr r^\ver - \btau_p$, \cf\ \cite[Remark~3.15]{Ern_Voh_p_rob_15}. The equality~\eqref{eq_r_ver_dual} follows immediately from~\eqref{eq_r_ver}, writing the maximum first for all $v \in \Hsa$ with $\norm{\Gr v}_\oma = 1$ and then noting that any function $v$ constant on $\oma$ is automatically excluded from the maximization set by the constraint $\norm{\Gr v}_\oma = 1$. \ep

Corollaries~\ref{cor_H1_ext_equiv} and~\ref{cor_Hdv_ext_equiv} allow us to draw insightful links with the literature. On the one hand,
Corollary~\ref{cor_H1_ext_equiv}
explains how the right-hand side in~\eqref{eq_H1_ext_equiv} links to the continuous minimization used in~\cite[Lemma~3.13]{Ern_Voh_p_rob_15}. Therein, in two space dimensions, the field $\Re^\perp(\Grb \tau_p)$ has been employed in the definition of the function $r_\verbis$ by formulas~(3.19) and~(3.32), where $\Re^\perp = \left(\begin{smallmatrix}0&-1\\1&\phantom{-}0\end{smallmatrix}\right)$ is the rotation by $\frac{\pi}{2}$; then $\norm{\Gr r_\verbis}_\omabis$ of~\cite{Ern_Voh_p_rob_15} equals the present $\min_{v \in \Hooi{\oma}} \norm{\Grb (\tau_p - v)}_\oma$, and, in particular, we have
\[
    \min_{v \in \Hooi{\oma}} \norm{\Grb (\tau_p - v)}_\oma =
    \max_{\substack{v \in \Hoi{\oma}\\ \norm{\Gr v}_\oma = 1}}
    \left\{-\left(\Re^\perp(\Grb \tau_p), \Gr v\right)_\oma\right\}.
\]
On the other hand, the maximization form in Corollary~\ref{cor_Hdv_ext_equiv}
has been used previously in~\cite[Theorem~7]{Brae_Pill_Sch_p_rob_09} and~\cite[Lemma~3.12 and Corollary~3.16]{Ern_Voh_p_rob_15}.

\subsection{Boundary vertices}
\label{sec:boundary_ver_ref}

In this section, we reformulate Theorems~\ref{thm_H1_ext_gen} and~\ref{thm_Hdv_ext_gen} as best-approximation results on discontinuous piecewise polynomial data on boundary patches. The proofs are omitted since they are similar to the previous ones. In view of application to a posteriori error analysis of model problems with non-homogeneous boundary conditions, it is convenient to introduce some additional boundary data denoted by $\uu_\ver^{\mathrm{D}}$ and $\sigma_\ver^{\mathrm{N}}$ in the $H^1$ and $\tH(\dv)$ settings, respectively.

\bc[$H^1$ best-approximation] \label{cor_H1_ext_gen} Let the assumptions of
Theorem~\ref{thm_H1_ext_gen} hold true. Let $\uu_\ver^{\mathrm{D}} \in \PP_p(\FaD) \cap C^0(\pt\omaD)$. Consider any $\tau_p\in \PP_p(\Ta)$ so
that $\tau_p|_\sd = 0$ $\forall \sd\in\Fab$, $\tau_p|_\sd -
\uu_\ver^{\mathrm{D}}|_\sd = r_\sd$ $\forall \sd\in\FaD$, and
$\jump{\tau_p}_\sd = r_\sd$ $\forall \sd \in \Fain$. Then the following holds true:
\be \label{eq_H1_ext_equiv_gen}
    \min_{\substack{v_p \in \PP_p(\Ta) \cap \Hoi{\oma}\\ v_p|_\sd = 0 \,\, \forall \sd \in \Fab\\
    v_p|_\sd = \uu_\ver^{\mathrm{D}}|_\sd \,\, \forall \sd \in \FaD}} \norm{\Grb (\tau_p - v_p)}_\oma
    \leq C_{\rm st}
    \min_{\substack{v \in \Hoi{\oma}\\ v|_\sd = 0 \,\, \forall \sd \in \Fab\\
    v|_\sd = \uu_\ver^{\mathrm{D}}|_\sd \,\, \forall \sd \in \FaD}} \norm{\Grb (\tau_p - v)}_\oma
\ee \ec

\bc[$\tH(\dv)$ best-approximation] \label{cor_Hdv_ext_gen} Let the
assumptions of Theorem~\ref{thm_Hdv_ext_gen} hold true. Let $\sigma_\ver^{\mathrm{N}} \in \PP_p(\FaN)$. Consider any $\btau_p \in
\RTN_p(\Ta)$ so that $\btau_p\scp\tn_\sd=r_\sd$ $\forall \sd\in\Fab$,
$\btau_p\scp\tn_\sd + \sigma_\ver^{\mathrm{N}} = r_\sd$ $\forall \sd\in\FaN$,
and $\jump{\btau_p} \scp \tn_\sd = r_\sd$ $\forall \sd \in \Fain$. Then the
following holds true:
\be \label{eq_Hdv_ext_equiv_gen}
    \min_{\substack{\tv_p \in \RTN_{p}(\Ta) \cap \Hdvi{\oma}\\
    \tv_p \scp \tn_\sd = 0 \,\, \forall \sd \in \Fab\\
    \tv_p \scp \tn_\sd = \sigma_\ver^{\mathrm{N}}|_\sd \,\, \forall \sd \in \FaN\\
    \Dv \tv_p|_\elm = r_\elm - \Dvb \btau_p|_\elm \,\, \forall \elm \in \Ta}} \norm{\btau_p + \tv_p}_\oma
    \leq C_{\rm st}
    \min_{\substack{\tv \in \Hdvi{\oma}\\
    \tv \scp \tn_\sd = 0 \,\, \forall \sd \in \Fab\\
    \tv \scp \tn_\sd = \sigma_\ver^{\mathrm{N}}|_\sd \,\, \forall \sd \in \FaN\\
    \Dv \tv|_\elm = r_\elm - \Dvb \btau_p|_\elm \,\, \forall \elm \in \Ta}} \norm{\btau_p + \tv}_\oma.
\ee
\ec

\section{Application to a posteriori error analysis}\label{sec_a_post_appl}

We show in this section how to apply our results to a posteriori error analysis. For this purpose, let $\Omega\subset\mathbb R^3$ be a polyhedral Lipschitz domain (open, bounded, and connected set). Let $\Th$ be a matching tetrahedral mesh of $\Om$, shape-regular with parameter $\gamma_{\Th}>0$ that bounds the ratio of any element diameter to the diameter of its largest inscribed ball. All faces of the mesh are collected in the set $\Fh$, with faces lying on the boundary of $\Om$ forming two disjoint sets $\FN$ and $\FD$ covering two subdomains $\GN$ and $\GD$ that form a partition of $\pt \Om$. Consider the Laplace problem
\bse \label{eq_Lap} \bat{2}
    \ds - \Lap \uu & = f \qquad & & \mbox{ in } \, \Om, \label{eq_Lap_eq} \\
    \ds \uu & = \uuD & & \mbox{ on } \, \GD, \label{eq_Lap_Dir}\\
    \ds - \Gr \uu \scp \tn_\Om & = \uuN & & \mbox{ on } \, \GN, \label{eq_Lap_Neu}
\eat \ese
where, for simplicity, $f \in \PP_{p'-1}(\Th)$, $\uuD \in \PP_{p'}(\FD) \cap
C^0(\GD)$, and $\uuN \in \PP_{p'-1}(\FN)$, for a polynomial degree $p' \geq
1$. If $|\GD| = 0$, we need to additionally suppose the Neumann compatibility condition $(f,1)_{\Om}=(\uuN,1)_{\pt \Om}$. The weak solution of problem~\eqref{eq_Lap} is
a function $\uu \in \Ho$ such that $\uu|_{\GD} = \uuD$ and such that
\be \label{eq_Lapl_VF_inh}
    (\Gr \uu, \Gr v)_{\Om} = (f,v)_{\Om} - (\uuN, v)_{\GN} \qquad \forall v \in \Ho \text{ such that } v|_{\GD}=0.
\ee
When $\GD = \emptyset$, uniqueness is imposed through $(\uu,1)_\Om=0$.
For more general data $f$, $\uuD$, and $\uuN$, data oscillation
terms arise in the a posteriori error analysis, see~\cite{Dol_Ern_Voh_hp_16} and references therein for details.

Let $\uh \in \PP_{p'}(\Th)$ be an approximate solution to the problem~\eqref{eq_Lap}; $\uh$ can be primal-nonconforming in the sense that $\uh \not \in \Ho$ and $\uh|_{\GD} \neq \uuD$, as well as dual-nonconforming in the sense that $-\Grb \uh \not \in \Hdv$, $\Dv(-\Grb \uh) \neq f$, and $(-\Grb \uh \scp \tn_\Om)|_{\GN} \neq \uuN$. The results of this paper have a direct application to a posteriori error analysis since they allow us to construct two central objects leading to guaranteed reliability and $p$-robust local efficiency. The first is a so-called potential reconstruction $\prh \in \PP_{p'+1}(\Th) \cap \Ho$, equal to $\uuD$ on $\GD$. The second one is a so-called equilibrated flux reconstruction $\frh \in \RTN_{p'}(\Th) \cap \Hdv$, such that $\Dv \frh = f$ in $\Om$ and $\frh \scp \tn|_{\GN} = \uuN$ on $\GN$.

Let us collect all the mesh vertices in the set $\Vh$, and for any mesh vertex $\ver \in \Vh$, let the patch $\Ta\subset\Th$ be given by the elements in $\Th$ having $\ver$ as vertex, whereas $\oma\subset\Omega$ is the corresponding open subdomain of $\Om$. Let $\psi_\ver$ be the ``hat'' function associated with the vertex $\ver$: this is a continuous function, piecewise affine with respect to the mesh $\Th$, which takes the value $1$ at the vertex $\ver$ and $0$ at the other vertices. Its support is thus the closure of the patch subdomain $\oma$. We also split the vertex set as $\Vh=\Vhint\cup\Vhext$, where $\Vhint$ contains all interior vertices and $\Vhext$ all boundary vertices. The faces of the elements in the interior patches $\Ta$ (\ie, associated with an interior vertex $\ver$) are collected in $\Fa=\Fain \cup \Fab$, in conformity with Section~\ref{sec_set}. For a boundary vertex $\ver \in \Vhext$, the split is $\Fa = \Fain \cup \Fab \cup \FaD \cup \FaN$, as in Section~\ref{sec:boundary_ver_set}, where
$\FaD$ collects the Dirichlet boundary faces from $\partial\oma\cap\GD$ and sharing the point $\ver$, and $\FaN$ the Neumann boundary faces from $\partial\oma\cap\GN$ and sharing the point $\ver$. To have a more unified formalism between interior and boundary vertices, we conventionally define $\FaD$ and $\FaN$ to be empty sets for all $\ver\in\Vhint$.

We define the potential reconstruction following~\cite[Construction~3.8 and Remark~3.10]{Ern_Voh_p_rob_15}, \cf\ also~\cite{Cars_Merd_survey_NC_13}, as $\prh \eq \sum_{\ver \in \Vh} \pra^\ver$, where $\pra^\ver$ is the discrete minimizer of Corollary~\ref{cor_H1_ext} given by~\eqref{eq_sah} for interior vertices and similarly the discrete minimizer of Corollary~\ref{cor_H1_ext_gen} for boundary vertices. We choose the polynomial degree $p$ of our theory to be $p \eq p'+1$ and we set for all $\ver \in \Vh$
\bse\label{eq_tau_r_pot}\bat{2}
    \tau_p & \eq \psi_\ver \uh \qquad & & \mbox{ in } \, \oma,\\
    r_\sd & \eq \psi_\ver \jump{\uh}_\sd \qquad & & \mbox{ on all } \, \sd \in
    \Fain,\\
    r_\sd & \eq 0 \qquad & & \mbox{ on all } \, \sd \in
    \Fab,\\
    r_\sd & \eq \psi_\ver(\uh - \uuD) \qquad & & \mbox{ on all } \, \sd \in
    \FaD,\\
    \uu_\ver^{\mathrm{D}} & \eq \psi_\ver \uuD \qquad & & \mbox{ on all } \, \sd \in \FaD.
\eat\ese
By construction, the polynomial data satisfy the compatibility conditions~\eqref{eq_conds_H1} and~\eqref{eq_conds_H1_gen}.
Similarly, following~\cite{Dest_Met_expl_err_CFE_99},
\cite{Braess_Scho_a_post_edge_08, Brae_Pill_Sch_p_rob_09},
and~\cite[Construction~3.4 and Remark~3.7]{Ern_Voh_p_rob_15}, we define the
equilibrated flux reconstruction as $\frh \eq \sum_{\ver \in \Vh}
\fra^\ver$, where $\fra^\ver$ is the discrete minimizer of Corollary~\ref{cor_Hdv_ext} given by~\eqref{eq_fah} for interior vertices
and similarly the discrete minimizer of Corollary~\ref{cor_Hdv_ext_gen} for boundary vertices, with the polynomial degree $p \eq p'$. Here we set, for all $\ver \in \Vh$,
\bse\label{eq_tau_r_fl}\bat{2}
    \btau_p & \eq \psi_\ver \Grb \uh \qquad & & \mbox{ in } \, \oma, \\
    r_\elm & \eq \psi_\ver (f + \Lapb \uh) \qquad & & \mbox{ in all } \, \elm \in \Ta,\\
    r_\sd & \eq \psi_\ver \jump{\Grb \uh} \scp \tn_\sd \qquad & & \mbox{ on all } \, \sd \in
    \Fain,\\
    r_\sd & \eq 0 \qquad & & \mbox{ on all } \, \sd \in
    \Fab,\\
    r_\sd & \eq \psi_\ver(\Grb \uh \scp \tn_\sd + \uuN) \qquad & & \mbox{ on all } \, \sd \in \FaN,\\
    \sigma_\ver^{\mathrm{N}} & \eq \psi_\ver \uuN \qquad & & \mbox{ on all } \, \sd \in \FaN.
\eat\ese
For all interior vertices and for those boundary vertices which are only
shared by Neumann faces (\ie, $\FaD=\emptyset$), the hat-function orthogonality
\be \label{eq_FE_comp}
    (\Grb \uh, \Gr \psi_\ver)_\oma = (f, \psi_\ver)_\oma - (\uuN, \psi_\ver)_{\pt \oma \cap \GN}
\ee
is a equivalent to the data compatibility conditions~\eqref{eq_comp} and~\eqref{eq_comp_gen}. This relation is not verified, for example, for certain
discontinuous Galerkin methods; the use of the discrete gradient $\Grd
\uh$ from \cite[Section~4.3]{Di_Pietr_Ern_book_12} in place of the broken gradient $\Grb \uh$ allows to fix this,
see~\cite{Ern_Voh_p_rob_15, Dol_Ern_Voh_hp_16}. This altogether leads to:

\bc[Guaranteed a posteriori error estimate] \label{cor_rel} Let the data in problem~\eqref{eq_Lap} satisfy $f \in \PP_{p'-1}(\Th)$, $\uuD \in \PP_{p'}(\FD) \cap C^0(\GD)$, and $\uuN \in \PP_{p'-1}(\FN)$, $p' \geq 1$.
Let $\uu \in \Ho$ such that $\uu|_{\GD} = \uuD$ be the weak solution of~\eqref{eq_Lapl_VF_inh}. Let $\uh \in \PP_{p'}(\Th)$ satisfying~\eqref{eq_FE_comp} be arbitrary. Let $\prh \eq \sum_{\ver \in \Vh} \pra^\ver$, where $\pra^\ver$ is the discrete minimizer of Corollary~\ref{cor_H1_ext} or~\ref{cor_H1_ext_gen} with $p = p'+1$ and data~\eqref{eq_tau_r_pot}. Let $\frh \eq \sum_{\ver \in \Vh} \fra^\ver$, where $\fra^\ver$ is the discrete minimizer of Corollary~\ref{cor_Hdv_ext} or~\ref{cor_Hdv_ext_gen} with $p = p'$ and data~\eqref{eq_tau_r_fl}. Then
\[
    \norm{\Grb(\uu - \uh)}_{\Om}^2 \leq \sum_{\elm \in \Th} \big(
    \norm{\Grb \uh + \frh}_\elm^2 + \norm{\Grb (\uh - \prh)}_\elm^2\big),
\]
as well as
\ban
    \norm{\Grb(\uu - \uh)}_{\Om}^2 + \sum_{\sd \in \Fhint \cup \FD} h_\sd^{-1} \norm{\Pi_\sd^0 \jump{\uu - \uh}}_\sd^2 \leq {} & \sum_{\elm \in \Th} \big(\norm{\Grb \uh + \frh}_\elm^2 + \norm{\Grb (\uh - \prh)}_\elm^2\big)\\
    & + \sum_{\sd \in \Fhint} h_\sd^{-1} \norm{\Pi_\sd^0 \jump{\uh}}_\sd^2
    + \sum_{\sd \in \FD} h_\sd^{-1} \norm{\Pi_\sd^0 (\uh - \uuD) }_\sd^2.
\ean
\ec

\bp See~\cite[Theorem~3.3]{Ern_Voh_p_rob_15}
or~\cite[Theorem~3.3]{Dol_Ern_Voh_hp_16} and the references therein.\ep

Let
\bse \label{eq_Hsa} \bat{2}
    \Hsa & \eq \{v \in \Hoi{\oma}; \, (v, 1)_\oma = 0\},
        \qquad \qquad \quad & & \ver \in \Vhint \text{ or } \big(\ver \in \Vhext \text{ and } \FaD = \emptyset \big),  \\
    \Hsa & \eq \{v \in \Hoi{\oma}; \, v = 0 \text{ on all }
        \sd \in \FaD\}, & & \ver \in \Vhext \text{ and } \FaD \neq \emptyset.
\eat \ese
Then the Poincar\'e(--Friedrichs) inequality states that
\[
    \norm{v}_\oma \leq C_{\mathrm{PF},\oma} h_\oma \norm{\Gr v}_\oma \qquad \forall v \in \Hsa.
\]
Similarly, the broken Poincar\'e(--Friedrichs) inequality~\cite{Brenner_Poin_disc_03, Voh_Poinc_disc_05} states that
\[
    \norm{v}_\oma \leq C_{\mathrm{bPF}, \oma} h_\oma \Biggl[\norm{\Grb v}_\oma + \Biggl\{\sum_{\sd \in \Fain \cup \FaD} h_\sd^{-1} \norm{\Pi_\sd^0 \jump{v}}_\sd^2 \Biggr\}^\ft\Biggr]
\]
for all $v \in \HTa$ such that $(v,1)_\oma = 1$ if $\ver \in \Vhint$, all $v \in \HTa$ such that $\sum_{\sd \in \FaN} (v, 1)_\sd = 0$ if $\ver \in \Vhext$ and
$\FaD = \emptyset$, and all $v \in \HTa$ if $\ver \in \Vhext$ and
$\FaD \neq \emptyset$.
Let $C_{\mathrm{cont,PF}} \eq \max_{\ver \in \Vh} \{1 +
C_{\mathrm{PF},\oma} h_\oma \norm{\Gr \psi_\ver}_{\infty,\oma}\}$ and $C_{\mathrm{cont,bPF}} \eq \max_{\ver \in \Vh} \{1 +
C_{\mathrm{bPF},\oma} h_\oma \norm{\Gr \psi_\ver}_{\infty,\oma}\}$, where both constants only depend on the shape-regularity parameter $\gamma_{\Th}$.
Then we have
\be \label{eq_contPF}
    \norm{\Gr(\psi_\ver v)}_\oma \leq C_{\mathrm{cont,PF}} \norm{\Gr v}_\oma \qquad \forall v \in \Hsa,
\ee
see~\cite{Brae_Pill_Sch_p_rob_09} or~\cite[Lemma~3.12]{Ern_Voh_p_rob_15}, and, similarly,
\be \label{eq_contbPF} \begin{split}
    \norm{\Grb(\psi_\ver v)}_\oma
    \leq {} & C_{\mathrm{cont,bPF}} \Bigg[ \norm{\Grb v}_\oma + \Biggl\{\sum_{\sd \in \Fain \cup \FaD} h_\sd^{-1} \norm{\Pi_\sd^0 \jump{v}}_\sd^2 \Biggr\}^\ft\Bigg],
\end{split}
\ee
for all $v \in \HTa$ with the above constraints. Let $\VK$ stand for the vertices of the element $\elm$. The crucial application of our results is:

\bc[Local efficiency and polynomial-degree robustness] \label{cor_eff} For the estimators of Corollary~\ref{cor_rel}, the following holds true:
\bse \ba
    \norm{\Grb (\uh - \prh)}_\elm \leq {} & C_{\rm st} C_{\mathrm{cont,bPF}} \sum_{\ver \in \VK} \Biggl[ \norm{\Grb (\uu - \uh)}_\oma \nn \\
     & + \Biggl\{\sum_{\sd \in \Fain \cup \FaD} h_\sd^{-1} \norm{\Pi_\sd^0 \jump{\uu - \uh}}_\sd^2 \Biggr\}^\ft \Biggr] \qquad \forall \elm \in \Th, \label{eq_eff_H1}\\
    \norm{\Grb \uh + \frh}_\elm \leq {} & C_{\rm st} C_{\mathrm{cont,PF}} \sum_{\ver \in \VK} \norm{\Grb(\uu - \uh)}_\oma \qquad \forall \elm \in \Th,\label{eq_eff_Hdv}\\
    h_\sd^\mft \norm{\Pi_\sd^0 \jump{\uh}}_\sd = {} & h_\sd^\mft \norm{\Pi_\sd^0 \jump{\uu - \uh}}_\sd \qquad \forall \sd \in \Fhint, \label{eq_eff_jmps_int}\\
    h_\sd^\mft \norm{\Pi_\sd^0 (\uh - \uuD) }_\sd = {} & h_\sd^\mft \norm{\Pi_\sd^0 \jump{\uu - \uh}}_\sd \qquad \forall \sd \in \FD. \label{eq_eff_jmps_GD}
\ea \ese
\ec

\bp Corollary~\ref{cor_H1_ext} for interior
vertices and Corollary~\ref{cor_H1_ext_gen} for boundary vertices immediately give, for any $\ver \in \Vh$,
\[
    \norm{\Grb (\psi_\ver \uh - \pra^\ver)}_\oma \leq C_{\rm st}
    \min_{\substack{v \in \Hoi{\oma}\\ v|_\sd = 0 \,\, \forall \sd \in \Fab\\
    v|_\sd = \psi_\ver \uuD \,\, \forall \sd \in \FaD}}
    \norm{\Grb (\psi_\ver \uh - v)}_\oma \leq C_{\rm st} \inf_{\substack{v \in \Hoi{\oma}\\
    v|_\sd = \uuD \,\, \forall \sd \in \FaD}}
    \norm{\Grb (\psi_\ver(\uh - v))}_\oma.
\]
Indeed, the right inequality follows immediately as any function $v \in
\Hoi{\oma}$, equal to $\uuD$ on the faces from $\FaD$, belongs to the
minimization set of the middle term above when multiplied by the hat function
$\psi_\ver$. This means that the discrete fully computable estimator
$\norm{\Grb (\psi_\ver \uh - \pra^\ver)}_\oma$
is a local lower bound for a $\psi_\ver$-weighted distance
to the $\Hoi{\oma}$ space (or an affine subspace if $\FaD\ne\emptyset$).
We now make the weak solution $\uu$
of~\eqref{eq_Lap} appear in the bound. For $\ver \in \Vhint$, let $\tilde \uu
\eq \uu - c_\ver$, where the constant $c_\ver$ is chosen so that $(\tilde
u, 1)_\oma = (\uh, 1)_\oma$. For a boundary vertex $\ver \in \Vhext$ such
that $\FaD = \emptyset$, let $\tilde \uu \eq \uu - c_\ver$, where the
constant $c_\ver$ is chosen so that $\sum_{\sd \in \FaN} (\tilde u, 1)_\sd
= \sum_{\sd \in \FaN} (\uh, 1)_\sd$. In the other situations, we let $\tilde
\uu \eq \uu$. Note that in all three cases, $\Gr \tilde u = \Gr u$ on the patch $\oma$ and $\jump{\tilde u} = \jump{u}$ on all the faces $\sd \in \Fain$. Then, using~\eqref{eq_contbPF} for $v = \uh - \tilde \uu$ together with $\norm{\Grb (\uh - \prh)}_\elm \leq \sum_{\ver \in \VK} \norm{\Grb (\psi_\ver \uh - \pra^\ver)}_\oma$, we obtain~\eqref{eq_eff_H1}.
Note that if the mean values of the jumps of $\uh$ are zero, \ie,
$(\jump{\uh},1)_\sd = 0$ for all the faces $\sd \in \Fhint$ and
$(\uh,1)_\sd = (\uuD,1)_\sd$ for all the Dirichlet faces $\sd \in \FD$,
\eqref{eq_eff_H1} actually simplifies to
\[
    \norm{\Grb (\uh - \prh)}_\elm
    \leq C_{\rm st} C_{\mathrm{cont,bPF}} \sum_{\ver \in \VK} \norm{\Grb (\uu - \uh)}_\oma.
\]

Corollary~\ref{cor_Hdv_ext} for interior vertices and Corollary~\ref{cor_Hdv_ext_gen} for boundary vertices, in conjunction with Corollary~\ref{cor_Hdv_ext_equiv}, yield, for any $\ver \in \Vh$,
\ban
    \norm{\psi_\ver \Grb \uh + \fra^\ver}_\oma
    & \leq C_{\rm st} \max_{\substack{v \in \Hsa\\ \norm{\Gr v}_\oma = 1}}
    \left\{\sum_{\elm \in \Ta} (r_\elm, v)_\elm - \sum_{\sd \in
    \Fa} (r_\sd, v)_\sd\right\} \\
    & = C_{\rm st} \max_{\substack{v \in \Hsa\\ \norm{\Gr v}_\oma = 1}}
    \left\{(f, \psi_\ver v)_\oma - (\Grb \uh, \Gr(\psi_\ver v))_\oma - \sum_{\sd \in \FaN}(\uuN,\psi_\ver v)_\sd \right\},
\ean
using the definition of $r_\elm$ and $r_\sd$ in~\eqref{eq_tau_r_fl} and
the Green formula. Thus, the fully computable estimator $\norm{\psi_\ver \Grb
\uh + \fra^\ver}_\oma$ is a local lower bound for the local dual norm of
the residual with $\psi_\ver$-weighted test functions. We now note that
$\psi_\ver v$, extended by zero outside of the patch subdomain $\oma$, is a function in $\Ho$
which is zero on the Dirichlet part of the boundary $\GD$. Thus we can use the definition~\eqref{eq_Lapl_VF_inh} of the weak solution to replace the right-hand side by $(\Grb (\uu - \uh), \Gr(\psi_\ver v))_\oma$. Invoking~\eqref{eq_contPF} together with $\norm{\Grb \uh + \frh}_\elm \leq \sum_{\ver \in \VK} \norm{\psi_\ver \Grb \uh + \fra^\ver}_\oma$ gives~\eqref{eq_eff_Hdv}.

Finally, \eqref{eq_eff_jmps_int} and~\eqref{eq_eff_jmps_GD} are immediate by definition. \ep

\section{Proof for broken $H^1$ polynomial extensions}
\label{sec:proof_pot}

We prove here Theorem~\ref{thm_H1_ext}. In particular, we show in~Section~\ref{sec:nonempty_pot} the existence of the minimizers in~\eqref{eq_H1_ext}, in Section~\ref{sec:un_pot} their uniqueness, and in~Section~\ref{sec:stab_pot} the stability bound~\eqref{eq_H1_ext}. Let
$p \ge 1$ and let $r \in \PP_{p}(\Fain)$ satisfy the compatibility
conditions~\eqref{eq_conds_H1}. Define
\bse \label{eq:spaces_min_pot}\ba
    V_p(\Ta) &\eq\{ v_p\in \PP_p(\Ta);~ v_p|_\sd=0\,\, \forall \sd\in\Fab,\,\jump{v_p}_\sd = r_\sd \,\, \forall \sd \in \Fain\},\\
    V(\Ta) &\eq\{ v\in H^1(\Ta);~ v|_\sd=0\,\, \forall \sd\in\Fab,\,\jump{v}_\sd = r_\sd \,\, \forall \sd \in \Fain\}.
    \ea\ese
Then the stability bound~\eqref{eq_H1_ext} becomes
\be
    \min_{v_p \in V_p(\Ta)} \norm{\Grb v_p}_\oma
    \leq C_{\rm st} \min_{v \in V(\Ta)} \norm{\Grb v}_\oma.
\ee
To prove it, we crucially consider the enumeration of the cells in the patch $\Ta$ from Lemma~\ref{lem:int_patch_enum} below in the form $\elm_1,\ldots,\elm_{|\Ta|}$. Without loss of generality (see Remark~\ref{rem:orient}), we orient all the interior faces $\sd=\partial \elm_i\cap \partial \elm_j\in \Fain$ so that $\tn_\sd$ points from $\elm_j$ to $\elm_i$ with $j<i$.

In what follows, we abbreviate as $A\lesssim B$ the inequality $A\le cB$ with a generic constant $c$ whose value can only depend on the patch regularity parameter $\gamma_{\Ta}$; the constant $C$ is in particular independent of the polynomial degree $p$.

\subsection{Existence of the minimizers}
\label{sec:nonempty_pot}

Let us first prove that the minimization sets $V_p(\Ta)$ and $V(\Ta)$ are
non-empty; then the existence of the minimizers immediately follows. Since
$V_p(\Ta)\subset V(\Ta)$, we only consider $V_p(\Ta)$. The proof is
constructive in that we build a function in $V_p(\Ta)$ by enumerating all the
cells in $\Ta$ while prescribing suitable Dirichlet data on some faces of
each cell. For all $1\le i\le |\Ta|$, let us set $\sd_i\upb\eq\pt
\elm_i\cap\pt\oma$, \ie, $\sd_i\upb$ is the face of $\elm_i$ lying on the patch subdomain boundary $\pt\oma$. Note that $\sd_i\upb\in \Fab$. Consider a function $w_p\in\PP_p(\Ta)$ such that its restrictions
$w_p^i \eq w_p|_{\elm_i}$, for all $1\leq i\leq |\Ta|$, are defined by induction
as follows:
\bse \label{eq_sp_h1_rec} \begin{enumerate}[(i)]
\item For $i=1$, $w_p^1$ is any function in \be V_p(\elm_1)\eq\{v_p\in
\PP_p(\elm_1);~ v_p|_{\sd_1\upb}=0\}. \ee
\item For all $1<i\le|\Ta|$, $w_p^i$ is any function in
\be \label{eq:VpKi} V_p(\elm_i)\eq\{v_p\in \PP_p(\elm_i);~ v_p|_{\sd_i\upb}=0,\, v_p|_\sd =
-r_\sd+w_p^j|_\sd \, \forall \sd\in \F_i^\sharp\},
\ee
where $j=j(i,\sd)$ is the index of the cell sharing $\sd$ with $\elm_i$, \ie,
$\sd=\partial \elm_i\cap\partial \elm_j$. Recall that by definition of
the set of previously enumerated faces $\F_i^\sharp$ in
Appendix~\ref{app_graphs}, we have $j<i$, so that $w_p^j$ is already known
from a previous step of the construction.
\end{enumerate} \ese
Lemma~\ref{lem:nonempty_H1} below shows that the (affine) subspaces $V_p(\elm_{i})$ are all non-empty, \ie, the above construction is meaningful. Then, it is easy to see that any function $w_p$ constructed as above is in the discrete minimization set $V_p(\Ta)$; in particular, we note that the prescription~\eqref{eq:VpKi} on the faces in $\F_i^\sharp$ implies that $\jump{w_p}_\sd=w_p^j|_\sd-w_p^i|_\sd = r_\sd$.

\begin{lemma}[Non-emptiness] \label{lem:nonempty_H1}
For all $1\leq i\leq |\Ta|$, the set $V_p(K_i)$ is non-empty.
\end{lemma}

\bp The proof is carried out by induction.\\
(1) First, the linear space $V_p(\elm_1)$ is non-trivial. Its dimension is actually equal to the number of Lagrange nodes of order $p$ in the tetrahedron $\elm_1$ that are not located on the face $\sd_1\upb$.\\
(2) Let now $1<i\le|\Ta|$ and suppose that $V_p(\elm_{i-1})$ is non-empty. To prove that $V_p(\elm_i)$ is non-empty, we need
to verify that the prescribed values on the faces of $\elm_i$ are continuous
across the edges they share. Once this is established, it will follow that
$V_p(\elm_i)$ is an affine space whose tangent space has dimension equal to
the number of Lagrange nodes of order $p$ in $\elm_i$ not located in the
faces of $\elm_i$ where a value is prescribed. We distinguish three cases.
\\
(2.a) Case $1<i<|\Ta|$ and $|\F_i^\sharp|=1$, say $\F_i^\sharp=\{\sd_i^1\}$.
There is one edge to consider, namely $\edg=\sd_i\upb\cap \sd_i^1$. The
compatibility condition~\eqref{eq_conds_H1_1} implies that $r_\sd|_\edg=0$, and
we have by construction $w_p^j|_{\edg}=0$ since $\edg\subset \pt\oma$. Hence, the
value we are prescribing on the face $\sd_i^1$ restricted to the
edge $\edg$ is zero.
\\
(2.b) Case $1<i<|\Ta|$ and $|\F_i^\sharp|=2$, say
$\F_i^\sharp=\{\sd_i^1,\sd_i^2\}$ (note that $|\F_i^\sharp|<3$ if $i<|\Ta|$
owing to Lemma~\ref{lem:int_patch_enum}\eqref{en_en_ver}, although we do not
need to make use of this property here). There are three edges to consider,
namely $\edg^1=\sd_i\upb\cap \sd_i^1$, $\edg^2=\sd_i\upb\cap \sd_i^2$, and
$\edg^{12}=\sd_i^1\cap \sd_i^2$. For $\edg^1$ and $\edg^2$, the reasoning is the same as above. For $\edg\eq \edg^{12}$, we first use
Lemma~\ref{lem:int_patch_enum}\eqref{en_en_edg}, giving that $\elm_i$ is the last cell to be enumerated in the rotational path of cells around $\edg$. Thus $w_p$ has been already defined in the previous elements by the induction argument, and the algebraic properties of the jump operator (see~\eqref{eq:sum_edge_jumps}) give
\be \label{eq_comp_el}
    \sum_{\sd \in \Fe \setminus \{\sd_i^1,\sd_i^2\}} \iota_{\sd, \edg}\jump{w_p} = w_p^{j_1} - w_p^{j_2},
\ee
where $\sd_i^1=\partial \elm_i\cap \partial \elm_{j_1}$ and $\sd_i^2=\partial \elm_i\cap \partial \elm_{j_2}$, with $\elm_{j_1}$ following $\elm_i$ in the rotation direction around the edge $\edg$. Second, employing in~\eqref{eq_comp_el} $\jump{w_p}_\sd = r_\sd$  for all $\sd\in\Fe\setminus \{\sd_i^1,\sd_i^2\}$ and the compatibility condition~\eqref{eq_conds_H1_2}, we conclude that
\be \label{eq_jump_i1_i2}
    w_p^{j_1} - w_p^{j_2} = - \iota_{\sd_i^1, \edg} r_{\sd_i^1} -  \iota_{\sd_i^2, \edg} r_{\sd_i^2} = r_{\sd_i^1} -  r_{\sd_i^2}
\ee
on the edge $\edg$; the last equality also employs that $\iota_{\sd_i^1, \edg} = -1$ and $\iota_{\sd_i^2, \edg} = 1$ in the chosen notation (recall that both normal vectors $\tn_{\sd_i^1}$ and $\tn_{\sd_i^2}$ point inward $K_i$ as $j_1, j_2 < i$), see the right panel of Figure~\ref{fig:patch_int}. Equation~\eqref{eq_jump_i1_i2} provides the desired continuity property on $\edg$. \\
(2.c) Case $i=|\Ta|$. Then $\F_i^\sharp$ contains three faces, and the six
edges of $\elm_{|\Ta|}$ need to be considered. The reasoning is the same as
above for the three edges located on $\pt\oma$ and the three edges inside
$\oma$ (using again Lemma~\ref{lem:int_patch_enum}\eqref{en_en_edg} and the compatibility condition~\eqref{eq_conds_H1_2}).
\ep

\subsection{Uniqueness of the minimizers}
\label{sec:un_pot}

The uniqueness of the minimizers in~\eqref{eq_H1_ext} results from the fact
that $V_p(\Ta)$ and $V(\Ta)$ are convex sets (being affine spaces) and that the functional we are minimizing is strictly convex on the tangent
spaces of $V_p(\Ta)$ and $V(\Ta)$ (both tangent
spaces are composed of functions
vanishing on $\pt\oma$, where the $H^1$-seminorm defines a strictly convex
functional).

\subsection{Proof of the stability bound~\eqref{eq_H1_ext}}
\label{sec:stab_pot}

We now construct two functions $\zeta_p\in \PP_p(\Ta)$ and $\zeta\in
H^1(\Ta)$ such that their restrictions $\zeta_p^i \eq \zeta_p|_{\elm_i}$ and
$\zeta^i \eq \zeta|_{\elm_i}$, for all $1\le i\le |\Ta|$, are defined by
induction as follows:
\begin{enumerate}[(i)]
\item For $i=1$, we define the spaces \bse\ba
V_p(\elm_1)&\eq\{v_p\in \PP_p(\elm_1);~ v_p|_{\sd_1\upb}=0\}, \\
V(\elm_1)&\eq\{v\in H^1(\elm_1);~ v|_{\sd_1\upb}=0\}, \ea\ese and consider
the following unique minimizers:
\be
    \zeta_p^1 \eq \argmin_{v_p\in V_p(\elm_1)} \norm{\nabla v_p}_{\elm_1}, \qquad \zeta^1 \eq \argmin_{v\in V(\elm_1)} \norm{\nabla v}_{\elm_1}.
\ee
Note that these problems are
actually trivial, so that we have $\zeta_p^1=\zeta^1=0$ in $\elm_1$.
\item For all $1<i\le |\Ta|$, we define the spaces
\bse\ba V_p(\elm_i)&\eq\{v_p\in \PP_p(\elm_i);~ v_p|_{\sd_i\upb}=0,\,
v_p|_\sd = -r_\sd + \zeta_p^j|_\sd \, \forall \sd\in \F_i^\sharp\}, \\
V(\elm_i)&\eq\{v\in H^1(\elm_i);~ v|_{\sd_i\upb}=0,\, v|_\sd = -r_\sd +
\zeta_p^j|_\sd \, \forall \sd\in \F_i^\sharp\}, \ea\ese
where $j=j(i,\sd)$ is the index of the cell sharing $\sd$ with $\elm_i$,
\ie, $\sd=\partial \elm_i\cap\partial \elm_j$. Note that both spaces are
defined using the same Dirichlet data which are piecewise polynomials of degree $p$. Consider the following unique minimizers:
\be \label{eq_zeta_i} \zeta_p^i \eq \argmin_{v_p \in V_p(\elm_i)} \norm{\nabla v_p}_{\elm_i},
\qquad \zeta^i \eq \argmin_{v\in V(\elm_i)} \norm{\nabla v}_{\elm_i}. \ee
Note that the above minimization problems are well-posed since the
minimization sets are non-empty. Indeed, the Dirichlet condition is a
continuous, piecewise polynomial, as established in
Section~\ref{sec:nonempty_pot}. Moreover, the set of faces where a
Dirichlet condition is prescribed is always non-empty, and the
minimized functional is strictly convex. We can also observe that the
continuous minimizer $\zeta^i\in H^1(\elm_i)$ is the weak solution to the
problem
\bse \label{eq_aux_1} \bat{2}
    - \Lap \zeta^i & = 0 \qquad & \quad& \mbox{ in } \, \elm_i, \\
    \zeta^i|_\sd & = -r_\sd + \zeta_p^j|_\sd & & \mbox{ on all } \sd \in \F_i^\sharp, \\
    \zeta^i|_\sd & = 0 & & \mbox{ on } \sd_i\upb, \\
    -\Gr \zeta^i \scp \tn_{\elm_i}|_\sd & = 0 & & \mbox{ on all } \sd \in
    \F_i^\flat,
\eat \ese
whereas $\zeta_p^i$ is its (spectral) finite element approximation in $\PP_p (\elm_i)$.
\end{enumerate}
We will show the following two statements for all $1<i\le |\Ta|$ (the
statements being trivial for $i=1$):
\bse\label{eq:stab_all} \ba
\norm{\nabla \zeta_p^i}_{\elm_i} &\lesssim \norm{\nabla\zeta^i}_{\elm_i}, \label{eq:DGS} \\
\norm{\nabla \zeta^i}_{\elm_i} &\lesssim \norm{\nabla_\T v^*}_{\oma} + \sum_{j<i}
\norm{\nabla\zeta_p^j}_{\elm_j}, \label{eq:stab} \ea\ese
where the sum in~\eqref{eq:stab} is void if $i=1$ and
where $v^*\in V(\Ta)$ is the global continuous minimizer in~\eqref{eq_H1_ext}. Since the above inductive
construction implies that $\zeta_p \in V_p(\Ta)$, the global discrete minimizer
$v_p^*\in V_p(\Ta)$ in~\eqref{eq_H1_ext} is such that
\[
\norm{\nabla_\T v_p^*}_{\oma} \le \norm{\nabla_\T\zeta_p}_{\oma}.
\]
Moreover, combining~\eqref{eq:DGS} with~\eqref{eq:stab} proves by induction
that $\|\nabla \zeta_p^i\|_{\elm_i} \lesssim \|\nabla_\T v^*\|_{\oma}$ for
all $1\le i\le |\Ta|$, so that $\|\nabla_\T\zeta_p\|_{\oma} \lesssim
\|\nabla_\T v^*\|_{\oma}$. Hence, $\|\nabla_\T v_p^*\|_{\oma}\lesssim
\|\nabla_\T v^*\|_{\oma}$, and this concludes the proof of~\eqref{eq_H1_ext}.

\bp[Proof of~\eqref{eq:DGS}] We apply Lemma~\ref{lem_H1_simpl} on
$\elm=\elm_i$ with $\FKD = \{F_i\upb\} \cup \F_i^\sharp$ and the $p$-degree polynomials given by $0$ for $\sd=\sd_i\upb$ and
$-r_\sd + \zeta_p^j|_\sd$ for all $\sd\in\F_i^\sharp$. The proof that these
polynomials are continuous over the edges shared by two faces in $\FKD$
has been given in Section~\ref{sec:nonempty_pot}. \ep

\bp[Proof of~\eqref{eq:stab}] We distinguish three cases.
\\
(1) Case $1<i<|\Ta|$ and $|\F_i^\sharp|=1$, say $\F_i^\sharp=\{\sd\}$. Let
$\elm_j\in\Ta$ be the cell such that $\sd = \pt \elm_i \cap \pt \elm_j$; the
definition of $\F_i^\sharp$ implies that $j<i$. Let $\tT : \elm_j\to \elm_i$
be the (unique) bijective affine map leaving $\sd$ pointwise invariant.
Consider in the cell $\elm_i$ the function
\[
v\eq
v^*|_{\elm_i} - (v^*|_{\elm_j}-\zeta_p^j) \circ \tT^{-1}.
\]
The crucial observation is that $v\in V(\elm_i)$. Indeed, the properties
$v\in H^1(\elm_i)$ and $v|_{\sd_i\upb}=0$ are straightforward to verify.
Moreover, since $j<i$, we have
$v^*|_{\elm_j}-v^*|_{\elm_i}=\jump{v^*}_\sd=r_\sd$, so that indeed $v|_\sd =
-r_\sd + \zeta_p^j|_\sd$. Using the definition~\eqref{eq_zeta_i} of $\zeta^i$ together with
the properties of the map $\tT$ which follow from mesh regularity, we infer
that
\begin{align*}
\|\Gr\zeta^i\|_{\elm_i} &\le \|\Gr v\|_{\elm_i} \le
\|\Gr v^*\|_{\elm_i} + \|\Gr((v^*|_{\elm_j}-\zeta_p^j) \circ \tT^{-1})\|_{\elm_i} \\
& \lesssim \|\Gr v^*\|_{\elm_i} + \|\Gr(v^*|_{\elm_j}-\zeta_p^j)\|_{\elm_j}
\leq \|\Gr v^*\|_{\elm_i} + \|\Gr v^*\|_{\elm_j} + \|\Gr \zeta_p^j\|_{\elm_j},
\end{align*}
so that~\eqref{eq:stab} holds true (recall that $j<i$).
\\
(2) Case $1<i<|\Ta|$ and $|\F_i^\sharp|=2$, say $\F_i^\sharp=\{\sd^1,\sd^2\}$
with $\edg=\sd^1\cap \sd^2$. We consider the conforming refinement $\T_\edg'$ of
$\T_\edg$ from Lemma~\ref{lem:2color} applied with $\elm_*=\elm_i$ (observe that the partition $\T_\edg'$ may add one element with respect to $\T_\edg$). Recall that
$\T_\edg'$ contains $\elm_i$, that all the tetrahedra in $\T_\edg'$ have $\edg$ as
edge with their two other vertices lying on $\pt\oma$, that the shape regularities of $\T_\edg'$ and $\T_\edg$ are comparable, and that, collecting
all the vertices of $\T_\edg'$ that are not endpoints of $\edg$ in the set $\V_\edg'$, there
is a two-color map $\mathtt{col} : \V_\edg' \to \{1,2\}$ so that for all
$\kappa\in \T_\edg'$, the two vertices of $\kappa$ that are not endpoints of
$\edg$, say $\{\ver_\kappa^n\}_{1\leq n\leq 2}$, satisfy
$\mathtt{col}(\ver_\kappa^n)=n$. We use the two-color map to define, for all
$\kappa\in \T_\edg'$, the (unique) bijective affine map $\tT_\kappa : \kappa \to
\elm_i$ leaving $\edg$ pointwise invariant and preserving the color of the two
other vertices of $\kappa$, \ie, $\mathtt{col}(\tT_{\kappa}(\ver_\kappa^n)) = n$ for
each $n\in\{1,2\}$. Consider in the cell $\elm_i$ a function $v$ defined
from the global continuous minimizer $v^*$ in $\elm_i$ and from its
difference with the piecewise polynomial $\zeta_p$ on the previously enumerated
elements by
\be \label{eq_v_pot}
    v \eq v^*|_{\elm_i} +
    \sum_{\kappa\in\T_\edg'\setminus\{\elm_i\}}
    \frac{\epsilon_\kappa}{\epsilon_{\elm_i}} (v^*-\zeta_p)|_{\kappa}\circ
    \tT_\kappa^{-1},
\ee
where for all $\kappa\in \T_{\edg}'$ (thus including $\elm_i$), $\epsilon_\kappa=1$ if the orientation of the vector
$\ver_\kappa^1\ver_\kappa^2$ is compatible with the fixed orientation of
rotation around the edge $\edg$ and $\epsilon_\kappa=-1$ otherwise. Note
that the binary coloring implies that $\epsilon_\kappa+\epsilon_{\kappa'}=0$
if the two cells $\kappa$ and $\kappa'$ share an interior face. Let us verify
that $v\in V(\elm_i)$. The properties $v\in H^1(\elm_i)$ and
$v|_{\sd_i\upb}=0$ are again straightforward to verify. It remains to show
that $v|_\sd = -r_\sd + \zeta_p^j|_\sd$ for all $\sd\in \F_i^\sharp$. We
will do the proof for the face $\sd^1$; the proof for the face $\sd^2$ is
similar. Let $\tx\in \sd^1$, and assume without loss of generality that the
vertex of $\sd^1$ that is not in $\edg$ has color $1$. Let $\F_\edg^1$ collect all
the interior faces in $\T_\edg'$ whose vertex that is not in $\edg$ has color 1.
Then, $\sd^1\in \F_\edg^1$ and $\sd^2\not\in \F_\edg^1$. Moreover, the definition
of $\tT_\kappa$ implies that, for all $\kappa\in \T_\edg'$, $\kappa\ne \elm_i$,
the point $\tT_\kappa^{-1}(\tx)$ belongs to the face of $\kappa$ in the set
$\F_\edg^1$. Recall that $\epsilon_\kappa$ has opposite sign on the two cells
sharing each interior face. As a result, we find that the function
$v$ in~\eqref{eq_v_pot} satisfies
\be \label{eq_v_pot_dev}
    v|_{\sd^1} = v^*|_{\elm_i}|_{\sd^1} - (v^*|_{\elm_{j_1}}-\zeta_p^{j_1})|_{\sd^1} \circ \tT_{\elm_{j_1}}^{-1}
    + \sum_{\sd\in\F_\edg^1\setminus\{\sd^1\}} \epsilon_\sd (\jump{v^*}_\sd-\jump{\zeta_p}_\sd)|_{\sd} \circ \tT_{\kappa_\sd}^{-1},
\ee
where $\elm_{j_1}$ is the cell sharing $\sd^1$ with $\elm_i$, \ie,
$\sd^1=\partial \elm_i \cap \partial \elm_{j_1}$, whereas
$\epsilon_\sd=\epsilon_{\kappa_\sd}/\epsilon_{\elm_i} = \pm1$ where
$\kappa_\sd$ is the element sharing $\sd$ having the lowest enumeration
index. Since $\jump{v^*}_\sd=\jump{\zeta_p}_\sd$ for all $\sd\in\F_\edg^1$ (the
common value being $r_\sd$ if $\sd$ is already a face in $\T_\edg$ or zero if
$\sd$ is a newly created face in $\T_\edg'$) and since $j_1<i$, this yields
\[
    v|_{\sd^1} = v^*|_{\elm_i}|_{\sd^1} - (v^*|_{\elm_{j_1}}-
    \zeta_p^{j_1})|_{\sd^1} = -r_{\sd^1}|_{\sd^1}+\zeta_p^{j_1}|_{\sd^1}.
\]
Hence, $v\in V(\elm_i)$. In view of~\eqref{eq_v_pot} and~\eqref{eq_zeta_i}, we conclude that
\[
\|\Gr\zeta^i\|_{\elm_i} \le \|\Gr v\|_{\elm_i}
\lesssim \|\Gr v^*\|_{\elm_i} + \sum_{\kappa\in\T_\edg'\setminus\{\elm_i\}} \{\|\Gr v^*\|_{\kappa} + \|\Gr \zeta_p\|_{\kappa}\}
\le |\T_\edg'|^\ft\|\nabla_\T v^*\|_{\oma} + 2^\ft\sum_{j<i} \|\nabla\zeta_p\|_{\elm_j},
\]
and~\eqref{eq:stab} follows.\\
(3) Case $i=|\Ta|$. Note that owing to Lemma~\ref{lem:int_patch_enum}\eqref{en_en_ver}, this is the only case where $|\F_i^\sharp|=3$ happens, so that we need to work with all the three interior faces of the element $\elm_i$. For this purpose, we apply Lemma~\ref{lem:3color} with $\elm_*=\elm_i$.
Recall that $\T_\ver'$ contains $\elm_i$, that all the tetrahedra in
$\T_\ver'$ have $\ver$ as vertex with their three other vertices lying on
$\pt\oma$, and that, collecting all the vertices of $\T_\ver'$ that lie on
$\pt\oma$ in the set $\V_\ver'$, there is a three-color map $\mathtt{col} :
\V_\ver' \to \{1,2,3\}$ so that for all $\kappa\in \T_\ver'$, the three
vertices of $\kappa$ that are not $\ver$, say $\{\ver_\kappa^n\}_{1\leq n\leq
3}$, satisfy $\mathtt{col}(\ver_\kappa^n)=n$. We use the three-color map to
define, for all $\kappa\in \T_\ver'$, the (unique) bijective affine map
$\tT_\kappa : \kappa \to \elm_i$ leaving $\ver$ invariant and preserving the
color of the three other vertices of $\kappa$. Consider in the cell $ \elm_i$
the function
\be \label{eq_map_3col}
    v \eq v^*|_{\elm_i} + \sum_{\kappa\in \T_\ver'\setminus\{\elm_i\}}
    \frac{\epsilon_\kappa}{\epsilon_{\elm_i}} (v^*-\zeta_p)|_\kappa \circ \tT_\kappa^{-1},
\ee
where, for all $\kappa\in\T_\ver'$, $\epsilon_\kappa=1$ if the vector
$\ver_\kappa^1\ver_\kappa^2 {\times} \ver_\kappa^1\ver_\kappa^3$ points
outward $\oma$ and $\epsilon_\kappa=-1$ otherwise. Let us verify that $v\in
V(\elm_i)$. The properties $v\in H^1(\elm_i)$ and $v|_{\sd_{|\Ta|}\upb}=0$
are straightforward. It remains to verify that $v$ satisfies the appropriate
boundary condition on the three faces in $\sd_i^\sharp$, \ie, on the three
faces of $\elm_i$ sharing the vertex $\ver$. We can call these faces
$\sd^{12}$, $\sd^{13}$, and $\sd^{23}$, where the superscripts refer to the
two colors of the two vertices of the face that are not $\ver$. Let us verify
the boundary condition on $\sd^{12}$; the proof for the two other faces is
similar. Let $\F^{12}$ collect all the interior faces in $\T_\ver'$ such that
their two vertices which are not $\ver$ have colors 1 and 2. Since any
interior face in $\F^{12}$ is shared by two cells in $\T_\ver'$ having
opposite number $\epsilon_\kappa$ and since any cell in $\T_\ver'$ has one
interior face in $\F^{12}$, we infer that
\[
v|_{\sd^{12}} = v^*|_{\elm_i} - (v^*|_{\elm_{j_{12}}}-\zeta_p^{j_{12}})|_{\sd^{12}} \circ \tT_{\elm_{j_{12}}}^{-1}
+ \sum_{\sd\in\F^{12}\setminus\{\sd^{12}\}} \epsilon_\sd (\jump{v^*}_\sd-\jump{\zeta_p}_\sd)|_{\sd} \circ \tT_{\kappa_\sd}^{-1},
\]
with $\epsilon_\sd$ and $\kappa_\sd$ defined as above and where $j_{12}$ is the index of the
cell sharing $\sd^{12}$ with $\elm_i$. Since
$\jump{v^*}_\sd=\jump{\zeta_p}_\sd$ (the common value being either $r_\sd$ or
0) and since $j_{12}<i$, we conclude that $v|_{\sd^{12}} =
-r_{\sd^{12}}|_{\sd^{12}} + \zeta_p^{j_{12}}|_{\sd^{12}}$. Hence,
$v\in V(\elm_i)$. Finally, we can bound $\|\Gr v\|_{\elm_i}$ as above, and
this completes the proof. \ep

\section{Proof for broken $\tH(\dv)$ polynomial extensions}
\label{sec:proof_flux}

We prove here Theorem~\ref{thm_Hdv_ext}. In particular, we show in~Section~\ref{sec:flux_nonempty} the existence of the minimizers in~\eqref{eq_Hdv_ext}, in Section~\ref{sec:flux_un} their uniqueness, and in~Section~\ref{sec:proof_bnd_flux} the stability bound~\eqref{eq_Hdv_ext}. Let $p \ge 0$. Let $r \in \PP_p(\Ta) \times \PP_{p}(\Fa)$ satisfy the compatibility condition~\eqref{eq_comp}. Let us set
\bse \ba \tV_p(\Ta) \eq \{ &\tv_p \in \RTN_p(\Ta);~
    \tv_p \scp \tn_\sd = r_\sd \,\, \forall \sd \in \Fab, \,
    \Dvb \tv_p|_\elm = r_\elm \,\, \forall \elm \in \Ta, \nn \\
    &\jump{\tv_p}\scp \tn_\sd = r_\sd \,\, \forall \sd \in \Fain\}, \\
\tV(\Ta) \eq \{ &\tv \in \HdvTa;~
    \tv \scp \tn_\sd = r_\sd \,\, \forall \sd \in \Fab,\,
    \Dvb \tv|_\elm = r_\elm \,\, \forall \elm \in \Ta,\nn \\
    &\jump{\tv}\scp \tn_\sd = r_\sd \,\, \forall \sd \in \Fain\}.
\ea \ese
Then the stability bound~\eqref{eq_Hdv_ext} becomes
\be \min_{\tv_p \in
\tV_p(\Ta)} \norm{\tv_p}_\oma
    \leq C_{\rm st} \min_{\tv \in \tV(\Ta)} \norm{\tv}_\oma.
\ee
As in Section~\ref{sec:proof_pot}, we consider the enumeration of the cells in
$\Ta$ from Lemma~\ref{lem:int_patch_enum} in the form
$\elm_1,\ldots,\elm_{|\Ta|}$. Without loss of generality (see
Remark~\ref{rem:orient}), we orient all the interior faces
$\sd=\partial \elm_i\cap \partial \elm_j\in \Fain$ so that $\tn_\sd$ points
from $\elm_j$ to $\elm_i$ with $j<i$.

\subsection{Existence of the minimizers}
\label{sec:flux_nonempty}

Let us first prove that the minimization sets $\tV_p(\Ta)$ and $\tV(\Ta)$ are
non-empty, yielding the existence of the minimizers. Since $\tV_p(\Ta)\subset
\tV(\Ta)$, we only consider $\tV_p(\Ta)$. Recall that, for all $1\le i\le
|\Ta|$, $\sd_i\upb=\pt \elm_i\cap\pt\oma$ is the face of the element $\elm_i$ lying on the patch boundary $\pt\oma$. Consider a function $\tw_p\in\RTN_p(\Ta)$ such that its
restrictions $\tw_p^i \eq \tw_p|_{\elm_i}$, for all $1\leq i\leq |\Ta|$, are
defined by induction as follows:

\bse\label{eq:spaces_H_dv}\begin{enumerate}[(i)]
\item For $i=1$, $\tw_p^1$ is any function in \be \tV_p(\elm_1)\eq\{\tv_p\in
\RTN_p(\elm_1);~ \tv_p\scp\tn_{\sd_1\upb}=r_{\sd_1\upb}, \,\, \Dv\tv_p =
r_{\elm_1}\}. \ee
\item For all $1<i\le|\Ta|$, $\tw_p^i$ is any function in
\be \label{eq:tVpKi}
\tV_p(\elm_i)\eq\{\tv_p\in \RTN_p(\elm_i);~
\tv_p\scp\tn_{\sd_i\upb}=r_{\sd_i\upb},\, \Dv\tv_p = r_{\elm_i},\,
\tv_p\scp\tn_\sd = -r_\sd+\tw_p^j\scp\tn_\sd \, \forall \sd\in
\F_i^\sharp\},
\ee
where $j=j(i,\sd)$ is the index of the cell sharing
$\sd$ with $\elm_i$, \ie, $\sd=\partial \elm_i\cap\partial \elm_j$. Recall
that by definition of $\F_i^\sharp$, we have $j<i$, so that $\tw_p^j$ is already
known from a previous step.
\end{enumerate}\ese

Lemma~\ref{lem:nonempty_Hdiv} below shows that the (affine) subspaces $\tV_p(K_i)$ are all non-empty, \ie, the above construction is meaningful. Then, it is easy to see that any function $\tw_p$ constructed as above is in the discrete minimization set $\tV_p(\Ta)$; in particular, we note that the prescription~\eqref{eq:tVpKi} on the faces in $\F_i^\sharp$ implies that $\jump{\tw_p}\scp\tn_\sd = (\tw_p^j-\tw_p^i)\scp\tn_\sd = r_\sd$.

\begin{lemma}[Non-emptiness] \label{lem:nonempty_Hdiv}
For all $1\leq i\leq |\Ta|$, the set $\tV_p(K_i)$ is non-empty.
\end{lemma}

\bp The proof is actually simpler than that of Lemma~\ref{lem:nonempty_H1} in the $H^1$-setting in Section~\ref{sec:nonempty_pot}, as Neumann boundary data in~\eqref{eq:spaces_H_dv} do not request any condition of continuity on edges between faces of each $\elm_i$. Thus, property~\eqref{en_en_edg} of Lemma~\ref{lem:int_patch_enum} is not used here. The only non-trivial property to verify is the compatibility between the prescriptions of the normal component and the divergence whenever $\F_i^\sharp \cup \{\sd_i\upb\}=\F_{\elm_i}$, \ie, whenever the normal component is prescribed over the whole boundary of $\elm_i$. Here, it is important that this situation only happens when $i=|\Ta|$, \ie, for the last cell in the enumeration; this is indeed the case owing to
Lemma~\ref{lem:int_patch_enum}\eqref{en_en_ver}. Then the non-emptiness of $\tV_p(\elm_i)$ follows from classical properties of Raviart--Thomas--N\'ed\'elec finite elements. Let thus $i=|\Ta|$. We need to
check the Neumann compatibility condition
\be\label{eq_comp_last}
    (r_{\elm_i},1)_{\elm_i} = (r_{\sd_i\upb},1)_{\sd_i\upb} + \sum_{\sd\in\F_i^\sharp}
    (r_\sd-\tw_p^j\scp\tn_\sd,1)_\sd.
\ee
(Note that $\tn_\sd$ points inward $\elm_i$ for all $\sd\in\F_i^\sharp$ since $i=|\Ta|$.) Using the divergence theorem in each cell $\elm_j$, $1\le
j<|\Ta|$, we write
\ban \sum_{1\le j<|\Ta|} (r_{\elm_j},1)_{\elm_j} &= \sum_{1\le j<|\Ta|}
(\Dv\tw_p^j,1)_{\elm_j}
= \sum_{1\le j<|\Ta|} \sum_{\sd\in\F_{\elm_j}} (\tw_p^j\scp \tn_{\elm_j},1)_\sd \\
&= \sum_{1\le j<|\Ta|} (r_{\sd_j\upb},1)_{\sd_j\upb}
+ \sum_{\sd\in\Fain\setminus\F_i^\sharp}(\jump{\tw_p}\scp\tn_\sd,1)_\sd + \sum_{\sd\in \F_i^\sharp} (\tw_p^j\scp \tn_\sd,1)_\sd \\
&= \sum_{1\le j<|\Ta|} (r_{\sd_j\upb},1)_{\sd_j\upb} +
\sum_{\sd\in\Fain\setminus\F_i^\sharp}(r_\sd,1)_\sd + \sum_{\sd\in
\F_i^\sharp} (\tw_p^j\scp \tn_\sd,1)_\sd, \ean
where $\tn_{\elm_j}$ is the unit normal pointing outward $\elm_j$.
Then~\eqref{eq_comp_last} follows by combining the above relation with the compatibility
condition~\eqref{eq_comp}.
\ep

\subsection{Uniqueness of the minimizers}
\label{sec:flux_un}

As the affine subspaces $\tV_p(\Ta)$ and $\tV(\Ta)$ are non-empty convex sets, the
uniqueness of the minimizers in~\eqref{eq_Hdv_ext} follows from the fact that
the functional we are minimizing is strictly convex on the tangent spaces
(both tangent spaces are composed of divergence-free functions, so that the
$\|{\cdot}\|_\oma =\|{\cdot}\|_{\HdvTa}$ for such functions).

\subsection{Proof of the stability bound~\eqref{eq_Hdv_ext}}
\label{sec:proof_bnd_flux}

We now construct two functions $\bzeta_p\in \RTN_p(\Ta)$ and $\bzeta\in
\HdvTa$ such that their restrictions $\bzeta_p^i \eq \bzeta_p|_{\elm_i}$ and
$\bzeta^i \eq \bzeta|_{\elm_i}$, for all $1\le i\le |\Ta|$, are defined by
induction as follows:
\begin{enumerate}[(i)]
\item For $i=1$, we define the spaces \bse\ba \tV_p(\elm_1)&\eq\{\tv_p\in
\RTN_p(\elm_1);~ \tv_p\scp\tn_{\sd_1\upb}=r_{\sd_1\upb},
\,\, \Dv\tv_p = r_{\elm_1}\}, \\
\tV(\elm_1)&\eq\{\tv\in \Hdvi{\elm_1};~
\tv\scp\tn_{\sd_1\upb}=r_{\sd_1\upb}, \,\, \Dv\tv = r_{\elm_1}\}, \ea\ese
and consider the following unique minimizers:
\be \bzeta_p^1 \eq
\argmin_{\tv_p\in \tV_p(\elm_1)} \|\tv_p\|_{\elm_1}, \qquad \bzeta^1 \eq
\argmin_{\tv\in \tV(\elm_1)} \|\tv\|_{\elm_1}. \ee
\item For all $1<i\le |\Ta|$, we define the spaces
\bse\ba \tV_p(\elm_i)&\eq\{\tv_p\in \RTN_p(\elm_i);~
\tv_p\scp\tn_{\sd_i\upb}=r_{\sd_i\upb}, \,\, \Dv\tv_p = r_{\elm_i},\,
\tv_p\scp\tn_\sd = -r_\sd+\bzeta_p^j\scp\tn_\sd \, \forall \sd\in \F_i^\sharp\}, \\
\tV(\elm_i)&\eq\{\tv\in \Hdvi{\elm_i};~
\tv\scp\tn_{\sd_i\upb}=r_{\sd_i\upb}, \,\, \Dv\tv = r_{\elm_i},\,
\tv\scp\tn_\sd = -r_\sd+\bzeta_p^j\scp\tn_\sd \, \forall \sd\in
\F_i^\sharp\}, \label{eq_VKi} \ea\ese
where $j=j(i,\sd)$ is the index of the cell sharing $\sd$ with $\elm_i$,
\ie, $\sd=\partial \elm_i\cap\partial \elm_j$. In~\eqref{eq_VKi}, the normal trace is prescribed according to~\eqref{eq:def_vn_K}. Consider the following unique minimizers:
\be \label{eq_elm_min_Hdv} \bzeta_p^i \eq \argmin_{\tv_p\in \tV_p(\elm_i)} \|\tv_p\|_{\elm_i}, \qquad \bzeta^i \eq \argmin_{\tv\in \tV(\elm_i)} \|\tv\|_{\elm_i}. \ee
\end{enumerate}
The same reasoning as in Sections~\ref{sec:flux_nonempty} and~\ref{sec:flux_un} shows that these
minimization problems are well-posed. We can also observe that the
continuous minimizer $\bzeta^i$ of~\eqref{eq_elm_min_Hdv} is given by $- \Gr
\zeta^i$, where $\zeta^i \in H^1(\elm_i)$ is the weak solution to the problem
\batn{2}
    - \Lap \zeta^i & = r_{\elm_i} \qquad & \quad& \mbox{ in } \, \elm_i, \\
    -\Gr \zeta^i \scp \tn_\sd & = -r_\sd+\bzeta_p^j\scp\tn_\sd & & \mbox{ on all } \sd \in
    \F_i^\sharp,\\
    -\Gr \zeta^i \scp \tn_{\sd_i\upb} & = r_{\sd_i\upb} & & \mbox{ on } \sd_i\upb,\\
    \zeta^i|_\sd & = 0 & & \mbox{ on all } \sd \in \F_i^\flat,
\eatn
whereas $\bzeta_p^i$ is its (spectral) mixed finite element approximation in $\RTN_p$.

We will show the following two statements for all $1\le i\le |\Ta|$: \bse\ba
\|\bzeta_p^i\|_{\elm_i} &\lesssim \|\bzeta^i\|_{\elm_i}, \label{eq:DGS_flux} \\
\|\bzeta^i\|_{\elm_i} &\lesssim \|\tv^*\|_{\oma} + \sum_{j<i}
\|\bzeta_p^j\|_{\elm_j}, \label{eq:stab_flux} \ea\ese where the sum
in~\eqref{eq:stab_flux} is void for $i=1$ and where $\tv^*\in \tV(\Ta)$ is
the global minimizer in~\eqref{eq_Hdv_ext}. With these two bounds, we can
conclude the proof as in~Section~\ref{sec:stab_pot}.

\bp[Proof of~\eqref{eq:DGS_flux}] We apply Lemma~\ref{lem_Hdv_simpl} on
$\elm=\elm_i$ with $\FKN=\{\sd_i\upb\}\cup\F_i^\sharp$, the $p$-degree
polynomial prescribing the divergence being $r_{\elm_i}$, and the $p$-degree
polynomials prescribing the normal components being $r_{\sd_i\upb}$ for
$\sd=\sd_i\upb$ and $r_\sd-\bzeta_p^j\scp\tn_\sd$ for all $\sd\in
\F_i^\sharp$ (recall that $\tn_\sd$ points inward $\elm_i$ since
$\sd=\partial \elm_j\cap\partial \elm_i$ with $j<i$). The compatibility
condition for these polynomials on the last element follows by the same reasoning as
in~Section~\ref{sec:flux_nonempty}. \ep

\bp[Proof of~\eqref{eq:stab_flux}] The principle of the proof is the same as
that of~\eqref{eq:stab} in~Section~\ref{sec:stab_pot}, the only salient difference
being that the pullback by the geometric map has to be replaced by the
contravariant Piola transformation. Let us exemplify this modification in the
case where $1<i<|\Ta|$ and $|\F_i^\sharp|=1$, say $\F_i^\sharp=\{\sd\}$. Let
$\elm_j\in\Ta$ be the cell such that $\sd = \pt \elm_i \cap \pt \elm_j$; the
definition of $\F_i^\sharp$ implies that $j<i$. Let $\tT : \elm_j\to \elm_i$
be the (unique) bijective affine map leaving $\sd$ pointwise invariant. Let
$\Jac$ be the (constant) Jacobian matrix of $\tT$ and consider the
transformation $\piola(\tv) = \tA(\tv\circ \tT)$ with $\tA =
\det(\Jac)\Jac^{-1}$. Then $\piola$ is an isomorphism from $\Hdvi{\elm_i}$ to
$\Hdvi{\elm_j}$, and also from $\RTN_p(\elm_i)$ to $\RTN_p(\elm_j)$. Consider
in the cell $\elm_i$ the function
\be \label{eq_def_vi_1}
\tv\eq
\tv^*|_{\elm_i} - \piola^{-1}(\tv^*|_{\elm_j}-\bzeta_p^j),
\ee
and let us prove that $\tv\in \tV(\elm_i)$. It is clear that $\tv\in
\Hdvi{\elm_i}$. Concerning the divergence, we use the property $\Dv
\piola^{-1}(\tv_j) = \det(\Jac)^{-1} (\Dv\tv_j)\circ \tT^{-1}$ in $\elm_i$
for any function $\tv_j$ defined in $\elm_j$. Applying this identity to the
function $\tv_j = \tv^*|_{\elm_j}-\bzeta_p^j$ which is divergence-free (since
$j<i$), we infer that $\Dv\tv = \Dv\tv^*|_{\elm_i} = r_{\elm_i}$. Let us now
consider the normal component of $\tv$. Recalling~\eqref{eq:def_vn_K}
and that $\tn_\sd$ points inward $\elm_i$, we need to prove that
\be \label{eq:2b_proved_Hdv}
(\tv,\Gr\phi)_{K_i} = -(r_{K_i},\phi)_{K_i} + (r_{\sd_i\upb},\phi)_{\sd_i\upb}
+ (r_\sd-\bzeta_p^j\scp\tn_\sd,\phi)_\sd,
\ee
for all $\phi\in H^1(\elm_i)$ such that $\phi|_\sd=0$ for all $\sd\in \F_i^\flat$. Let $\wtphi : \oma \to\RR$ be such that $\wtphi|_{K_i}=\phi$, $\wtphi|_{K_j}=\phi\circ\tT$, and $\wtphi = 0$ otherwise, and observe that $\wtphi\in H^1(\oma)$. Using $\wtphi$ as the test function in~\eqref{eq_norm_jump_def} for the global minimizer $\tv^*$, we infer that
\[
(\tv^*,\Gr\phi)_{K_i} + (\tv^*,\Gr(\phi\circ\tT))_{K_j} = -(r_{K_i},\phi)_{K_i}
-(r_{K_j},\phi\circ\tT)_{K_j} + (r_{\sd_i\upb},\phi)_{\sd_i\upb} + (r_{\sd_j\upb},\phi\circ\tT)_{\sd_j\upb} + (r_\sd,\phi)_\sd.
\]
Considering the term $(\tv,\Gr\phi)_{K_i}$, the definition~\eqref{eq_def_vi_1}, changing variables in the last term in the right-hand side, and employing $\epsilon_\Jac =
\frac{\det(\Jac)}{|\det(\Jac)|}=-1$, we obtain
\ban
(\tv,\Gr\phi)_{K_i} &= (\tv^*,\Gr\phi)_{K_i} - (\piola^{-1}(\tv^*|_{\elm_j}-\bzeta_p^j),\Gr\phi)_{K_i} \\
& = (\tv^*,\Gr\phi)_{K_i} + (\tv^*,\Gr(\phi\circ\tT))_{K_j} - (\bzeta_p^j,\Gr(\phi\circ\tT))_{K_j} \\
& = (\tv^*,\Gr\phi)_{K_i} + (\tv^*,\Gr(\phi\circ\tT))_{K_j} + (r_{K_j},\phi\circ\tT)_{K_j}- (r_{\sd_j\upb},\phi\circ\tT)_{\sd_j\upb} - (\bzeta_p^j\scp\tn_\sd,\phi)_\sd \\
&= -(r_{K_i},\phi)_{K_i} + (r_{\sd_i\upb},\phi)_{\sd_i\upb}
+ (r_\sd-\bzeta_p^j\scp\tn_\sd,\phi)_\sd,
\ean
where we used the Green formula for the term $(\bzeta_p^j,\Gr(\phi\circ\tT))_{K_j}$ and the prescribed properties of $\bzeta_p^j$ together with the above relation satisfied by $\tv^*$. Hence, \eqref{eq:2b_proved_Hdv} holds true, and $\tv \in \tV(\elm_i)$ as announced.

The reasoning is similar when $\F_i^\sharp$ contains two or three interior
faces of $\elm_i$. Let us still briefly discuss the case where $1<i<|\Ta|$
and $|\F_i^\sharp|=2$, say $\F_i^\sharp=\{\sd^1,\sd^2\}$ with $\edg=\sd^1\cap
\sd^2$. We consider the two-color conforming refinement $\T_\edg'$ of the
rotational path $\T_\edg$ around $\edg$ of Lemma~\ref{lem:2color} below. Define in the cell $\elm_i$ the function
\be \label{eq_def_vi_2}
\tv \eq \tv^*|_{\elm_i} + \sum_{\kappa\in\T_\edg'\setminus\{\elm_i\}}
\frac{\epsilon_\kappa}{\epsilon_{\elm_i}} \piola_\kappa^{-1}((\tv^*-\bzeta_p)|_{\kappa}),
\ee
where we use the same notation as in~Section~\ref{sec:stab_pot}, together with the contravariant Piola transformation $\piola_\kappa$ built using the geometric map $\tT_\kappa$ with Jacobian matrix $\Jac_\kappa$. We need to show that $\tv \in \tV(\elm_i)$. It is again clear that $\tv\in \Hdvi{\elm_i}$ and similarly, remarking that the restrictions $(\tv^*-\bzeta_p)|_{\kappa}$ for all $\kappa\in\T_\edg'\setminus\{\elm_i\}$ are divergence-free, we infer that $\Dv\tv = \Dv\tv^*|_{\elm_i} = r_{\elm_i}$. We are thus left to prove that the normal components conditions in~\eqref{eq_VKi} are satisfied. Let $\sd^1 = \partial \elm_i \cap \partial \elm_{j_1}$ and $\sd^2 = \partial \elm_i \cap \partial \elm_{j_2}$, where we remark that $j_1, j_2 < i$. Using~\eqref{eq:def_vn_K}, we need to show that
\be \label{eq:2b_proved_Hdv_2}
(\tv,\Gr\phi)_{K_i} = -(r_{K_i},\phi)_{K_i} + (r_{\sd_i\upb},\phi)_{\sd_i\upb}
+ (r_{\sd^1}-\bzeta_p^{j_1}\scp\tn_{\sd^1},\phi)_{\sd^1} + (r_{\sd^2}-\bzeta_p^{j_2}\scp\tn_{\sd^2},\phi)_{\sd^2},
\ee
for all $\phi\in H^1(\elm_i)$ such that $\phi|_\sd=0$ for all $\sd\in \F_i^\flat$. Let now $\wtphi : \oma \to\RR$ be such that $\wtphi|_{K_i}=\phi$, $\wtphi|_{\kappa}=\phi\circ\tT_\kappa$, $\kappa\in\T_\edg'\setminus\{\elm_i\}$, and $\wtphi = 0$ otherwise. Observe that $\wtphi$ takes the value zero on the faces lying on the boundary of the submesh $\T_\edg'$ and sharing the vertex $\ver$ and is continuous over the faces in $\T_\edg'$ sharing the edge $\edg$. Consequently, $\wtphi\in H^1(\oma)$. Using as above $\wtphi$ as the test function in~\eqref{eq_norm_jump_def}, we see that
\be \label{eq_v*_2} \bs
(\tv^*,\Gr\phi)_{K_i} + \sum_{\kappa\in\T_\edg'\setminus\{\elm_i\}} (\tv^*,\Gr(\phi\circ\tT_\kappa))_{\kappa} = {} & -(r_{K_i},\phi)_{K_i} - \sum_{\kappa\in\T_\edg'\setminus\{\elm_i\}} (r_{\kappa},\phi\circ\tT_\kappa)_{\kappa}
+ (r_{\sd_i\upb},\phi)_{\sd_i\upb}\\
& {} + \sum_{\kappa\in\T_\edg'\setminus\{\elm_i\}} (r_{\sd_\kappa\upb},\phi\circ\tT_\kappa)_{\sd_\kappa\upb} + \sum_{\sd \in \Fe} (r_\sd,\phi \circ\tT_{\kappa_\sd})_\sd,
\es \ee
with the obvious association of $\sd_\kappa\upb$ with $\sd_j\upb$ and where $\kappa_\sd \in \T_\edg'$ is the element sharing $\sd$ having the lowest enumeration index. Employing the definition~\eqref{eq_def_vi_2}, changing variables, noting that $\epsilon_\kappa \epsilon_{\Jac_\kappa}$ is independent of the two cells sharing $\sd \in \Fe$ since $\tn_{\kappa}$ changes orientation, and using that the jumps of $\bzeta_p$ on the faces from $\Fe$ other than $\sd^1$ and $\sd^2$ are given by $r_\sd$ whereas the normal component of $\bzeta_p^j$ has zero jumps inside of the original simplices $\elm_j$ from $\Te$, we deduce
\ban
(\tv,\Gr\phi)_{K_i} = {} & (\tv^*,\Gr\phi)_{K_i} - \sum_{\kappa\in\T_\edg'\setminus\{\elm_i\}}\Big(
\frac{\epsilon_\kappa}{\epsilon_{\elm_i}} \piola_\kappa^{-1}((\tv^*-\bzeta_p)|_{\kappa}),\Gr\phi\Big)_{K_i} \\
 = {} & (\tv^*,\Gr\phi)_{K_i} + \sum_{\kappa\in\T_\edg'\setminus\{\elm_i\}} (\tv^*,\Gr(\phi\circ\tT_\kappa))_{\kappa} - \sum_{\kappa\in\T_\edg'\setminus\{\elm_i\}} (\bzeta_p,\Gr(\phi\circ\tT_\kappa))_{\kappa}\\
 = {} & (\tv^*,\Gr\phi)_{K_i} + \sum_{\kappa\in\T_\edg'\setminus\{\elm_i\}} (\tv^*,\Gr(\phi\circ\tT_\kappa))_{\kappa}
+ \sum_{\kappa\in\T_\edg'\setminus\{\elm_i\}} (r_{\kappa},\phi\circ\tT_\kappa)_{\kappa}\\
{} & - \sum_{\kappa\in\T_\edg'\setminus\{\elm_i\}} (r_{\sd_\kappa\upb},\phi\circ\tT_\kappa)_{\sd_\kappa\upb}
-\sum_{\sd \in \Fe \setminus\{\sd^1, \sd^2\}} (\jump{\bzeta_p}\scp\tn_\sd,\phi\circ\tT_{\kappa_\sd})_\sd\\
{} & - (\bzeta_p^{j_1}\scp\tn_{\sd^1},\phi\circ\tT_{\kappa_{\sd^1}})_{\sd^1} - (\bzeta_p^{j_2}\scp\tn_{\sd^2},\phi\circ\tT_{\kappa_{\sd^2}})_{\sd^2}\\
= {} & -(r_{K_i},\phi)_{K_i} + (r_{\sd_i\upb},\phi)_{\sd_i\upb}
+ (r_{\sd^1}-\bzeta_p^{j_1}\scp\tn_{\sd^1},\phi)_{\sd^1} + (r_{\sd^2}-\bzeta_p^{j_2}\scp\tn_{\sd^2},\phi)_{\sd^2},
\ean
where we have also employed that both $\tn_{\sd^1}$ and $\tn_{\sd^2}$ point inwards $\elm_i$, and, in the last step, \eqref{eq_v*_2}. Thus~\eqref{eq:2b_proved_Hdv_2} holds true. \ep

\section{Proofs for boundary vertices}
\label{sec:boundary_ver}

In this section, we prove Theorems~\ref{thm_H1_ext_gen} and~\ref{thm_Hdv_ext_gen}, by relating the case of a boundary patch to that of an interior patch via geometric piecewise affine mappings and symmetry arguments. Let $\ver$ be a boundary vertex as specified in Section~\ref{sec:boundary_ver_set}. Recall that by assumption, the patch domain $\oma$ only contains an open ball around $\ver$ minus some sector with solid angle $\theta_\ver\in(0,4\pi)$. In what follows, as throughout the paper, we abbreviate as $A\lesssim B$ the inequality $A\le cB$ with a generic constant $c$ whose value can only depend on the patch regularity parameter $\gamma_{\Ta}$.
Moreover, $A \approx B$ stands for simultaneously $A\le c B$ and $B\le c A$.

\subsection{$H^1$ setting}
\label{sec:bnd_ver_H1}

We prove in this section Theorem~\ref{thm_H1_ext_gen}. Recall that we have assumed that either $\FaN=\emptyset$ or $\FaD=\emptyset$. Moreover, either all the faces in $\Fab$ have at least one vertex lying in the interior of $\pt\omab$, see the left panel of Figure~\ref{fig:patch_bnd} for an example, or there are at most two tetrahedra in the patch, see the right panel of Figure~\ref{fig:patch_bnd} for an example.
We show the details under the first assumption and only give brief comments on the case where this does not hold but $|\Ta| \leq 2$.

When all the faces in $\Fab$ have at least one vertex lying in the interior of $\pt\omab$, it is possible to design a mapping $\tT$ of $\Ta$ into a flattened patch $\tTa$ lying in a half-space in $\RR^3$ determined by some plane containing the boundary vertex $\ver$. More precisely, the mapping $\tT$ is defined by means of a collection of bijective affine mappings $\{\tT_\elm\}_{\elm \in \Ta}$ such that (see the left panel of Figure~\ref{fig:patch_bnd_T_S} for an illustration in the context of Figure~\ref{fig:patch_bnd}):
1) $\tT_\elm : \elm \to \telm$, where $\telm$ is a tetrahedron;
2) the tetrahedra $\telm$ form a patch $\tTa$ topologically equivalent to $\Ta$ (i.e., the connectivity between elements, faces, edges, and vertices of $\tTa$ is the same as that in $\Ta$);
3) all the faces $\sd \in \FaD \cup \FaN$ are mapped into a plane $P$ and $\tTa$ lies in a half-space of $\RR^3$ bounded by this plane;
4) the shape-regularity parameter of $\tTa$ is equivalent to that of $\Ta$ up to a constant that only depends on $\gamma_{\Ta}$.

\begin{figure}[htb]
\centerline{\includegraphics[height=0.6\textwidth]{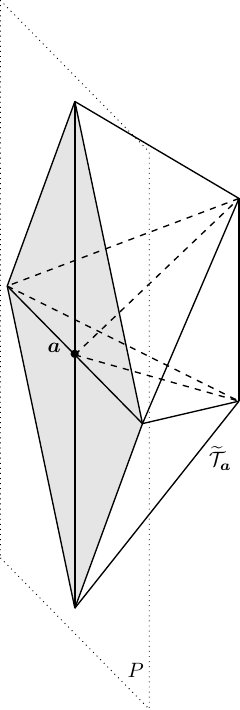} \qquad \qquad \includegraphics[height=0.6\textwidth]{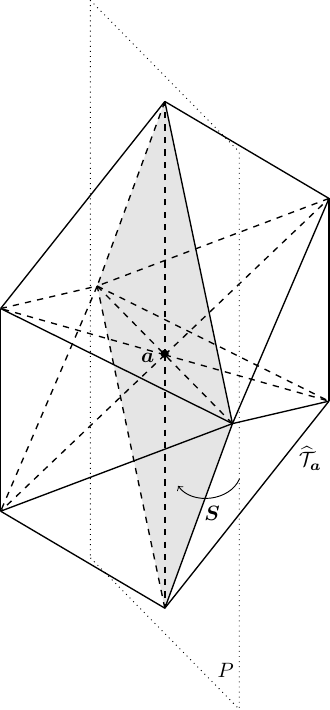}}
\caption{Left: boundary patch $\Ta$ of the left panel of Figure~\ref{fig:patch_bnd} mapped by $\tT$ to the flattened patch $\tTa$ lying in the half-space of $\RR^3$ determined by the plane $P$; right: patch $\tTa$ mapped by the symmetry $\tS$ over the plane $P$ and the resulting interior patch $\hTa$}
\label{fig:patch_bnd_T_S}
\end{figure}

Let us further map the flattened patch $\tTa$ by the symmetry $\tS$ over the plane $P$ into the other half-space of $\RR^3$ to produce together with $\tTa$ a new patch of elements denoted by $\hTa$; see the right panel of Figure~\ref{fig:patch_bnd_T_S} for an illustration.
We keep the orientation for the face normals $\tn_{\tsd}$ of faces $\tsd$ in $\tTa$ the same and let it arbitrary for the newly created faces in $\hTa$.
Then, $\hTa$ is an ``interior'' patch of tetrahedra sharing the vertex $\ver$ as described in Section~\ref{sec_set},
where in particular the patch subdomain $\homa$, defined as the interior of the set $\bigcup_{\helm\in\hTa} \helm$, contains an open ball around $\ver$.

\subsubsection{Case 1: pure Dirichlet conditions} \label{sec_bound_H1_Dir}

Let us first treat the case where $\FaN = \emptyset$, \ie, $\Fa = \Fain \cup \Fab \cup \FaD$ and all the faces located on the patch boundary $\pt \oma$ and sharing the vertex $\ver$ are contained in the (Dirichlet) set $\FaD$.
Denote
\bse \label{eq:spaces_min_pot_bound}\ba
    V_p(\Ta) &\eq \big\{ v_p\in \PP_p(\Ta);~ v_p|_\sd=0\,\, \forall \sd\in\Fab,\, v_p|_\sd = r_\sd\,\,\forall \sd\in\FaD,\,\jump{v_p}_\sd = r_\sd \,\, \forall \sd \in \Fain \big\},\\
    V(\Ta) & \eq \big\{ v\in H^1(\Ta);~ v|_\sd=0\,\, \forall \sd\in\Fab,\, v|_\sd = r_\sd\,\,\forall \sd\in\FaD,\,\jump{v}_\sd = r_\sd \,\, \forall \sd \in \Fain\big\}.
    \ea\ese
Then the stability bound~\eqref{eq_H1_ext_gen} becomes
\be \label{eq:min_pot_bound}
    \min_{v_p \in V_p(\Ta)} \norm{\Grb v_p}_\oma
    \leq C_{\rm st} \min_{v \in V(\Ta)} \norm{\Grb v}_\oma.
\ee
As in Section~\ref{sec:proof_pot}, we denote respectively by $v_p^*\in V_p(\Ta)$ and $v^*\in V(\Ta)$ the minimizers in~\eqref{eq:min_pot_bound}, supposing for the moment that they exist.
Considering on $\tTa$ the equivalents $\widetilde V_p(\tTa)$ and $\widetilde V(\tTa)$ of the spaces $V_p(\Ta)$ and $V(\Ta)$ from~\eqref{eq:spaces_min_pot_bound}, where $\widetilde r_{\widetilde \sd} \eq r_\sd \circ \tT_\elm^{-1}$ for $\sd \in \FK$, $\tsd \in \tFK$, $\telm = \tT_\elm (\elm)$, $\tsd = \tT_\elm (\sd)$,
one readily shows that
\be \label{eq_en_map}
    \norm{\Grbt \widetilde v_p^*}_\toma \approx \norm{\Grb v_p^*}_\oma,\qquad
    \norm{\Grbt \widetilde v^*}_\toma \approx \norm{\Grb v^*}_\oma,
\ee
where $\toma$ is the interior of the set $\bigcup_{\telm\in\tTa} \telm$ and
\be \label{eq:min_pot_bound_T}
    \widetilde v_p^* \eq \arg\min_{\widetilde v_p \in \widetilde V_p(\tTa)} \norm{\Grbt \widetilde v_p}_\toma, \quad
    \widetilde v^* \eq \arg\min_{\widetilde v \in \widetilde V(\tTa)} \norm{\Grbt \widetilde v}_\toma.
\ee
For instance, since $v^*\circ \tT^{-1} \in \widetilde V(\tTa)$, we have
$\norm{\Grbt \widetilde v^*}_\toma \le \norm{\Grbt (v^*\circ \tT^{-1})}_\toma
= \norm{\tJ^{\mathrm{T}}(\Grb v^*)\circ \tT^{-1}}_\toma \le
|\det(\tJ)|^{1/2} \|\tJ\|_{\ell^2} \norm{\Grb v^*}_\oma \lesssim \norm{\Grb v^*}_\oma$,
since the Jacobian matrix $\tJ$ of $\tT^{-1}$ and its determinant $\det(\tJ)$ are
both of order unity.



For each interior face $\hsd \in \hFain$ of the symmetrized patch $\hTa$, \cf\ Figure~\ref{fig:patch_bnd_T_S}, right, consider a polynomial $\widehat r_\hsd \in \PP_{p}(\hsd)$ such that
\bse\label{eq_r_ext_Dir}\bat{2}
    \widehat r_\hsd & = \widetilde r_\tsd \quad & & \text{ if } \hsd = \tsd \in \tFain, \\
    \widehat r_\hsd & = \widetilde r_\tsd & & \text{ if } \hsd = \tsd \in \tFaD, \\
    \widehat r_\hsd & = 0 & & \text{ if } \hsd = \tS (\tsd), \tsd \in \tFain,
\eat\ese
where, in the last line, more precisely, $\tsd \in \tFK$, $\hsd \in \hFK$, $\helm = \tS (\telm)$, $\hsd = \tS (\tsd)$.
Remark that, crucially, $\widehat r_\hsd|_{\hsd \cap \pt \homa} = 0$ on all interior faces $\hsd \in \hFain$ and $\sum_{\hsd \in \hFe} \iota_{\hsd, \hedg} \, \widehat r_\hsd|_\hedg = 0$ on all interior edges $\hedg \in \hEa$. Indeed, the former property follows by~\eqref{eq_conds_H1_1_gen}, whereas the latter property is trivial for edges only present in the extension by $\tS$ (not lying in the patch $\tTa$, to the left of the plane $P$ in the right panel of Figure~\ref{fig:patch_bnd_T_S}) and follows from~\eqref{eq_conds_H1_2_gen} together with our definition in~\eqref{eq_r_ext_Dir} for edges already present in $\tTa$. Thus, by the results of Sections~\ref{sec:nonempty_pot} and~\ref{sec:un_pot} for interior patches, there exist unique minimizers of the problems
\be \label{eq:min_pot_bound_S}
    \widehat v_p^* \eq \arg\min_{\widehat v_p \in \widehat V_p(\hTa)} \norm{\Grbh \widehat v_p}_\homa, \quad
    \widehat v^* \eq \arg\min_{\widehat v \in \widehat V(\hTa)} \norm{\Grbh \widehat v}_\homa,
\ee
where
\bse\ba
    \widehat V_p(\hTa) & \eq \big\{ \widehat v_p\in \PP_p(\hTa);~ \widehat v_p|_\hsd=0\,\, \forall \hsd\in\hFab,\, \jump{\widehat v_p}_\hsd = \widehat r_\hsd \,\, \forall \hsd \in \hFain\big\},\\
    \widehat V(\hTa) &\eq \big\{ \widehat v\in H^1(\hTa);~ \widehat v|_\hsd=0\,\, \forall \hsd\in\hFab,\, \jump{\widehat v}_\hsd = \widehat r_\hsd \,\, \forall \hsd \in \hFain\big\}.
\ea\ese
Moreover, as proved in Section~\ref{sec:stab_pot}, the claim~\eqref{eq_H1_ext} of Theorem~\ref{thm_H1_ext} holds true, \ie,
\be \label{eq:min_pot_bound_ext}
    \norm{\Grbh \widehat v_p^*}_\homa \lesssim \norm{\Grbh \widehat v^*}_\homa.
\ee

We now show that there exist unique minimizers in problems~\eqref{eq:min_pot_bound_T} and, a fortiori, in problems~\eqref{eq:min_pot_bound}. Consider $\widehat v_p^*$ given by~\eqref{eq:min_pot_bound_S} and define on $\toma$
\be \label{eq_H1_map}
    \widetilde v_p \eq \widehat v_p^*|_\toma - \widehat v_p^*|_{\homa \setminus \toma} \circ \tS.
\ee
Immediately, $\widetilde v_p \in \PP_p(\tTa)$; moreover, it is easy to verify that actually $\widetilde v_p \in \widetilde V_p(\tTa)$. Thus, the minimization sets in~\eqref{eq:min_pot_bound_T} are non-empty, and uniqueness of the minimizers follows as in Section~\ref{sec:un_pot}. From~\eqref{eq:min_pot_bound_T} and the definition~\eqref{eq_H1_map} of $\widetilde v_p$, we immediately conclude that
\be \label{eq_1_Dir}
     \norm{\Grbt \widetilde v_p^*}_\toma \leq \norm{\Grbt \widetilde v_p}_\toma \leq \sqrt 2 \norm{\Grbh \widehat v_p^*}_\homa.
\ee
Finally, let us extend the continuous minimizer $\widetilde v^*$ from~\eqref{eq:min_pot_bound_T} from $\toma$ to $\homa$ by zero, which we denote by $E(\widetilde v^*)$. Immediately, we have $E(\widetilde v^*) \in \widehat V(\hTa)$, so that
\be \label{eq_2_Dir}
     \norm{\Grbh \widehat v^*}_\homa \leq \norm{\Grbh E(\widetilde v^*)}_\homa = \norm{\Grbt \widetilde v^*}_\toma.
\ee
Combining~\eqref{eq_1_Dir}, \eqref{eq:min_pot_bound_ext}, and~\eqref{eq_2_Dir} gives
$\norm{\Grbt \widetilde v_p^*}_\toma \lesssim \norm{\Grbt \widetilde v^*}_\toma$,
and~\eqref{eq_en_map} yields the desired stability property~\eqref{eq:min_pot_bound}.

\bigskip

\noindent{\em At most two elements in the patch $\Ta$}

\medskip

\begin{figure}[htb]
\centerline{\includegraphics[height=0.4\textwidth]{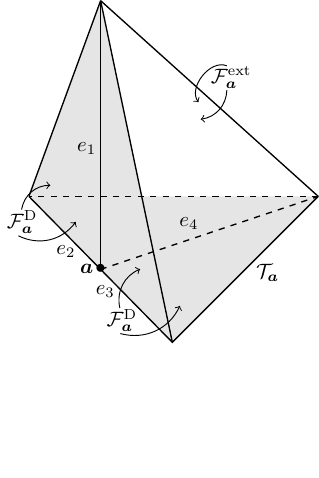} \qquad \qquad \includegraphics[height=0.4\textwidth]{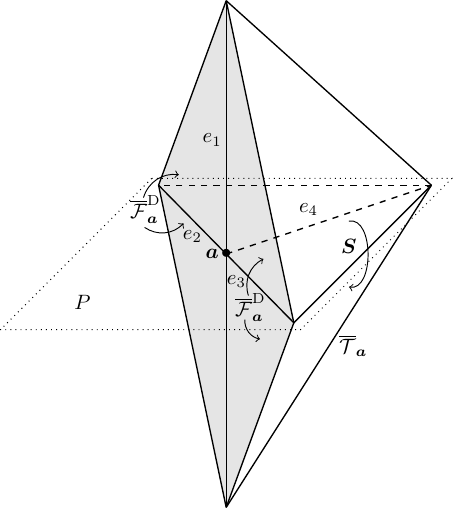}}
\caption{Boundary patch $\Ta$ with $\FaN = \emptyset$, the two faces in $\Fab$ not having a vertex lying inside $\pt\omab$, and four edges $\edg_i$, $1 \leq i \leq 4$ having $\ver$ as vertex and lying in $\pt \omaD$ (left); extended patch $\bTa$ (right)}
\label{fig:patch_bnd_ext}
\end{figure}

Let there be a face in $\Fab$ not having a vertex lying on $\pt\omab$ but let $|\Ta| \leq 2$, as in the left panel of Figure~\ref{fig:patch_bnd_ext}. We now show that~\eqref{eq:min_pot_bound} also holds true in this case. If $|\Ta| = 1$, \ie, the patch consists of a single tetrahedron, then~\eqref{eq:min_pot_bound} holds true in virtue of Lemma~\ref{lem_H1_simpl}. If $|\Ta| = 2$, then there exists an edge $\edg$ between the vertex $\ver$ and one of the vertices in $\Fab$ not lying inside $\pt\omab$ such that all the faces $\sd \in \FaD$ that share $\edg$ either lie in a plane $P$, or the patch $\Ta$ is such that there exits a collection of bijective affine maps $\tT_\elm$ (denoted simply $\tT$) as in the first step above, such that all the faces $\sd \in \FaD$ lie in a plane $P$ after the mapping of $\Ta$ by $\tT$. This condition is satisfied in the situation illustrated in the left panel of Figure~\ref{fig:patch_bnd_ext} for the faces $\edg_1$ and $\edg_4$ but not for $\edg_2$ and $\edg_3$. For simplicity, we keep the notation $\Ta$ for the patch even after a possible mapping by $\tT$, since an equivalent of~\eqref{eq_en_map} holds here as well.

Let us now map $\Ta$ by a collection of bijective affine maps $\tS$ in a symmetry over the plane $P$ into the other half-space of $\RR^3$, as in the second step above, see the right panel of Figure~\ref{fig:patch_bnd_ext} for an illustration. Note that again,
the shape-regularity parameter of $\bTa$ is equivalent to that of $\Ta$ up to a constant that only depends on $\gamma_{\Ta}$. The new patch $\bTa$ is not an interior patch, since otherwise, the condition on the edge $\edg$ would be violated. An illustration is given in the right panel of Figure~\ref{fig:patch_bnd_ext}.

Extend now the polynomial data $r_\sd$ ``by zero'' from $\Ta$ to $\bTa$. More precisely, for each face $\bsd \in \bFa$, consider a polynomial $\overline r_\bsd \in \PP_{p}(\bsd)$ such that
\bse\label{eq_r_ext_Dir_ni_int}\bat{2}
    \overline r_\bsd & = r_\sd \quad & & \text{ if } \bsd = \sd \in \Fain, \\
    \overline r_\bsd & = r_\sd & & \text{ if } \bsd = \sd \in \FaD, \\
    \overline r_\bsd & = 0 & & \text{ if } \bsd \in (\bFain \cup \bFaD) \setminus (\Fain \cup \FaD).
\eat\ese
Here $\bFa = \bFain \cup \bFab \cup \bFaD$ with our usual notation; in particular all faces lying in $\pt \boma$ and sharing the vertex $\ver$ are gathered in the Dirichlet set $\bFaD$.
Again, $\overline r_\bsd|_{\bsd \cap \pt \boma} = 0$ on all interior and Dirichlet faces $\bsd \in \bFain \cup \bFaD$ and $\sum_{\bsd \in \bFe} \iota_{\bsd, \bedg} \, \overline r_\bsd|_\bedg = 0$ on all edges $\bedg \in \bEa$, so that we can apply the result for the case where all the faces in $\bFab$ have at least one vertex lying in the interior of $\pt\bomab$. We then conclude using restriction from $\bTa$ to $\Ta$ by symmetry over $P$ and prolongation from $\Ta$ to $\bTa$ by zero.

\subsubsection{Case 2: pure Neumann conditions} \label{sec_bound_H1_Neu}

We now treat the case where $\FaD = \emptyset$, \ie, $\FaN$ collects all the faces lying on the boundary of the patch $\oma$ and having $\ver$ as vertex. Let
\bse \label{eq:spaces_min_pot_bound_Neu}\ba
    V_p(\Ta) & \eq \big\{ v_p\in \PP_p(\Ta);~ v_p|_\sd=0\,\, \forall \sd\in\Fab,\, \jump{v_p}_\sd = r_\sd \,\, \forall \sd \in \Fain \big\},\\
    V(\Ta) & \eq \big\{ v\in H^1(\Ta);~ v|_\sd=0\,\, \forall \sd\in\Fab,\, \jump{v}_\sd = r_\sd \,\, \forall \sd \in \Fain\big\}.
\ea\ese
Recall from Figure~\ref{fig:patch_bnd_T_S} the flattened patch $\tTa$ and the symmetrized patch $\hTa$.
For each interior face $\hsd \in \hFain$, we now consider a polynomial $\widehat r_\hsd \in \PP_{p}(\hsd)$ such that
\bse\label{eq_r_ext_Neu}\bat{2}
    \widehat r_\hsd & = \widetilde r_\tsd \qquad \qquad & & \text{ if } \hsd = \tsd \in \tFain, \\
    \widehat r_\hsd & = 0 & & \text{ if } \hsd = \tsd \in \tFaN, \\
    \widehat r_\hsd & = \widetilde r_\tsd \circ \tS^{-1} & & \text{ if } \hsd = \tS (\tsd), \tsd \in \tFain,
\eat\ese
where, in the last line, more precisely, $\tsd \in \tFK$, $\hsd \in \hFK$, $\helm = \tS (\telm)$, $\hsd = \tS (\tsd)$.
Remark that, crucially, we have $\widehat r_\hsd|_{\hsd \cap \pt \homa} = 0$ on all interior faces $\hsd \in \hFain$ and $\sum_{\hsd \in \hFe} \iota_{\hsd, \hedg} \, \widehat r_\hsd|_\hedg = 0$ on all interior edges $\hedg \in \hEa$.
Indeed, as in Section~\ref{sec_bound_H1_Dir}, the former property follows from~\eqref{eq_conds_H1_1_gen}, whereas the latter property is here a consequence of the conditions imposed in~\eqref{eq_r_ext_Neu} together with the choice of the orientation of the normals $\tn_{\hsd}$ of the faces $\hsd$ not present in $\tTa$. If the sum $\sum_{\hsd \in \hFe} \iota_{\hsd, \hedg} \, \widehat r_\hsd|_\hedg $ only contains faces either from $\tTa$ or from $\hTa \setminus \tTa$, then it is equal to zero because of~\eqref{eq_conds_H1_2_gen}, whereas if it also contains faces from both $\tTa$ and $\hTa \setminus \tTa$, then the summands from $\tTa$ and $\hTa \setminus \tTa$ cancel out owing to the symmetry of $\tTa$ with $\hTa \setminus \tTa$.

Then the result of Sections~\ref{sec:nonempty_pot} and~\ref{sec:un_pot} can again be used here, so there exist in particular unique minimizers to the problems~\eqref{eq:min_pot_bound_S}, and Theorem~\ref{thm_H1_ext} implies~\eqref{eq:min_pot_bound_ext}. Consider $\widehat v_p^*$ given by~\eqref{eq:min_pot_bound_S} and define
\be \label{eq_H1_map_Neu}
    \widetilde v_p \eq \widehat v_p^*|_\toma,
\ee
(compare with~\eqref{eq_H1_map} in the case of Dirichlet conditions). Recall that $\widehat V_p(\hTa)$ is defined by the mapping $\tS$ from~\eqref{eq:spaces_min_pot_bound_Neu} and that $\widetilde V_p(\tTa)$ is the equivalent of $V_p(\Ta)$ on $\tTa$.
Then $\widetilde v_p \in \widetilde V_p(\tTa)$; in particular, no condition needs to be satisfied on the Neumann faces $\tFaN$.
From~\eqref{eq:min_pot_bound_T} and the definition~\eqref{eq_H1_map_Neu} of $\widetilde v_p$, we immediately conclude that
\be \label{eq_1_Neu}
     \norm{\Grbt \widetilde v_p^*}_\toma \leq \norm{\Grbt \widetilde v_p}_\toma \leq \norm{\Grbh \widehat v_p^*}_\homa.
\ee
Extending the continuous minimizer $\widetilde v^*$ from~\eqref{eq:min_pot_bound_T} from $\toma$ to $\homa$ by symmetry as $E(\widetilde v^*)|_{\homa \setminus \toma} \eq \widetilde v^*|_\toma \circ \tS^{-1}$, we see that $E(\widetilde v^*) \in \widehat V(\hTa)$, so that
\be \label{eq_2_Neu}
     \norm{\Grbh \widehat v^*}_\homa \leq \norm{\Grbh E(\widetilde v^*)}_\homa \leq \sqrt{2} \norm{\Grbt \widetilde v^*}_\toma.
\ee
Combining~\eqref{eq_1_Neu}, \eqref{eq:min_pot_bound_ext}, and~\eqref{eq_2_Neu} gives $\norm{\Grbt \widetilde v_p^*}_\toma \lesssim \norm{\Grbt \widetilde v^*}_\toma$ and~\eqref{eq_en_map} yields the desired stability property~\eqref{eq:min_pot_bound}.

\bigskip

\noindent{\em At most two elements in the patch $\Ta$}

\medskip

If there exists a face in $\Fab$ without a vertex lying inside $\pt\omab$ but $|\Ta| \leq 2$, as in the right panel of Figure~\ref{fig:patch_bnd}, the face jumps $r_\sd$ have to be set to zero on the original Neumann faces and extended by symmetry over $P$ otherwise. We map from $\bTa$ to $\Ta$ by simple restriction and from $\Ta$ to $\bTa$ by symmetry, \cf~Figure~\ref{fig:patch_bnd_ext}.


%

\subsection{$\tH(\dv)$ setting}
\label{sec:bnd_ver_Hdiv}

We sketch here the proof of Theorem~\ref{thm_Hdv_ext_gen}. As in the $H^1$ setting, we present the case where all the faces in $\Fab$ have
at least one vertex lying in the interior of $\pt\omab$, see the left
panel of Figure~\ref{fig:patch_bnd} for an example; the case $|\Ta| \leq 2$, see the right panel of Figure~\ref{fig:patch_bnd}, can be treated as in Section~\ref{sec:bnd_ver_H1}. Contrary to the $H^1$ setting,
we do not make here any assumption on the subsets $\FaD$ and $\FaN$;
for the sake of clarity of exposition, we still distinguish the case with pure Neumann conditions, that with pure Dirichlet conditions, and finally
we treat the general case of mixed Neumann--Dirichlet conditions.
We again rely on the flattened patch $\tTa$ and the
symmetrized patch $\hTa$; see Figure~\ref{fig:patch_bnd_T_S}.

\subsubsection{Case 1: pure Neumann conditions} \label{sec_bound_Hdv_Neu}

Let us first assume that $\FaD = \emptyset$. Denote
\bse \label{eq:spaces_min_flux_bound} \ba \tV_p(\Ta) \eq \{ &\tv_p \in \RTN_p(\Ta);~
    \tv_p \scp \tn_\sd = r_\sd \,\, \forall \sd \in \Fab, \,\tv_p \scp \tn_\sd = r_\sd \,\, \forall \sd \in \FaN, \nn \\
    &\jump{\tv_p}\scp \tn_\sd = r_\sd \,\, \forall \sd \in \Fain, \, \Dvb \tv_p|_\elm = r_\elm \,\, \forall \elm \in \Ta\}, \\
\tV(\Ta) \eq \{ &\tv \in \HdvTa;~
    \tv \scp \tn_\sd = r_\sd \,\, \forall \sd \in \Fab,\,
    \tv \scp \tn_\sd = r_\sd \,\, \forall \sd \in \FaN,\nn \\
    &\jump{\tv}\scp \tn_\sd = r_\sd \,\, \forall \sd \in \Fain, \, \Dvb \tv|_\elm = r_\elm \,\, \forall \elm \in \Ta\}.
\ea \ese
Then the stability bound~\eqref{eq_Hdv_ext_gen} becomes
\be \label{eq:min_flux_bound}
\min_{\tv_p \in
\tV_p(\Ta)} \norm{\tv_p}_\oma
    \leq C_{\rm st} \min_{\tv \in \tV(\Ta)} \norm{\tv}_\oma.
\ee

The data $r_\elm$ and $r_\sd$ are here extended ``by zero'' from the flattened patch $\tTa$ to the symmetrized patch $\hTa$:
\bse\label{eq_r_ext_Dir_flux}\bat{2}
    \widehat r_\hsd & = \widetilde r_\tsd \quad & & \text{ if } \hsd = \tsd \in \tFain \cup \tFab, \\
    \widehat r_\helm & = \widetilde r_\telm & & \text{ if } \helm = \telm \in \tTa,\\
    \widehat r_\hsd & = \widetilde r_\tsd & & \text{ if } \hsd = \tsd \in \tFaN, \\
    \widehat r_\hsd & = 0 & & \text{ if } \hsd = \tS (\tsd), \tsd \in \tFain \cup \tFab,\\
    \widehat r_\helm & = 0 & & \text{ if } \helm = \tS (\telm), \telm \in \tTa,
\eat\ese
the spaces on $\hTa$ become
\bse \label{eq:spaces_min_flux_bound_hTa} \ba \widehat \tV_p(\hTa) \eq \{ &\widehat \tv_p \in \RTN_p(\hTa);~
    \widehat \tv_p \scp \tn_\hsd = \widehat r_\hsd \,\, \forall \hsd \in \hFab, \nn \\
    &\jump{\widehat \tv_p}\scp \tn_\hsd = \widehat r_\hsd \,\, \forall \hsd \in \hFain, \, \Dvb \widehat \tv_p|_\helm = \widehat r_\helm \,\, \forall \helm \in \hTa\}, \\
\widehat \tV(\hTa) \eq \{ &\tv \in \tH(\dv,\hTa);~
    \tv \scp \tn_\hsd = \widehat r_\hsd \,\, \forall \hsd \in \hFab,\nn \\
    &\jump{\tv}\scp \tn_\hsd = \widehat r_\hsd \,\, \forall \hsd \in \hFain, \, \Dvb \tv|_\helm = \widehat r_\helm \,\, \forall \helm \in \hTa\},
\ea \ese
and the stability property in $\hTa$ becomes
\be \label{eq:min_flux_bound_hTa}
\min_{\widehat \tv_p \in
\widehat \tV_p(\hTa)} \norm{\widehat \tv_p}_\homa
    \leq C_{\rm st} \min_{\widehat \tv \in \widehat \tV(\hTa)} \norm{\widehat \tv}_\homa.
\ee
Owing to~\eqref{eq_comp_gen} and~\eqref{eq_r_ext_Dir_flux}, we find that the compatibility condition~\eqref{eq_comp} in $\widehat\Ta$, \ie,
\be \label{eq_comp_gen_bnd}
    \sum_{\helm \in \hTa} (\widehat r_\helm, 1)_\helm - \sum_{\hsd \in
    \hFa} (\widehat r_\hsd, 1)_\hsd = 0
\ee
holds true. The last two ingredients of the proof are the restriction of the discrete minimizer $\widehat \tv_p^*$ of~\eqref{eq:min_flux_bound_hTa}, left, from $\hTa$ to $\tTa$ given by
\be \label{eq_ext_Hdv_Neu}
    \widetilde \tv_p \eq \widehat \tv_p^*|_\toma - \piola (\widehat \tv_p^*|_{\homa \setminus \toma}),
\ee
where the contravariant Piola transformation $\piola$ is built using the geometric mapping $\tS$, and the extension of the continuous minimizer $\widetilde \tv^*$ on $\tTa$ from $\tTa$ to $\hTa$ by zero.


\subsubsection{Case 2: pure Dirichlet conditions} \label{sec_bound_Hdv_Dir}

The case $\FaN = \emptyset$ can be treated as above.
After moving from $\Ta$ to $\tTa$ in the first step,
the data are extended ``by symmetry'' from $\tTa$ to $\hTa$, similarly to~\eqref{eq_r_ext_Neu}. This ensures the compatibility condition~\eqref{eq_comp} on the ``extended interior'' patch $\hTa$, even though there has been no such a condition on the original patch $\Ta$. The restriction of the discrete minimizer $\widehat \tv_p^*$ of~\eqref{eq:min_flux_bound_hTa}, left, from $\hTa$ to $\tTa$ is given by $\widetilde \tv_p \eq \widehat \tv_p^*|_\toma$, whereas the extension of the continuous minimizer $\widetilde \tv^*$ on $\tTa$ from $\tTa$ to $\hTa$ is obtained by symmetry as $E(\widetilde \tv^*)|_{\homa \setminus \toma} \eq \piola^{-1}(\widetilde \tv^*|_\toma)$.

\subsubsection{Case 3: mixed Neumann--Dirichlet conditions} \label{sec_bound_Hdv_NeuDir}

The proof proceeds again as above.
The main change is that in~\eqref{eq_r_ext_Dir_flux}, we also need to impose $\widehat r_\hsd$ for all $\hsd = \tsd \in \tFaD$: we can choose arbitrary values for all Dirichlet faces except for one, where $\widehat r_\hsd$ is chosen such that~\eqref{eq_comp_gen_bnd} holds true. As in Section~\ref{sec_bound_Hdv_Neu}, the restriction of the discrete minimizer $\widehat \tv_p^*$ is given by~\eqref{eq_ext_Hdv_Neu}; note that all the Neumann conditions prescribed on the faces from $\tFaN$ (or $\FaN$) are satisfied, whereas no conditions need to be satisfied on the faces from $\tFaD$ (or $\FaD$). Finally, the continuous minimizer $\widetilde \tv^*$ on $\tTa$ is extended from $\tTa$ to $\hTa$ by zero.

\paragraph*{Acknowledgments.} The authors are grateful to F. Meunier (CERMICS) for insightful discussions on the results of Appendix~\ref{app_graphs}.

\appendix

\section{Stable polynomial extensions on a tetrahedron}
\label{app_ext_tetra}

In this section, we reformulate and extend some recent results by Demkowicz, Gopolakrishnan, and Sch\"oberl \cite{Demk_Gop_Sch_ext_I_09,Demk_Gop_Sch_ext_III_12} and by Costabel and McIntosh \cite{Cost_McInt_Bog_Poinc_10} on stable polynomial extensions on a single tetrahedron. We first consider $H^1$-stable extensions in the polynomial space $\PP_p(\elm)$.

\bl[$H^1$-stable polynomial extension on a tetrahedron]\label{lem_H1_simpl}
Let $\elm$ be a tetrahedron with $\FKD \subset \FK$ a possibly empty
subset of its faces. Let $r$ be a $p$-degree piecewise polynomial on $\FKD$,
$p \geq 1$, with restriction to each $\sd \in \FKD$ denoted by $r_\sd$.
Assume that $r$ is globally continuous over the Dirichlet boundary given by $\FKD$. Then there exists a
constant $C > 0$ only depending on the shape-regularity parameter
$\gamma_\elm$ of $\elm$ such that
\be \label{eq_H1_ext_simpl}
    \min_{\substack{v_p \in \PP_p(\elm)\\
    v_p|_\sd = r_\sd \,\,\forall \sd \in \FKD}} \norm{\Gr v_p}_\elm
    \leq C \min_{\substack{v \in \Hoi{\elm}\\
    v|_\sd = r_\sd \,\,\forall \sd \in \FKD}} \norm{\Gr v}_\elm.
\ee
\el

\br[Equivalent form] \label{rem_eq_H1} Consider the following problem: find
$\zeta_\elm$ such that
\batn{2}
    - \Lap \zeta_\elm & = 0 \qquad & & \mbox{ in } \, \elm, \\
    \zeta_\elm|_\sd & = r_\sd & & \mbox{ on all } \sd \in \FKD, \\
    -\Gr \zeta_\elm \scp \tn_\elm|_\sd & = 0 & & \mbox{ on all } \sd \in \FK \setminus \FKD,
\eatn
\ie, in weak form, $\zeta_\elm \in \Hoi{\elm}$ is such that
$\zeta_\elm|_\sd = r_\sd$ for all $\sd \in \FKD$ and
\[
    (\Gr \zeta_\elm, \Gr v)_\elm = 0 \qquad \forall v \in \Hoi{\elm}, \, v|_\sd =
    0 \,\, \forall \sd \in \FKD.
\]
Similarly, the (spectral) finite element method of order $p$ finds
$\zeta_{p,\elm} \in \PP_p(\elm)$ with $\zeta_{p,\elm}|_\sd = r_\sd$ for all
$\sd \in \FKD$ such that
\[
    (\Gr \zeta_{p,\elm}, \Gr v_p)_\elm = 0 \qquad \forall v_p \in \PP_p(\elm), \, v_p|_\sd =
    0 \mbox{ for all } \sd \in \FKD.
\]
As $\zeta_{p,\elm}$ and $\zeta_\elm$ are, respectively, the (unique) minimizers
from~\eqref{eq_H1_ext_simpl}, Lemma~\ref{lem_H1_simpl} can be rephrased as a
stability result for (spectral) finite elements on a single tetrahedron, \ie,
$\|\Gr \zeta_\elm\|_\elm \le \|\Gr \zeta_{p,\elm}\|_\elm \le
C\|\Gr\zeta_\elm\|_\elm$. \er

\bp
The result of Lemma~\ref{lem_H1_simpl} follows from the results
and proofs in \cite{Demk_Gop_Sch_ext_I_09}. For completeness, we give a proof using the present notation. \\
(1) Let us start by noting that in the case $\FKD = \emptyset$, both $\zeta_{p,\elm}$ and $\zeta_{\elm}$ can be taken as zero. Let us henceforth suppose that $\FKD \neq \emptyset$. \\
(2) We first establish~\eqref{eq_H1_ext_simpl} on the unit tetrahedron, say $\wK$; to this purpose, we proceed in three substeps.
\\
(2.a) Case $\elm=\wK$ and $\FKD = \FK$, \ie, the Dirichlet condition is prescribed on the
whole boundary $\pt \elm$. Then~\cite[Theorem~6.1]{Demk_Gop_Sch_ext_I_09}
shows that there exists a polynomial $\zeta_p(r) \in \PP_p(\elm)$ such that
$\zeta_p(r)|_{\pt \elm} = r$ and $\|\zeta_p(r)\|_{H^1(\elm)} \le
C_{\mathrm{DGS}}\norm{r}_{H^\ft(\pt \elm)}$, where $\norm{r}_{H^\ft(\pt
\elm)}$ is defined using the Sobolev--Slobodeckij norm as follows:
\[
\norm{r}_{H^\ft(\pt \elm)}^2 = \norm{r}_{L^2(\pt \elm)}^2 + \int_{\pt\elm}\int_{\pt\elm}
\frac{|r(\tx)-r(\ty)|^2}{\|\tx-\ty\|_{\ell^2}^d}\dx\dy,
\]
recalling that $d=3$. Since there exist constants $\underline{C}_{H^{\ft}}$
and $C_{H^{\ft}}$ of order unity so that
\be \label{eq:norme_demi}
    \underline{C}_{H^{\ft}} \min_{\substack{v\in H^1(\elm) \\
    v|_{\pt \elm} = r}} \norm{v}_{H^1(\elm)} \leq \norm{r}_{H^\ft(\pt \elm)} \le C_{H^{\ft}} \min_{\substack{v\in H^1(\elm) \\
    v|_{\pt \elm} = r}} \norm{v}_{H^1(\elm)},
\ee
and since $\|\zeta_p(r)\|_{H^1(\elm)}\geq \min \norm{v_p}_{H^1(\elm)}$ over
all $v_p\in \PP_p(\elm)$ such that $v_p|_{\pt
\elm} = r$, we infer that
\be\label{eq:stab_K_pot}
    \min_{\substack{v_p \in \PP_p(\elm)\\
    v_p|_\sd = r_\sd \,\,\forall \sd \in \FKD}} \norm{v_p}_{H^1(\elm)}
    \leq \wC \min_{\substack{v \in \Hoi{\elm}\\
    v|_\sd = r_\sd \,\,\forall \sd \in \FKD}} \norm{v}_{H^1(\elm)},
\ee
with $\wC=C_{\mathrm{DGS}}C_{H^{\ft}}$ and $\norm{v}_{H^1(\elm)}^2 = \norm{v}_\elm^2 + \norm{\Gr v}_\elm^2$.
Note that we are using the $H^1$-norm in~\eqref{eq:stab_K_pot}, and not the $H^1$-seminorm as in~\eqref{eq_H1_ext_simpl}.
\\
(2.b) Case $\elm=\wK$ and $\FKD \neq \FK$. Let us prove again~\eqref{eq:stab_K_pot}. Let us set
\be \label{eq:tilde_zeta_K}
\tilde\zeta_\elm \eq \argmin_{\substack{v\in H^1(\elm)\\
v|_\sd = r_\sd\,\,\forall\sd \in \FKD}} \|v\|_{H^1(\elm)}.
\ee
Note that $\tilde\zeta_\elm$ solves in strong form $-\Lap\tilde\zeta_\elm+
\tilde\zeta_\elm=0$ in $\elm$, $\tilde\zeta_\elm|_\sd=r|_\sd$ for all
$\sd\in\FKD$, and $-\Gr \tilde \zeta_\elm \scp \tn_\elm|_\sd=0$ for all
$\sd\in\FK\setminus\FKD$; note also that $\tilde\zeta_\elm$ is well defined since
$\FKD$ is assumed to be non-empty. Define a new function $\tilde r \in
H^\ft(\pt \elm)$ as the trace of $\tilde\zeta_\elm$ on $\pt \elm$; this
extends the boundary data $r$ originally defined only on the faces in $\FKD$ to
the whole $\pt \elm$, not necessarily by polynomials on the faces in $\FK
\setminus \FKD$. The definition~\eqref{eq:tilde_zeta_K} of
$\tilde\zeta_\elm$ combined with~\eqref{eq:norme_demi}
yields $\norm{\tilde r}_{H^\ft(\pt \elm)}\le
C_{H^{\ft}}\|\tilde\zeta_\elm\|_{H^1(\elm)}$. Let us now order the faces of
$\elm$ with the faces in $\FKD$ first and consider only the summands
corresponding to the faces in $\FKD$ instead of the full extension operator
of~\cite[equation~(6.1)]{Demk_Gop_Sch_ext_I_09}, applied to the function
$\tilde r$. Following~\cite[Theorem~6.1]{Demk_Gop_Sch_ext_I_09}, we obtain a
polynomial $\zeta_{p}(\tilde r) \in \PP_p(\elm)$ such that $\zeta_{p}(\tilde
r)|_\sd = r_\sd$ for all $\sd \in \FKD$ and $\norm{\zeta_{p}(\tilde
r)}_{H^1(\elm)} \leq C_{\mathrm{DGS}} \norm{\tilde r}_{H^\ft(\pt \elm)}$.
Combining with the above bound on $\norm{\tilde r}_{H^\ft(\pt \elm)}$ and
since $\|\zeta_p(\tilde r)\|_{H^1(\elm)}\geq \min \norm{v_p}_{H^1(\elm)}$
over all $v_p\in \PP_p(\elm)$ such that $v_p|_\sd = r_\sd$ on all
$\sd\in\FKD$, we infer that \eqref{eq:stab_K_pot} also holds true if $\FKD \neq
\FK$ with $\wC=C_{\mathrm{DGS}}C_{H^{\ft}}$. This completes the proof of~\eqref{eq:stab_K_pot}.\\
(2.c) Let us prove that~\eqref{eq_H1_ext_simpl} holds true when $\elm=\wK$. Let $c\in\RR$ be arbitrary and let us set $r' \eq r+c$; note that $r'$
is also a $p$-degree piecewise polynomial on $\FKD$ that is globally
continuous over $\FKD$. Since $\elm$ is the unit tetrahedron
and applying~\eqref{eq:stab_K_pot} with the datum $r'$, we infer that
\ban
\min_{\substack{v_p \in \PP_p(\elm)\\
    v_p|_\sd = r_\sd \,\,\forall \sd \in \FKD}} \norm{\Gr v_p}_{\elm}
&\le
\min_{\substack{v_p \in \PP_p(\elm)\\
    v_p|_\sd = r_\sd \,\,\forall \sd \in \FKD}} \norm{v_p+c}_{H^1(\elm)}
= \min_{\substack{v_p \in \PP_p(\elm)\\
    v_p|_\sd = r'_\sd \,\,\forall \sd \in \FKD}} \norm{v_p}_{H^1(\elm)}
\\
&\le C \min_{\substack{v \in \Hoi{\elm}\\
    v|_\sd = r'_\sd \,\,\forall \sd \in \FKD}} \norm{v}_{H^1(\elm)}
= C
\min_{\substack{v \in \Hoi{\elm}\\
    v|_\sd = r_\sd \,\,\forall \sd \in \FKD}} \norm{v+c}_{H^1(\elm)},
\ean
where the first bound follows by dropping the $L^2$-norm of $v_p+c$. Taking
the infimum over $c\in\RR$ on the right-hand side and using the Poincar\'e
inequality $\inf_{c\in\RR}\|v+c\|_{\elm}\le \frac{1}{\pi} h_\elm \|\nabla v\|_\elm \le c\|\nabla v\|_\elm$ with a constant $c$ of order unity, we infer that
\eqref{eq_H1_ext_simpl} holds true when $\elm=\wK$.
\\
(3) Finally, we use a scaling argument to prove~\eqref{eq_H1_ext_simpl} in any tetrahedron $\elm$. Let $\zeta_{p,\elm}$ and $\zeta_\elm$ be the two minimizers in~\eqref{eq_H1_ext_simpl} (see Remark~\ref{rem_eq_H1}). Let $\tT$ be the affine geometric map from the unit tetrahedron $\wK$ to $K$. Then the pullback by $\tT$ defined as $\psi(v)=v\circ \tT$ is an isomorphism from $H^1(\elm)$ to $H^1(\wK)$ and also from $\PP_p(\elm)$ to $\PP_p(\wK)$. Moreover, we have $\norm{\Gr(\psi(v))}_{\wK} \le C_\psi \norm{\Gr v}_\elm$ and $\norm{\Gr (\psi^{-1}(\widehat v))}_\elm \le C_{\psi^{-1}}\norm{\Gr \widehat v}_{\wK}$ with constants such that $C_\psi C_{\psi^{-1}}$ is uniformly bounded by the shape-regularity parameter of $\elm$.
Let now $\FwKD\eq\{\wF\in \FwK;\,
\tT(\wF)\in \FKD\}$ and let us introduce the piecewise polynomial $\wr$ such that $\wr_{\wF} = r\circ (\tT|_{\wF})$ for all
$\wF\in\FwKD$. Applying the result of Step~(2c) to $\wK$ with the polynomial data
$\wr$ and the subset $\FwKD$, and introducing the two corresponding minimizers, $\widehat\zeta_{p,\wK}$ and $\widehat \zeta_{\wK}$, we infer that
$\norm{\Gr \widehat\zeta_{p,\wK}}_{\wK}
\leq \wC \norm{\Gr \widehat \zeta_{\wK}}_{\wK}$
with $\wC$ of order unity. Finally, we have
\begin{align*}
\norm{\Gr \zeta_{p,\elm}}_{\elm} &\le \norm{\Gr (\psi^{-1}(\widehat\zeta_{p,\wK}))}_{\elm}
\le C_{\psi^{-1}}\norm{\Gr \widehat\zeta_{p,\wK}}_{\wK}
\le C_{\psi^{-1}}\wC \norm{\Gr \widehat \zeta_{\wK}}_{\wK} \\
&\le C_{\psi^{-1}}\wC \norm{\Gr \psi(\zeta_{\elm})}_{\wK}
\le C_\psi C_{\psi^{-1}}\wC \norm{\Gr \zeta_{\elm}}_{\elm},
\end{align*}
since $\psi^{-1}(\widehat\zeta_{p,\wK})$ is in the minimization set defining $\zeta_{p,\elm}$ and $\psi(\zeta_{\elm})$ is in that defining $\widehat \zeta_{\wK}$.
\ep

Let us now consider $\tH(\dv)$-stable extensions in the polynomial space
$\RTN_p(\elm)$. The following lemma rephrases the first two steps of the
proof of~\cite[Theorem~7]{Brae_Pill_Sch_p_rob_09}, while merging them together and extending them to three space dimensions. Recall that the normal trace of a field in $\Hdvi{\elm}$ is prescribed according to~\eqref{eq:def_vn_K}.

\bl[$\tH(\dv)$-stable polynomial extension on a
tetrahedron]\label{lem_Hdv_simpl} Let $\elm$ be a tetrahedron with unit
outward normal $\tn_\elm$. Let $\FKN \subset \FK$ be a possibly empty subset
of its faces. Let $r$ be a $p$-degree piecewise polynomial on $\FKN$, $p
\geq 0$, with restriction to each $\sd\in\FKN$ denoted by $r_\sd$. Let
$r_\elm$ be a $p$-degree polynomial in $\elm$. If $\FKN = \FK$, assume that
$\sum_{\sd \in \FK}(r_\sd,1)_\sd = (r_\elm,1)_\elm$. Then there exists a
constant $C > 0$ only depending on the shape-regularity parameter
$\gamma_\elm$ such that
\be \label{eq_Hdv_ext_simpl}
    \min_{\substack{\tv_p \in \RTN_p(\elm)\\
    \tv_p \scp \tn_\elm|_\sd = r_\sd \,\, \forall \sd \in \FKN\\
    \Dv \tv_p = r_\elm}}
    \norm{\tv_p}_\elm \leq C
    \min_{\substack{\tv \in \Hdvi{\elm}\\
    \tv \scp \tn_\elm|_\sd = r_\sd \,\, \forall \sd \in \FKN\\
    \Dv \tv = r_\elm}}
    \norm{\tv}_\elm.
\ee
\el

\br[Equivalent form] \label{rem_eq_Hdv} Consider the following problem: find
$\zeta_\elm$ such that
\begin{subequations} \label{eq:strong_form_zeta} \begin{alignat}{2}
    - \Lap \zeta_\elm & = r_\elm \qquad & & \mbox{ in } \, \elm,\\
    \zeta_\elm|_\sd & = 0 & & \mbox{ on all } \sd \in \FK \setminus \FKN, \label{eq_Lap_K_Hdv}\\
    -\Gr \zeta_\elm \scp \tn_\elm|_\sd & = r_\sd & & \mbox{ on all } \sd \in \FKN,
\end{alignat} \end{subequations}
\ie, in weak form, $\zeta_\elm\in H^1(\elm)$ is such that $\zeta_\elm|_\sd = 0$
on all $\sd \in \FK \setminus \FKN$, $(\zeta_\elm,1)_\elm=0$ if $\FKN=\FK$,
and
\be \label{eq_primal_Hdv_K}
    (\Gr \zeta_\elm,\Gr \phi)_\elm = (r_\elm,\phi)_\elm - \sum_{\sd\in\FKN} (r_\sd,\phi)_\sd
    \qquad \forall \phi\in H^1(\elm) \text{ with } \phi|_\sd = 0 \,\, \forall \sd \in \FK
    \setminus \FKN.
\ee
The dual weak formulation looks for $\bxi_\elm \in \Hdvi{\elm}$ with $\Dv
\bxi_\elm = r_\elm$ and $\bxi_\elm \scp \tn_\elm|_\sd = r_\sd$ on all $\sd
\in \FKN$ such that
\[
    (\bxi_\elm, \tv)_\elm = 0 \qquad \forall \tv \in \Hdvi{\elm} \text{ with $\Dv \tv = 0$ and $\tv \scp \tn_\elm|_\sd = 0$, $\forall \sd \in \FKN$},
\]
and it is well-known that \be \label{eq:equiv_prim_dual}
\bxi_\elm=\argmin_{\substack{\tv \in \Hdvi{\elm}\\
    \tv \scp \tn_\elm|_\sd = r_\sd \,\, \forall \sd \in \FKN\\
    \Dv \tv = r_\elm}} \|\tv\|_\elm = - \Gr \zeta_\elm.
\ee Similarly, the dual (or, equivalently, (dual) mixed) finite element method
(here, a dual spectral method) finds $\bxi_{p,\elm} \in \RTN_p(\elm)$ with
$\Dv \bxi_{p,\elm} = r_\elm$ and $\bxi_{p,\elm} \scp \tn_\elm|_\sd = r_\sd$
on all $\sd \in \FKN$ such that
\[
    (\bxi_{p,\elm}, \tv_p)_\elm = 0 \qquad \forall \tv_p \in \RTN_p(\elm) \text{ with $\Dv \tv_p = 0$ and $\tv_p \scp \tn_\elm|_\sd = 0$, $\forall \sd \in \FKN$}.
\]
As $\bxi_{p,\elm}$ and $\bxi_\elm$ are, respectively, the unique minimizers
from~\eqref{eq_Hdv_ext_simpl}, Lemma~\ref{lem_Hdv_simpl} can be rephrased as
a stability result for mixed finite elements on a single tetrahedron, \ie,
$\|\bxi_\elm\|_\elm\le\|\bxi_{p,\elm}\|_\elm\le C\|\bxi_\elm\|_\elm$. \er

\bp We first establish~\eqref{eq_Hdv_ext_simpl} on the unit tetrahedron, say
$\wK$; to this purpose, we proceed in three steps. Then, we
establish~\eqref{eq_Hdv_ext_simpl} on any tetrahedron by using the
contravariant Piola transformation.
\\
(1) Case $\elm=\wK$ and $\FKN = \emptyset$. We infer from
\cite[Corollary~3.4]{Cost_McInt_Bog_Poinc_10} that there is $\bxi_p(r_\elm)
\in \RTN_p(\elm)$ such that $\Dv \bxi_p(r_\elm) = r_\elm$ and
$\|\bxi_p(r_\elm)\|_\elm \le C_{\mathrm{CM}}\norm{r_\elm}_{H^{-1}(\elm)}$
where $\norm{r_\elm}_{H^{-1}(\elm)} \eq \max (r_\elm, \phi)_\elm$ over all
$\phi \in \Hooi{\elm}$ such that $\norm{\Gr \phi}_\elm = 1$. Furthermore,
since a Dirichlet boundary condition is prescribed over the whole boundary of
$\elm$ in~\eqref{eq:strong_form_zeta} when $\FKN = \emptyset$, we infer that $\zeta_\elm \in \Hooi{\elm}$ is such
that $(\Gr \zeta_\elm, \Gr \phi)_\elm = (r_\elm, \phi)_\elm$ for all $\phi
\in \Hooi{\elm}$. Then, we have by~\eqref{eq:equiv_prim_dual},
\[
    \min_{\substack{\tv \in \Hdvi{\elm}\\
    \Dv \tv = r_\elm}}
    \norm{\tv}_\elm = \norm{\bxi_\elm}_\elm = \norm{\Gr \zeta_\elm}_\elm =
    \max_{\substack{\phi \in \Hooi{\elm}\\
    \norm{\Gr \phi}_\elm = 1}} (r_\elm, \phi)_\elm = \norm{r_\elm}_{H^{-1}(\elm)}.
\]
Altogether,
\[
    \min_{\substack{\tv_p \in \RTN_p(\elm)\\
    \Dv \tv_p = r_\elm}}
    \norm{\tv_p}_\elm \leq
\norm{\bxi_p(r_\elm)}_\elm \le  C_{\mathrm{CM}}\norm{r_\elm}_{H^{-1}(\elm)}
= C_{\mathrm{CM}} \min_{\substack{\tv \in \Hdvi{\elm}\\
    \Dv \tv = r_\elm}}
    \norm{\tv}_\elm.
\]
(2) Case $\elm=\wK$, $\FKN \ne\emptyset$, and $r_\elm = 0$. We further
distinguish two cases.
\\
(2.a) Assume first that $\FKN=\FK$. Since a Neumann boundary condition is
prescribed over the whole boundary of $\elm$ in~\eqref{eq:strong_form_zeta},
$\zeta_\elm\in H^1(\elm)$ is such that $(\zeta_\elm,1)_\elm=0$ and $(\Gr
\zeta_\elm,\Gr \phi)_\elm = (r,\phi)_{\pt \elm}$ for all $\phi\in H^1(\elm)$. Since $(r,1)_{\pt \elm}=0$ by assumption,
\cite[Theorem~7.1]{Demk_Gop_Sch_ext_III_12} shows that there is $\bxi_p(r)
\in \RTN_p(\elm)$ (actually, in $[\PP_p(\elm)]^{3})$ such that
$\Dv\bxi_p(r)=0$, $\bxi_p(r)\scp \tn_\elm|_\sd=r_\sd$ for all $\sd\in\FK$,
and $\|\bxi_p(r)\|_{\Hdvi{\elm}} = \|\bxi_p(r)\|_\elm \le
C_{\mathrm{DGS}}\|r\|_{H^\mft(\pt \elm)}$. Here, $H^\mft(\pt \elm)$ is the
dual space of $H^\ft(\pt \elm)$ equipped with the norm
$\|r\|_{H^\mft(\pt \elm)} \eq \max (r,\phi)_{\pt \elm}$ over all $\phi \in
H^\ft(\pt \elm)$ such that $\norm{\phi}_{H^\ft(\pt \elm)} = 1$. On the unit
tetrahedron, we can use the lower bound in~\eqref{eq:norme_demi} and the
Poincar\'e inequality in the space $\{\phi\in H^1(\elm);\, (\phi,1)_\elm=0\}$ to infer that there is a constant $C_{H^{-\ft}}$ of
order unity so that
\[
\|r\|_{H^\mft(\pt \elm)} \le C_{H^{-\ft}} \sup_{\substack{\phi\in H^1(\elm) \\ (\phi,1)_\elm=0}}
\frac{(r,\phi)_{\pt \elm}}{\|\Gr \phi\|_\elm}.
\]
Owing to~\eqref{eq:equiv_prim_dual}, we infer that
\[
\|r\|_{H^\mft(\pt \elm)} \le C_{H^{-\ft}} \|\Gr \zeta_\elm\|_\elm
= C_{H^{-\ft}} \min_{\substack{\tv \in \Hdvi{\elm}\\
    \tv \scp \tn_\elm|_\sd = r_\sd \,\, \forall \sd \in \FK\\
    \Dv \tv = 0}}
    \norm{\tv}_\elm.
\]
Altogether,
\[
\min_{\substack{\tv_p \in \RTN_p(\elm)\\
    \tv_p \scp \tn_\elm|_\sd = r_\sd \,\, \forall \sd \in \FK\\
    \Dv \tv_p = 0}}
    \norm{\tv_p}_\elm
\le \|\bxi_p(r)\|_\elm \le C_{\mathrm{DGS}}\|r\|_{H^\mft(\pt \elm)}
\le C_{\mathrm{DGS}}C_{H^{-\ft}}\min_{\substack{\tv \in \Hdvi{\elm}\\
    \tv \scp \tn_\elm|_\sd = r_\sd \,\, \forall \sd \in \FK\\
    \Dv \tv = 0}}
    \norm{\tv}_\elm.
\]
(2.b) Assume now that $\emptyset \neq \FKN \subsetneq \FK$. Let us set
\[
\tilde\zeta_\elm \eq \argmin_{\substack{v\in H^1(\elm)\\
v|_\sd =0 \,\, \forall \sd\in\FK\setminus\FKN}} \Bigg\{\frac 1 2 \|\Gr v\|_\elm^2 + \sum_{\sd\in\FKN} (r_\sd,v)_\sd \Bigg\},
\]
\ie, in weak form $(\Gr\tilde\zeta_\elm,\Gr\phi)_\elm = - \sum_{\sd\in\FKN}
(r_\sd,\phi)_\sd$ for all $\phi\in H^1(\elm)$ such that $\phi|_\sd=0$ for all
$\sd\in\FK\setminus\FKN$. Since $\elm$ is a convex polyhedron, elliptic
regularity implies that $\tilde\zeta_\elm \in H^2(\elm)$, so that the normal
derivative $- \Gr \tilde\zeta_\elm\scp\tn_\elm$ can be given a pointwise
meaning on $\pt \elm$. Let us call $\tilde r$ this normal derivative (we could have also considered $\tilde\bzeta_\elm\scp\tn_\elm$, where $\tilde\bzeta_\elm$ is the $\arg$ $\min$ of $\norm{\tv}_\elm$ over all $\tv \in \Hdvi{\elm}$ such that $\tv \scp \tn_\elm|_\sd = r_\sd$ for all $\sd \in \FKN$ and such that $\Dv \tv = 0$). We
infer that $\norm{\tilde r}_{H^{-\ft}(\pt \elm)} \le
C\|\Gr\tilde\zeta_\elm\|_\elm$. Let us now order the faces in $\FKN$ first
and consider only the summands corresponding to the faces from $\FKN$ instead
of the full extension operator
of~\cite[equation~(7.1)]{Demk_Gop_Sch_ext_III_12}, applied to the function
$\tilde r$. Following~\cite[Theorem~7.1]{Demk_Gop_Sch_ext_III_12}, we obtain
a polynomial $\bzeta_p(\tilde r)\in [\PP_p(\elm)]^3$ such that
$\bzeta_p(\tilde r)\scp\tn_\elm|_\sd=r_\sd$ for all $\sd\in\FKN$ and
$\|\bzeta_p(\tilde r)\|_\elm \le C_{\mathrm{DGS}}\|\tilde r\|_{H^{-\ft}(\pt
\elm)}$. Combining the above bounds and reasoning as above, we infer
that~\eqref{eq_Hdv_ext_simpl} holds true in case (2) altogether.
\\
(3) Proof of~\eqref{eq_Hdv_ext_simpl} when $\elm=\wK$. Since $\elm$
is the unit tetrahedron, we can use the bounds established in Steps~(1)
and~(2). Let $\bxi_{p,\elm}'\in \RTN_p(\elm)$ be the discrete $\argmin$ with
only divergence prescribed by $r_\elm$ but no boundary flux prescribed. Using
the result of Step~(1), we infer that
\[
    \norm{\bxi_{p,\elm}'}_\elm =
    \min_{\substack{\tv_p \in \RTN_p(\elm)\\
    \Dv \tv_p = r_\elm}}
    \norm{\tv_p}_\elm \leq  C
    \min_{\substack{\tv \in \Hdvi{\elm}\\
    \Dv \tv = r_\elm}}
    \norm{\tv}_\elm \leq C
    \min_{\substack{\tv \in \Hdvi{\elm}\\
    \tv \scp \tn_\elm|_\sd = r_\sd \,\, \forall \sd \in \FKN\\
    \Dv \tv = r_\elm}}
    \norm{\tv}_\elm = C\norm{\bxi_\elm}_\elm,
\]
where the last inequality follows by restricting the minimization set and
introducing the unique minimizer $\bxi_\elm$ defined in
Remark~\ref{rem_eq_Hdv}. Let now $\bxi_{p,\elm}''\in \RTN_p(\elm)$ be the
discrete $\argmin$ with divergence prescribed to zero and boundary flux
prescribed to $r_\sd - \bxi_{p,\elm}' \scp \tn_\elm|_\sd$ for all
$\sd\in\FKN$. In the case where $\FKN=\FK$, the compatibility condition on
the prescribed fluxes holds true since
\[
\sum_{\sd\in\FKN} (r_\sd - \bxi_{p,\elm}' \scp \tn_\elm|_\sd,1)_\sd
= \sum_{\sd\in\FKN} (r_\sd,1)_\sd - (\Dv\bxi_{p,\elm}',1)_\elm
= \sum_{\sd\in\FKN} (r_\sd,1)_\sd - (r_\elm,1)_\elm = 0
\]
by assumption. Step~(2) implies that
\[
    \norm{\bxi_{p,\elm}''}_\elm = \min_{\substack{\tv_p \in \RTN_p(\elm)\\
    \tv_p \scp \tn_\elm|_\sd = r_\sd - \bxi_{p,\elm}' \scp \tn_\elm|_\sd\,\, \forall \sd \in \FKN\\
    \Dv \tv_p = 0}}
    \norm{\tv_p}_\elm \leq C
    \min_{\substack{\tv \in \Hdvi{\elm}\\
    \tv \scp \tn_\elm|_\sd = r_\sd - \bxi_{p,\elm}' \scp \tn_\elm|_\sd\,\, \forall \sd \in \FKN\\
    \Dv \tv = 0}}
    \norm{\tv}_\elm.
\]
Furthermore, a shift by $\bxi_{p,\elm}'$ allows to rewrite equivalently, and
then bound by the triangle inequality, the last minimum above as follows:
\[
    \min_{\substack{\tv \in \Hdvi{\elm}\\
    \tv \scp \tn_\elm|_\sd = r_\sd \,\, \forall \sd \in \FKN\\
    \Dv \tv = r_\elm}}
    \norm{\tv - \bxi_{p,\elm}'}_\elm \leq
    \min_{\substack{\tv \in \Hdvi{\elm}\\
    \tv \scp \tn_\elm|_\sd = r_\sd \,\, \forall \sd \in \FKN\\
    \Dv \tv = r_\elm}}
    \norm{\tv}_\elm + \norm{\bxi_{p,\elm}'}_\elm = \norm{\bxi_\elm}_\elm + \norm{\bxi_{p,\elm}'}_\elm,
\]
so that $\norm{\bxi_{p,\elm}''}_\elm\le C(\norm{\bxi_\elm}_\elm +
\norm{\bxi_{p,\elm}'}_\elm)$. Now $\bxi_{p,\elm}' + \bxi_{p,\elm}''$ belongs
to the discrete minimization set in~\eqref{eq_Hdv_ext_simpl} and
$\norm{\bxi_{p,\elm}' + \bxi_{p,\elm}''}_\elm$ is bounded by
$\norm{\bxi_\elm}_\elm$. This proves~\eqref{eq_Hdv_ext_simpl} on the unit
tetrahedron.
\\
(4) Proof of~\eqref{eq_Hdv_ext_simpl} on a general tetrahedron $\elm$. We are
given a subset $\FKN$ of the faces of $\elm$, a $p$-degree piecewise
polynomial $r$ on $\FKN$, and a $p$-degree polynomial $r_\elm$ in $\elm$ such
that, if $\FKN=\FK$, $\sum_{\sd \in \FK}(r_\sd,1)_\sd = (r_\elm,1)_\elm$. We
are going to prove~\eqref{eq_Hdv_ext_simpl} on $\elm$ by mapping the
minimization problems to the unit tetrahedron $\wK$. Consider an affine
bijective map $\tT : \wK\to \elm$ with Jacobian matrix $\Jac$. Note that
$\Jac$ is a constant (and invertible) matrix in $\wK$ since $\tT$ is an
affine (bijective) map.  Let $\FwKN$ collect the faces of
$\wK$ that are images by $\tT^{-1}$ of the faces of $\elm$ in $\FKN$.  Let us
set
\[
\wr_{\wF} \eq \det(\Jac) \|\Jac^{-\mathrm{T}}\tn_{\wF}\|_{\ell^2} (r_\sd\circ \tT|_{\wF}),\,\,\forall \wF\in\FwKN, \qquad
\wr_{\wK} \eq \det(\Jac) (r_\elm \circ \tT),
\]
where $\tn_{\wF}$ is the unit outward normal to $\wK$ on the face $\wF$.
Then, $\wr$ defined by its restrictions to the faces $\wF\in\FwKN$ is a
$p$-degree piecewise polynomial on $\FwKN$, and $\wr_{\wK}$ is a $p$-degree
polynomial on $\wK$, and in the case where $\FwKN=\FwK$, we additionally have
\ban
\sum_{\wF\in\FwK} (\wr_{\wF},1)_{\wF} &= \sum_{\wF\in\FwK} (r_\sd\circ \tT|_{\wF},\det(\Jac) \|\Jac^{-\mathrm{T}}\tn_{\wF}\|_{\ell^2})_{\wF} \\
&= \sum_{\sd\in\FK} \epsilon_\Jac (r_\sd,1)_\sd = \epsilon_\Jac
(r_\elm,1)_\elm = (r_\elm \circ \tT,\det(\Jac))_{\wK} =
(\wr_{\wK},1)_{\wK}, \ean
with $\epsilon_\Jac = \frac{\det(\Jac)}{|\det(\Jac)|}=\pm1$. Here, we used the following classical formulas to change the surface measure and the
volume measure: $\mathrm{d}s = |\det(\Jac)|
\|\Jac^{-\mathrm{T}}\tn_{\wF}\|_{\ell^2} \mathrm{d}\widehat{s}$ and
$\mathrm{d}\tx = |\det(\Jac)| \mathrm{d}\widehat{\tx}$. Let us now consider
the contravariant Piola transformation such that $\piola(\tv) = \tA(\tv\circ
\tT)$ with $\tA = \det(\Jac)\Jac^{-1}$. Then $\piola$ is an isomorphism from
$\Hdvi{\elm}$ to $\Hdvi{\wK}$, and also from $\RTN_p(\elm)$ to $\RTN_p(\wK)$.
Moreover, we have the following key properties:
\bse\ba
&\Dv \tv = r_\elm \,\,\text{in $\elm$} \Longleftrightarrow \Dv(\piola(\tv)) = \wr_{\wK}\,\,\text{in $\wK$}, \\
&\tv\scp\tn_{\elm}|_\sd = r_\sd\,\,\text{on all $\sd\in\FKN$} \Longleftrightarrow
\piola(\tv)\scp \tn_{\wF} = \wr_{\wF}\,\,\text{on all $\wF\in\FwKN$},
\ea\ese
with $\elm=\tT(\wK)$ and $\sd=\tT(\wF)$. The first equivalence results from
$\Dv (\piola(\tv)) = \det(\Jac) (\Dv\tv)\circ \tT$ and the definition of $\wr_{\wK}$.
To prove the second equivalence, recalling~\eqref{eq:def_vn_K}, the left-hand side means that
\[
(\tv,\Gr\phi)_\elm + (\Dv\tv,\phi)_\elm = \sum_{\sd\in\FKN} (r_\sd,\phi)_\sd,
\]
for all $\phi\in H^1(K)$ such that $\phi|_\sd=0$ for all $\sd\in\FK\setminus\FKN$.
Changing variables in the volume and surface integrals, the above identity amounts to
\[
(\piola(\tv),\Gr\wphi)_{\wK} + (\Dv(\piola(\tv)),\wphi)_{\wK} = \sum_{\wF\in\FwKN} (\wr_{\wF},\wphi)_{\wF},
\]
where $\wphi = \phi \circ \tT$. Since the pullback by the geometric map $\tT$ is an isomorphism from $\{\phi\in H^1(K);\, \phi|_\sd=0\, \forall \sd\in\FK\setminus\FKN\}$ to $\{\wphi\in H^1(\wK);\, \wphi|_{\wF}=0\, \forall \wF\in\FwK\setminus\FwKN\}$, the above identity means that $\piola(\tv)\scp \tn_{\wF} = \wr_{\wF}$ on all $\wF\in\FwKN$.

Let now $\bxi_{p,\elm}$ and $\bxi_\elm$
be the unique minimizers in~\eqref{eq_Hdv_ext_simpl} using
the polynomial data $r$ and $r_\elm$; similarly, let
$\wbxi_{p,\wK}$ and $\wbxi_{\wK}$ be the unique minimizers for the
minimization problems posed on $\wK$ using the polynomial data $\wr$ and $\wr_{\wK}$. We
infer from Step~(3) that $\|\wbxi_{p,\wK}\|_{\wK} \le \wC \|\wbxi_{\wK}\|_{\wK}$
with constant $\wC$ of order unity.
Since $\piola^{-1}(\wbxi_{p,\wK})$ is in the minimization set defining
$\bxi_{p,\elm}$ and since $\piola(\bxi_\elm)$ is in that of $\wbxi_{\wK}$, we have
\ban
\|\bxi_{p,\elm}\|_\elm & \le \|\piola^{-1}(\wbxi_{p,\wK})\|_\elm
\le \|\piola^{-1}\|_{\calL(L^2,L^2)}\|\wbxi_{p,\wK}\|_{\wK}
\le \|\piola^{-1}\|_{\calL(L^2,L^2)} \wC\|\wbxi_{\wK}\|_{\wK} \\
&\le \|\piola^{-1}\|_{\calL(L^2,L^2)} \wC \|\piola(\bxi_\elm)\|_{\wK}
\le \|\piola\|_{\calL(L^2,L^2)}\|\piola^{-1}\|_{\calL(L^2,L^2)} \wC \|\bxi_\elm\|_\elm,
\ean
where $\|\piola\|_{\calL(L^2,L^2)}$ and $\|\piola^{-1}\|_{\calL(L^2,L^2)}$
are the operator norms of $\piola$ and $\piola^{-1}$ as linear maps between
$\tL^2(\elm)$ and $\tL^2(\wK)$.
This completes the proof since the factor
$\|\piola\|_{\calL(L^2,L^2)}\|\piola^{-1}\|_{\calL(L^2,L^2)}$ is bounded by a
constant only depending on the shape-regularity parameter $\gamma_\elm$. \ep

\section{On cell enumeration and vertex coloring in patches}
\label{app_graphs}

We collect in this section some auxiliary results on cell enumeration and vertex coloring in simplicial patches, corresponding to the setting of an ``interior vertex'' as described in Section~\ref{sec_set}. For any cell $\elm\in\Ta$, its interior faces are collected in the set $\FKin \eq \FK\cap \Fain$. Let us first observe that any enumeration of the elements in the patch $\Ta$ in the form
$\elm_1,\ldots,\elm_{|\Ta|}$ induces a partition of each of the sets $\FKiin$, $1\le i\le |\Ta|$, into two disjoint subsets,
$\FKiin = \F_i^\sharp \cup \F_i^\flat$, where $\F_i^\sharp$ collects all the interior
faces of $\elm_i$ shared by an already enumerated cell $\elm_j$ with
$j<i$, \ie, $\F_i^\sharp \eq \{\sd\in\Fain,\, \sd=\partial \elm_i \cap \partial
\elm_j,\,j<i\}$, and $\F_i^\flat$ collects all the other interior faces of $\elm_i$, \ie, $\F_i^\flat \eq \{\sd\in\Fain,\, \sd=\partial \elm_i \cap \partial
\elm_j,\,j>i\}$. Note that $|\F_i^\flat|+|\F_i^\sharp|=3$. An immediate consequence of this
definition is that $\F_1^\flat=\F_{\elm_1}^{\mathrm{i}}$ and
$\F_1^\sharp=\emptyset$ on the first element, whereas
$\F_{|\Ta|}^\flat=\emptyset$ and
$\F_{|\Ta|}^\sharp=\F_{\elm_{|\Ta|}}^{\mathrm{i}}$ on the last element.

\begin{lemma}[Patch enumeration] \label{lem:int_patch_enum} Let $\Ta$ be an interior patch of tetrahedra as specified in Section~\ref{sec_set}. Then there exists an enumeration of the elements in the patch $\Ta$ so that:
\begin{enumerate}[(i)]

\item \label{en_en_edg} For all $1<i\le |\Ta|$, if there are at least two
faces in $\F_i^\sharp$, intersecting in an edge, then all the elements
sharing this edge come sooner in the enumeration, \ie, if
$|\F_i^\sharp|\ge2$ with $\sd^1,\sd^2\in \F_i^\sharp$, then
letting $\edg \eq \sd^1\cap \sd^2$, $\elm_j \in \T_\edg\setminus\{\elm_i\}$
implies that $j< i$.

\item \label{en_en_ver} For all $1<i<|\Ta|$, there are one or two neighbors of $\elm_i$ which have been already enumerated and correspondingly two or one neighbors of $\elm_i$ which have not been enumerated yet, \ie, $|\F_i^\sharp|\in\{1,2\}$ (so that $|\F_i^\flat|=3-|\F_i^\sharp|\in\{1,2\}$ as well), for all but the first and the last element. In particular, $\F_i^\sharp$ contains all the interior faces of $\elm_i$ (so that $\F_i^\flat$ is empty) if and only if $i=|\Ta|$.
\end{enumerate}
\end{lemma}

\begin{proof}
The key notion to assert the existence of the enumeration with the
requested properties is the shelling of a polytopal complex. Let us first
explain the concepts of polytopal complex and of shelling in the present context; we refer the reader to \cite[Definition~8.1]{Zieg_poly_95} for a more abstract presentation. The collection of the boundary faces, edges, and vertices of the patch constitutes a so-called pure, two-dimensional, polytopal complex, that is, a finite, nonempty collection of simplices (triangles, segments, and points, all lying on the boundary $\partial\oma$) that contains all the faces of its simplices (that is, the lower-dimensional simplices composing the boundary of each simplex) and such that the intersection of two distinct simplices in the complex is either empty or an edge or a vertex for each of them. The shellability of the polytopal complex (composed of the boundary faces, edges, and vertices) means that there exists an enumeration of the boundary faces (or, equivalently, the cells composing the patch) so that, for all $1\le i< |\Ta|$, the boundary of the set $\left(\cup_{j\le i} \elm_j\right) \cap \pt\oma$ is connected and contains only vertices of degree two (\ie, each vertex is connected by an edge to exactly two other vertices also belonging to this boundary).
The fact that this polytopal complex is shellable results from a theorem by Burggesser and Mani \cite{Burg_Mani_71}, see also \cite[Theorem~8.12]{Zieg_poly_95}. The main idea for the construction of the shelling is to imagine the complex as a ``little'' polyhedral planet, and launch a rocket from the interior of one of its faces; the rocket trajectory is a line that is supposed to intersect all the planes supporting the faces at one and only one point. The faces are enumerated by counting the launching face first, followed by the faces as they become visible from the rocket as it progresses to infinity. Once the rocket reaches infinity, it starts its travel back from minus infinity towards the complex; then all the previously hidden faces become visible and are enumerated as they disappear from the horizon. When the rocket reaches back the complex, all the faces have been enumerated, and the resulting enumeration produces a shelling with the desirable properties. We refer the reader to~\cite[pp.~240--243]{Zieg_poly_95} for details and illustrations.
\end{proof}

\begin{lemma}[Two-color refinement around edges]\label{lem:2color}
Fix a cell $\elm_*\in\Ta$ and an edge $\edg$ of $\elm_*$ having $\ver$ as one endpoint. Recall that $\T_\edg$ collects all the cells in $\Ta$ having $\edg$ as edge. Then there exists a conforming refinement $\T_\edg'$ of $\T_\edg$ composed of tetrahedra such
that
\begin{enumerate}[(i)]
\item $\T_\edg'$ contains $\elm_*$;
\item All the tetrahedra in $\T_\edg'$ have $\edg$ as edge, and their two other
vertices lie on $\pt\oma$;
\item The shape regularity parameter $\gamma_{\T_\edg'}$ is at most twice $\gamma_{\T_\edg}$;
\item Collecting all the vertices of $\T_\edg'$ that are not endpoints of $\edg$ in the set $\V_\edg'$, there is a two-color map $\mathtt{col}: \V_\edg' \to \{1,2\}$ so that for all $\kappa\in \T_\edg'$, the two vertices of $\kappa$ that are not endpoints of $\edg$, say $\{\ver_\kappa^n\}_{1\leq n\leq 2}$, satisfy
$\mathtt{col}(\ver_\kappa^n)=n$.
\end{enumerate}
\end{lemma}

\bp If $|\T_\edg|$ is even, we can just take $\T_\edg'=\T_\edg$ since the vertices of
$\T_\edg$ that are not endpoints of $\edg$ then form a cycle with an even number of
vertices that can be colored using alternating colors. If $|\T_\edg|$ is odd, we
pick one tetrahedron in $\T_\edg\setminus\{\elm_*\}$ and subdivide it into two
sub-tetrahedra by cutting it along the median plane containing $\edg$. By doing
so, we obtain a conforming simplicial refinement $\T_\edg'$ of $\T_\edg$ that has
all the desired properties. \ep

\begin{lemma}[Three-color patch refinement]\label{lem:3color}
Fix a cell $\elm_*\in\Ta$. There exists a conforming refinement $\T_\ver'$ of
$\Ta$ composed of tetrahedra such that
\begin{enumerate}[(i)]
\item $\T_\ver'$ contains $\elm_*$;
\item All the tetrahedra in $\T_\ver'$ have $\ver$ as vertex, and their
three other vertices lie on $\pt\oma$;
\item The shape regularity parameter $\gamma_{\T_\ver'}$ is at most a fixed multiple of $\gamma_{\T_\ver}$;
\item Collecting all the vertices of $\T_\ver'$ distinct from $\ver$ in the
 set $\V_\ver'$, there is a three-color map $\mathtt{col} : \V_\ver' \to
 \{1,2,3\}$ so that for all $\kappa\in \T_\ver'$, the three vertices of
 $\kappa$ distinct from $\ver$, say $\{\ver_\kappa^n\}_{1\leq n\leq 3}$,
 satisfy $\mathtt{col}(\ver_\kappa^n)=n$.
\end{enumerate}
\end{lemma}

\bp Since all the cells in $\Ta$ and in $\T_\ver'$ have $\ver$ as vertex and
their three other vertices lie on $\pt\oma$, we will reason on the trace of
$\Ta$ on $\pt\oma$. Using a homeomorphism, we can map $\cup_{\elm'\in
\Ta\setminus\{\elm_*\}} \{\elm'\cap \pt\oma\}$ to an interior triangulation,
say $\mathfrak{T}$, of the unit triangle $T$ in $\RR^2$ with the
particularity that the three sides of $T$ are edges of cells in
$\mathfrak{T}$ (these three triangular cells are the images by the above
homeomorphism of the trace on $\pt\oma$ of the three tetrahedra sharing a
face with $\elm_*$); see Figure~\ref{fig:patch}.
\begin{figure}[htb]
\begin{center}
\includegraphics[width=0.7\textwidth]{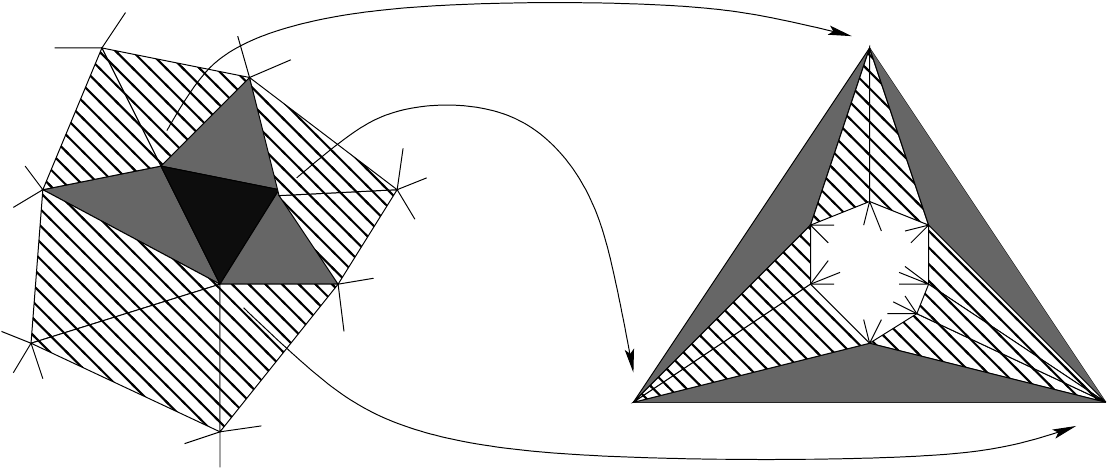}
\end{center}
\caption{Left: original patch $\Ta\cap\pt\oma$ locally around $\elm_*\cap\pt\oma$
(highlighted in dark grey), the three triangles $\elm'\cap\pt\oma$ for which the
tetrahedron $\elm'$ shares an interior face with $\elm$ are highlighted in light
grey, and the triangles $\elm'\cap\pt\oma$ for which the tetrahedron $\elm'$
shares only an interior edge with $\elm$ are dashed; Right: mapping by the
homeomorphism to a triangulation of the unit triangle $T$ in $\RR^2$;
the polygon at the heart of $T$ is the image by the homeomorphism of
all the triangles $\elm'\cap\pt\oma$ where $\elm'$ only shares the vertex $\ver$ with $\elm_*$.}
\label{fig:patch}
\end{figure}
We now devise a conforming triangular refinement of $\mathfrak{T}$ that
does not refine the three sides of $T$ and such that all the vertices in this
refinement are connected to an even number of other vertices (the number of
connections is called the degree of the vertex). The existence of a
three-coloring map on the vertices of this refinement will then follow
from~\cite{Tsai_West_11}. To this purpose, we proceed in several steps. Let
us call $\{\tz_1,\tz_2,\tz_3\}$ the three vertices of $T$. The three
triangles in $\mathfrak{T}$ supported on the three edges of $T$ are denoted by $\{\tau_1,\tau_2,\tau_3\}$ in such a way that $\tau_n$ does not touch the
vertex $\tz_n$, for all $n=1,2,3$. Let $\tz'_n$ denote the barycenter of
$\tau_n$.
\\
(1) We subdivide all the triangles in $\mathfrak{T}$ by barycentric
subdivision into six sub-triangles. By doing so, we create new vertices,
namely the barycenter of each triangle in $\mathfrak{T}$ (with degree 6), and
the midpoint of each edge in $\mathfrak{T}$ (with degree 4). Moreover, all
the original vertices of $\mathfrak{T}$ have now even degree, except for
$\{\tz_1,\tz_2,\tz_3\}$ which have odd degree. To avoid refining the three
edges of $T$, we remove for all $n\in\{1,2,3\}$ the connection between the
barycenter $\tz_n'$ and the midpoint of the edge of $T$ supporting the
triangle $\tau_n$. By doing so, the degree of the three barycenters
$\{\tz_1',\tz_2',\tz_3'\}$ changes from six to five. At this stage, we have a
conforming, triangular refinement preserving the three sides of $T$, but
which contains six vertices of odd degree, namely $\{\tz_1,\tz_2,\tz_3\}$ and
$\{\tz_1',\tz_2',\tz_3'\}$.
\\
(2) We subdivide the triangle with vertices $\{\tz_1,\tz_2',\tz_3\}$ into
three triangles by joining its barycenter, say $\tz_2''$, to the three
vertices. The degree of $\tz_1$, $\tz_2'$, and $\tz_3$ is now even as
desired, but we have created the new vertex $\tz_2''$ with degree three. This
new triangulation is illustrated in the central panel of
Figure~\ref{fig:octo} in a slightly simplified setting with respect to
Figure~\ref{fig:patch} since we have reduced to one the number of original
dashed triangles at each of the three vertices of $T$ (see the left panel of
Figure~\ref{fig:octo} for the original triangulation).
\\
(3) We now subdivide all the triangles having $\tz_1$ as vertex, except the
newly created one on the boundary of $T$, into two sub-triangles as depicted in the right panel of
Figure~\ref{fig:octo}. The vertices $\tz_2''$ and $\tz_2$ now have even
degree, and we have created additional vertices at some edge midpoints that
all have degree four while we have also increased by two the degree of some
vertices.
\\
(4) Finally, we use a similar process, as depicted also in the right panel of
Figure~\ref{fig:octo}, so that the vertices $\tz_1'$ and $\tz_3'$ now have
even degree, while we create additional vertices at some edge midpoints that
all have degree four. We now have a triangulation where all the vertices have
even degree. \ep
\begin{figure}[htb]
\begin{center}
\scalebox{0.64}{\input{octo.pspdftex}}
\end{center}
\caption{Left: original triangulation in the simple case where there is only one dashed triangle at each of the three vertices $\{\tz_1,\tz_2,\tz_3\}$; Center: refined triangulation at the end of Step (2) showing the barycenters  $\{\tz_1',\tz_2',\tz_3'\}$ and the newly created one $\tz_2''$; Right: final refined triangulation where now all the vertices have even degree.}
\label{fig:octo}
\end{figure}

\bibliographystyle{siam}
\bibliography{biblio}

\end{document}